\documentclass[a4paper,english,oneside]{article}
\usepackage[T1]{fontenc}
\usepackage[utf8]{inputenc}
\usepackage{authblk}
\usepackage{indentfirst}
\usepackage{etoolbox}
\usepackage{geometry}
\usepackage{makecell}
\usepackage{booktabs}
\usepackage{amsmath}
\usepackage{amsfonts}
\usepackage{amssymb}
\usepackage{amsthm}
\usepackage[english]{babel}
\usepackage{color}
\usepackage{amsthm, amsfonts,mathrsfs}
\usepackage[T1]{fontenc}
\usepackage{times}
\usepackage{mathrsfs}
\usepackage{lmodern}
\usepackage{pstricks-add}
\usepackage{hyperref}
\hypersetup{hidelinks}
\usepackage{mathscinet}
\usepackage{enumitem}
\usepackage[cal=cm,bb=ams]{mathalfa}
\usepackage{microtype}
\usepackage{mathtools}
\usepackage{esint}

\newtheorem{theorem}{Theorem}[section]
\newtheorem{corollary}{Corollary}[section]

\newtheorem{lemma}{Lemma}[section]
\newtheorem{proposition}{Proposition}[section]
\newtheorem{remark}{Remark}[section]
\newtheorem{example}{Example}[section]
\numberwithin{equation}{section}

\begin{document}
\title
{Global existence for the semilinear regular Euler-Poisson-Darboux-Tricomi equation,  II}

\author[a]{Yuequn Li\thanks{
Email:
\href{mailto:022670205616@ujn.edu.cn}{022670205616@ujn.edu.cn}
}}

\author[b]{Fei Guo\thanks{
Corresponding author;
Email:
\href{mailto:guof@njnu.edu.cn}{guof@njnu.edu.cn}
}}

%\author[a]{Yuequn Li}
%\author[b]{Fei Guo\footnote{
%\href{mailto:guof@njnu.edu.cn}{guof@njnu.edu.cn}}}

\affil[a]{School of Mathematical Sciences, University of Jinan, Jinan 250022, China}
\affil[b]{School of Mathematical Sciences, Ministry of Education Key Laboratory of NSLSCS,
Nanjing Normal University,  Nanjing 210023, China}

\date{}

\maketitle

% --------------------------------------------------------------

\begin{abstract}

In our previous work \cite{LiGuo2025}, we established the global existence of solutions to the semilinear Euler-Poisson-Darboux-Tricomi equation with power nonlinearity \(|u|^p\) in the case where the damping term dominates the mass term. In contrast, this paper studies the reverse case, namely, the existence of global solutions to the equation when the mass term dominates the damping term.

We prove that small-data solutions exist globally if
\[
p>\max\{p_0((m+1)n,\mu), \widetilde{p}_1(n,m,\mu)\}
\]
for \(\delta\doteq(\mu-1)^2-4\nu^2\leq0\), where \(\delta\) measures the relative dominance between the damping and mass terms, \(p_0(n,\mu)\) is the \textit{Strauss-related} exponent defined by \eqref{p0ndef}, and \(\widetilde{p}_1(n,m,\mu)\) is defined by \eqref{eqp1}.
\end{abstract}

% -----------------------------------------

\begin{flushleft}
\end{flushleft}

\begin{flushleft}
\textbf{AMS Classification (2020)} 35B33,   35G25, 35L15, 35L71.
\end{flushleft}

\section{Introduction}\label{s1}
This paper is a continuation of our study \cite{LiGuo2025, LiGuo20252} on the regular semiliniear Euler-Poisson-Darboux-Tricomi (EPDT) equation  with a power nonlinearity
\begin{equation}\label{eq}
  \left\{
\begin{aligned}
&\partial_t^2u-t^{2m}\Delta u+\frac{\mu}{t}\partial_tu+\frac{\nu^2}{t^2}u=|u|^p,&\quad& t>1,\ x\in\mathbb{R}^n, \\
&u(1,x)= u_0(x), \ \partial_tu(1,x)= u_1(x),               &\quad& x\in\mathbb{R}^n,
\end{aligned}
 \right.
  \end{equation}
where $m>-1$ and $n\geq1$ denotes the  spatial dimension. In \cite{LiGuo2025}, the global existence of solutions to (\ref{eq}) was investigated in the case where the damping term $\frac{\mu}{t}\partial_tu$ dominates over the mass term $\frac{\nu^2}{t^2}u$. In this paper, attention is further devoted to the case where the mass term prevails over the damping term.

Let us begin with the decay estimates for solutions to the following three different types of linear wave equations. For
\begin{align}
&\makebox[5.2cm][l]{\text{Free wave equation (W):}}
\partial_t^2u-\Delta u=0,\label{freewave}\\
&\makebox[5.2cm][l]{\text{Damped wave equation (DW):}}
\partial_t^2u-\Delta u+\partial_tu=0,\notag\\
&\makebox[5.2cm][l]{\text{Klein-Gordon equation (KG):}}
\partial_t^2u-\Delta u+u=0\notag
\end{align}
with the same  initial data $(u,\partial_tu)(0,x)=(u_0,u_1)$,
the corresponding solutions decay at different rates:
\begin{align}
&\text{(W):}&\ &  \Vert \big(\partial_tu(t,\cdot),\nabla u(t,\cdot)\big)\Vert_{L^q}\leq C(1+t)^{-\frac{n-1}{2}(\frac{1}{p}-\frac{1}{q})}\Vert(u_1,\nabla u_0)\Vert_{W^{p,N_p}}, &\ & \text{\cite{Str70a,Str70b}},\label{decayrate1}\\
&\text{(DW):}&\ & \Vert\big(\partial_tu(t,\cdot),\nabla u(t,\cdot)\big)\Vert_{L^q}\leq C(1+t)^{-\frac{n}{2}(\frac{1}{p}-\frac{1}{q})-\frac{1}{2}}\Vert(u_1,\nabla u_0)\Vert_{W^{p,N_p}},&\ & \text{\cite{Mat1977}},\notag\\
&\text{(KG):} &\ & \Vert\big( \partial_tu(t,\cdot),\nabla u(t,\cdot)\big)\Vert_{L^q}\leq C(1+t)^{-\frac{n}{2}(\frac{1}{p}-\frac{1}{q})}\Vert(u_1,\nabla u_0)\Vert_{W^{p,N_p}},&\ & \text{\cite{Rac92}},\notag
\end{align}
where $1<p\leq2,\frac{1}{p}+\frac{1}{q}=1$ and $N_p>n(\frac{1}{p}-\frac{1}{q})$.  The distinct long-time behaviors of the solutions to these three types of equations suggest that the corresponding critical exponents for the global existence and blow-up of the semilinear equations with power-type nonlinearity are different. We now turn to the analysis of the corresponding results for the three models with time-dependent coefficients.

The critical  exponent of  classical semilinear wave equation $\partial_t^2u-\Delta u=\vert u\vert^p$ is clear, that is Strauss exponent  $p_S(n)$, which is the
positive root of the quadratic equation
\begin{equation}\label{strauss}
(n-1)p^2-(n+1)p-2=0\quad \text{for}\ n\geq2,
\end{equation}
and $p_S(1)=+\infty$,  one can refer to  \cite{GeLiSo1997, Glassey1981, John1979,  Zh2006, Zhou2007, Zhou2014}.   In other words, the global (in time) solution to (\ref{strauss}) exists if $p>p_S(n)$, otherwise, the blow-up phenomenon occurs if $1<p\leq p_S(n)$.
It is also clear that the Fujita exponent  $p_F(n)\doteq1+\frac{2}{n}$ is the critical exponent for the semilinear damped wave equation $\partial_t^2u-\Delta u+\partial_tu=\vert u\vert^p$, see  \cite{ ToYo2001, Zh2001}.  We point out that  $p_F(n)$  is the critical exponent for the semilinear heat equation $\partial_tu-\Delta u=\vert u\vert^p$, \cite{Fujita1969}.
Accordingly, we commonly refer to a semilinear equation as \textit{wave-like} if its critical exponent is related to the Strauss exponent, and as \textit{heat-like} or \textit{parabolic-like} if its critical exponent is related to the Fujita exponent.

Wirth \cite{Wirth2006, Wirth2007} studied systematically the decay behavior of solutions to the linear equation
\begin{equation}\label{eqdampinglinear}
\partial_t^2u-\Delta u+\frac{\mu}{t^\alpha}\partial_tu=0, \ \ \mu>0,\ \ t\geq1.
\end{equation}
When $\alpha>1$, the  decay rate of solutions to (\ref{eqdampinglinear}) is  close to that of the free wave equation (\ref{freewave}), and hence  the damping term  in this case  is referred to as \textit{noneffective}.
For $-1\leq\alpha<1$, the decay rate of solutions is comparable to that of the heat equation, which indicates that $\partial_tu$  dominates over $\partial_t^2u$ in (\ref{eqdampinglinear}), the damping term is considered as  \textit{effective}. When $\alpha<-1$, the solution
exhibits no decay,  and the damping is classified as \textit{over-damping}.  If $\alpha=1$,  (\ref{eqdampinglinear}) is \textit{scale-invariant} under the transformation
\begin{equation}\label{scaletrans1}
\widetilde{u}(t,x)=u(\lambda t,\lambda x),
\end{equation} and the authors in  \cite{Wirth2004,Yagdjian2021} established the  explicit representation formula and the integral representation formula for the solution.
In addition, the most distinctive feature in this case is that the behavior of $u$  is closely related to the magnitude of $\mu$, but unfortunately, determining the precise threshold of $\mu$ remains an open problem \cite{LaiTaWa2017}.

Since the classification of  the linear equation (\ref{eqdampinglinear}) is well understood, it provides important guidance for the study of the critical exponent of
\begin{equation}\label{eqdampingsemilinear}
\partial_t^2u-\Delta u+\frac{\mu}{t^\alpha}\partial_tu=\vert u\vert^p, \ \ \mu>0,\ \ t\geq1.
\end{equation}
When $\alpha\in(-1,1)$,     the authors in \cite{DaLuRe2013,LinNiZhai2012} proved that the critical exponent for (\ref{eqdampingsemilinear}) is $p_F(n)$, so the damping is also \textit{effective}.  In the \textit{scattering } case  $\alpha>1$, Lai-Takamura \cite{LaiTa2018} proved when $1<p<p_S(n),$  the blow-up phenomenon occurs, thereafter,  Wakasa-Yordanov \cite{ WaYo2019} proved the critical case $p=p_S(n)$ still belongs to the blow-up regime;  about the global existence, Liu-Wang  \cite{LiuWang2018} established the global existence in the supercritical case $p>p_S(n)$ for $n=3,4$.  Based on the above results, it can be roughly stated that in the case $\alpha>1$, the damping term should be regarded as \textit{non-effective}. %because in this regime the critical exponent of (\ref{eqdampingterm}) is identical to the semilinear wave equation,
%which indicates that the presence of $\frac{\mu}{t^\alpha}\partial_tu$ has no effect on the  critical exponent of  the classical semilinear wave equation.
Ikeda-Wakasugi \cite{IkWa2020} proved in the \textit{overdamping} case $\alpha<-1$,  the solution to (\ref{eqdampingsemilinear}) exists globally for any $p>1$,  which means that the  critical exponent disappears, and this is a rather interesting phenomenon.

As for the scale-invariant case $\alpha=1$, (\ref{eqdampingsemilinear}) is known as  the semilinear Euler-Poisson-Darboux (EPD)  eqation. The result obtained in \cite{DaLuRe2015} naturally leads  a  conjecture that the critical exponent is
\begin{equation}\label{conjectureepd}
p_{crit}(n,\mu)\doteq\max\{p_S(n+\mu),p_F(n)\}=\left\{\begin{aligned}
&p_S(n+\mu),&\quad&0<\mu\leq\bar{\mu}(n),\\
&p_F(n),&\quad&\mu\geq\bar{\mu}(n),
\end{aligned}
\right.
\end{equation}
where $\bar{\mu}(n)=\frac{n^2+n+2}{n+2}$.   The blow-up result for subcritical and critical case $1<p\leq p_{crit}(n,\mu)$ has solved in \cite{IkSo2018,LaiTaWa2017,PaRe2019,TuLin2019,Wa2014}. The proof of global existence in the supercritical case $p>p_{crit}(n,\mu)$ presents formidable challenges. M. D$^\prime$Abbicco \cite{Da2021} used the asymptotic properties of Bessel functions together with the Fourier multiplier theorem to establish the  decay estimates for the solution to  $1-$dimensional linear EPD equation, and subsequently applied the contraction mapping principle to derive the global solution for $p>p_{crit}(1,\mu)$.  In the serious papers \cite{HeLiYin2025,LiWangYin2025,LiYin2025},  He-Li-Wang-Yin employed the explicit integral representation of solutions to the $2$-dimensional linear Tricomi equation together with the vector field method to derive weighted Strichartz estimates for the linear EPD equation,  and then the global existence of solutions was  established via the classical iteration argument when $p>p_{crit}(2,\mu)$. For $n\geq3$, the global existence of solutions to the EPD equation remains, to the best of our knowledge, an open problem.

Based on the above analysis, the relevant results for the time-dependent damped wave equation (\ref{eqdampingsemilinear}) are summarized  in Table \ref{resultsdampedwave1}.

\begin{table}[h]\label{resultsdampedwave1}
\centering
\setlength{\tabcolsep}{0.5mm}{
  \setlength{\abovecaptionskip}{0pt} % 设置标题与表格之间的垂直间距
  \begin{tabular}{c|c|c}
    \toprule
    The range of $\alpha$ & Effect of damping & Critical exponent  \\
    \hline
    $\alpha>1$ & Non-effective (Scattering)&$p_S(n) $ \\
    \hline
    $-1<\alpha<1$ & Effective  & $p_F(n)$\\
    \hline
    $\alpha=1$ & Scale-invariant weak dissipation\quad\quad  &\quad $\max\{p_S(n+\mu),p_F(n)\}$ \\
    \hline
    $\alpha<-1$&Over-damping    & Vanishing\\
    \bottomrule
  \end{tabular}}
  \vspace{5pt} % 手动添加一些额外的垂直间距
    \caption{Effect of damping of (\ref{eqdampingsemilinear})}
\end{table}

Consider the linear Klein-Gordon equation
\begin{equation}\label{eqmass1}
\partial_t^2u-\Delta u+\frac{m^2}{t^2}u=0.
\end{equation}
By \cite{BoPHD2011,BoRe2012}, from the energy conservation perspective, when $m^2\geq\frac{1}{4}$, the energy functional is independent of $m$, whereas for $m^2<\frac{1}{4}$, it depends closely on $m$. This indicates that the value of the parameter $m$ has a non-negligible effect on the behavior of solutions to (\ref{eqmass1}).
Combining the  blow-up result  in  \cite{KeTa1999} and the global result in \cite{LiSo1996}, the critical exponent of the classical semilinear Klein-Gordon equation  $\partial_t^2u-\Delta u+u=\vert u\vert^p$ is $p_F(n)$
only for $n=1,2,3$.   As they pointed out, this conclusion fails when $n\geq4$.   As far as we know, for semilinear Klein–Gordon equations with time-dependent coefficients
$\partial_t^2u-\Delta u+m(t)^2u=\vert u\vert^p$, there has been little progress regarding either global existence or blow-up.

The corresponding linear equation for
\begin{equation}\label{waveeq}
\partial_t^2u-\Delta u+\frac{\mu}{t}\partial_tu+\frac{\nu^2}{t^2}u=\vert u\vert^p,\ \mu>0,\ \nu^2\geq0,\ t\geq1
\end{equation}
is still invariant under the   transformation  (\ref{scaletrans1}).  By the analysis about  the case of $\alpha=1$ in  (\ref{eqdampinglinear}) and (\ref{eqmass1}),  the same scale for the damping term and the mass term  causes an interaction such that the  magnitudes of the coefficients $\mu$ and $\nu^2$  influence whether the damping term or the mass term plays the dominant role in (\ref{waveeq}). Here the same scale means both the linear EPD equation ($\alpha=1$ in  (\ref{eqdampinglinear})) and the linear  Klein-Gordon equation (\ref{eqmass1}) are invariant under the same transformation (\ref{scaletrans1}). And hence (\ref{waveeq}) should be regarded as the combination of EPD  equation and Klein-Gordon equation.  Palmieri systematically investigated the critical exponent  for  (\ref{waveeq}) and determined   the quantity that measures the interaction between the damping term and mass term, that is
\begin{equation}\label{deltadefine}
\delta\doteq(\mu-1)^2-4\nu^2.
\end{equation}
Broadly speaking, the larger the value of $\delta$  (that is, the larger $\mu$ compared to $\nu^2$), the more dominant the damping term  becomes over the mass term, and the critical exponent of (\ref{waveeq}) is closely related to the Fujita exponent. In this case, (\ref{waveeq}) exhibits a \textit{parabolic-like} behavior. Conversely, the smaller the value of $\delta$, the critical exponent becomes more closely related to the Strauss exponent or other indices, and (\ref{waveeq}) exhibits a \textit{wave-like} or \textit{Klein–Gordon-like} behavior. The  results for (\ref{waveeq})  are presented in Tables 2 and 3 in a more intuitive manner.
\newpage
\begin{table}[h]\label{blowuptable2}
\renewcommand\arraystretch{1.2}
  \centering
  \setlength{\abovecaptionskip}{0pt} % 设置标题与表格之间的垂直间距
  \resizebox{\textwidth}{15mm}{
  \begin{tabular}{c|c|c}
    \toprule
    The range of $\delta$ & The range of $p$ & References \\
    \hline
    $\delta\geq0$ & $1<p\leq p_F(n+\frac{\mu-1-\sqrt\delta}{2})$ & Thm.2.4 \cite{PaRe2017}, Thm.1.3 \cite{PaRe2018} \\
    \hline
    $0<\delta\leq1$ & \makecell{$1<p<\max\{p_F(n+\frac{\mu-1-\sqrt\delta}{2}),p_S(n+\mu)\},\ n\geq1$\\
                                    $p=p_F(n+\frac{\mu-1-\sqrt\delta}{2}),\ n\geq1$\\
                                    $p=p_S(n+\mu),\ n\geq2$ }   & Coro.1.4 \cite{PaRe2019}\\
    \hline
    $\delta\geq 0$ & $1<p\leq p_S(n+\mu)$  &Thm.1.2 \cite{Tu2019} \\
    \hline
    $0\leq\delta<n^2$ & $p= p_S(n+\mu)$  &Thm.1.3 \cite{Tu2019} \\
    \bottomrule
  \end{tabular}}
  \vspace{4pt} % 手动添加一些额外的垂直间距
    \caption{Blow-up results of (\ref{waveeq})}
\end{table}

%\vspace{-15pt} % 调整两个表格之间的垂直距离
\begin{table}[h]\label{globaltable3}
  \centering
  \begin{tabular}{c|c|c}
    \toprule
    The range of $\delta$ & The range of $p$ & References  \\
    \hline
    $\delta\leq0$ & \makecell{$2\leq p,  n=1,2$ \\ $2\leq p\leq3, n=3$\\ $p=2, n=4$} & Thm.2.1 \cite{PaRe2017} \\
    \hline
    $\delta\geq (n-1)^2$ & $p>p_F(n+\frac{\mu-1-\sqrt\delta}{2})$ & Thm.1.1, Thm.1.2 \cite{PaRe2018} \\
    \hline
    $\delta<0$ &$p>\max\{p_0(n,\mu),p_1(n,\mu)\}$          & Thm.1 \cite{DaPa2021}\\
    \bottomrule
  \end{tabular}
  \vspace{4pt} % 手动添加一些额外的垂直间距
  \caption{Global existence of (\ref{waveeq})}
\end{table}

%\vspace{3pt} % 调整表格与下一行文字之间的垂直空白

\noindent  Here $p_0(n,\mu)>1$ is  positive root of
\begin{equation}\label{p0ndef}
(n+\frac{\mu-1}{2})(p-1)^2+(n+\mu-3)(p-1)-4=0
\end{equation}
and $p_1(n,\mu)>1$  is  positive root of
\begin{equation}\label{p1ndef}
\frac{n+\mu-1}{2}(p-1)^2+(n+\mu-2)(p-1)-4=0.
\end{equation}
It follows from (\ref{strauss}) and (\ref{p0ndef}) that $p_0(n,1)=p_S(n+1)$,  indicating that
$p_0(n,\mu)$ can be regarded as a \textit{Strauss-related} exponent.

Palmieri  \cite{Pa2025} first named (\ref{eq}) the EPDT equation and established blow-up results for $1<p<\max\{p_F((m+1)n+\frac{\mu-1-\sqrt{\delta}}{2}), p_S(n+\frac{\mu}{m+1},m)\}$ as well as for  $p=p_F((m+1)n+\frac{\mu-1-\sqrt{\delta}}{2})$. Here,  $\delta$ is also defined by (\ref{deltadefine}), while $p_S(n,m)$ denotes the modified Strauss exponent, namely,   the positive root of
$$\big((m+1)n-1\big)p^2-\big((m+1)(n-2)+3\big)p-2(m+1)=0.$$
Indeed,  $p_S(n,m)$ is the critical exponent for the classical similinear Tricomi equation $\partial_t^2u-t^{2m}\Delta u=\vert u\vert^p$,  as shown  by He-Witt-Yin \cite{He2017II, He2017I, He2020}. Since $p_S(n,0)=p_S(n)$, this suggests that many properties of the Tricomi equation may be analogous to those of the wave equation.  The corresponding linear equation of (\ref{eq}) is invariant under the scaling
\begin{equation}\label{scaletrans2}
\widetilde{u}(t,x)=u(\lambda^{\frac{1}{m+1}}t,\lambda x),
\end{equation}
this implies that the damping term and the mass term are  of  different scale with respect to  time and spatial variables.  Therefore, investigating the corresponding critical exponent and the quantity characterizing the interaction between the damping term and mass term remains a highly meaningful problem.
In our recent paper \cite{LiGuo2025}, we give the precise answer that the interaction between the damping term and the mass term in (\ref{eq}) is still measured by $\delta$.  Moreover, it is shown that when $\delta$ is appropriately large, the critical exponent of (\ref{eq}) is  $p_F((m+1)n+\frac{\mu-1-\sqrt{\delta}}{2})$, indicating that the equation is \textit{parabolic-like} when the damping term dominates.  Additionally, in \cite{LiGuo20252}, we provided a blow-up result when $p=p_S(n+\frac{\mu}{m+1},m)$ for all $\delta>0$.

Accordingly, if the mass term in (\ref{eq})  dominates the damping term, it is natural to ask: \\
$\bullet$ Does a critical exponent exist for (\ref{eq}) ? \\
$\bullet$ Does the solution to (\ref{eq}) exhibit \textit{wave-like} behavior?

Motivated by these questions, this paper investigates the existence of global solutions to (\ref{eq}) in the regime of small $\delta$.   In \cite{LiGuo2025}, we obtained  the  \textit{damping-type} transformations under which    the corresponding linear equation of (\ref{eq})
is transformed into a wave equation containing only the damping term.  Using the Fourier transform, the linear equation is converted into a Bessel equation in the frequency domain, from which the $(L^1\cap L^2)-L^2$ estimate for the solution are subsequently derived.   The overall strategy remains applicable, but a different method is needed to treat the case of small $\delta$.  In this paper, we discover  a serious of transformations (\ref{sub11}), (\ref{sub21}) and (\ref{sub31}),  which can transfer the linear EPDT equation  (\ref{lineareq}) into a wave equation (\ref{masseq}) involving only the mass term.  This set of transformations is referred to as the \textit{mass-type} transformations.    Inspired by the approach in \cite{BoRe2012} for handling the linear Klein-Gordon equation, we apply the Fourier transform with respect to the spatial variable $x$, which reduces the problem to the confluent hypergeometric equation (\ref{solutionhuaw}) in the frequency domain. Consequently, the solution to (\ref{masseq})  can be reduced to analysing the Fourier integral operators (\ref{desired1})-(\ref{desired4}).  Since the confluent hypergeometric function exhibits different asymptotic behaviors for large and small values of the variable,  it is natural to apply  Littlewood-Paley theory to decompose the frequency domain into high, medium, and low regions, and to establish $L^q-L^{q'}$ estimates    for the solution to  (\ref{masseq}) in each region, where $\frac{1}{q}+\frac{1}{q'}=1$. Note  that the variable transformations (\ref{sub11}), (\ref{sub21}) and (\ref{sub31}) are all invertible, the $L^q-L^{q'}$ estimates for the solution to (\ref{eq}) then follow.  Based on the $L^q-L^{q'}$ estimates for the solution to  the linear EPDT equation, we can construct suitable function spaces, and by applying the contraction mapping principle together with Duhamel's principle to  establish the global existence of solutions to (\ref{waveeq}).

The organization of this paper is as follows.
In Section \ref{mainresults}, three main results and some remarks  are presented:   Theorems \ref{propuwithtau=1}, \ref{estimateu0=0}  concerning  the $L^q-L^{q'}$ estimates for  the solution to  linear EPDT equation (\ref{lineareq}), and   Theorem \ref{mainresult}  establishing the global existence of solutions to (\ref{eq}). Section \ref{linearepdtestimate} gives the proofs of Theorems \ref{propuwithtau=1}, \ref{estimateu0=0}.   Section \ref{sectionproofmainresult} provides the proof of Theorem \ref{mainresult}.  In Section \ref{finialremarks}, we discuss some open problems
related to (\ref{eq}). A brief proof of the local well-posedness for (\ref{eq}) is given in the Appendix.

\section{Main results}\label{mainresults}
For  the corresponding linear equation of (\ref{eq}), i.e.,
\begin{equation}\label{lineareq}
  \left\{
\begin{aligned}
&\partial_t^2u-t^{2m}\Delta u+\frac{\mu}{t}\partial_tu+\frac{\nu^2}{t^2}u=0,\ &\quad &t>\tau\geq1,\ x\in\mathbb{R}^n,\\
&u(\tau,x)= u_0(x), \ \partial_tu(\tau,x)= u_1(x),               &\quad &x\in\mathbb{R}^n,
\end{aligned}
 \right.
 \end{equation}
we establish the following estimates on the conjugate line.

\begin{theorem}\label{propuwithtau=1}
Let us consider $m>-1, \mu, \nu^2\geq0$  such that $\delta\doteq(\mu-1)^2-4\nu^2<\min\{(m+1)^2,1\}$.  Let $u$  be the solution to the linear equation (\ref{lineareq}) with the initial data $(u_0,u_1)\in H^{r+1}_q\times H^{r}_q$.   Then,  $u$ satisfies the estimates
\begin{align*}
&\Vert u(t,\cdot)\Vert_{L^{q^\prime}}\lesssim t^{-\frac{m}{2}}B_0(t)(1+\log t)^{\chi_\delta}\big(\Vert u_0\Vert_{H^{r}_q}+\Vert u_1\Vert_{H^{[r-1]_+}_q}\big),\\
&\Vert \nabla u(t,\cdot)\Vert_{L^{q^\prime}}\lesssim t^{-\frac{m}{2}}B_1(t)(1+\log t)^{\chi_\delta}\big(\Vert u_0\Vert_{H^{r+1}_q}+\Vert u_1\Vert_{H^{r}_q}\big),\\
&\Vert \partial_t u(t,\cdot)\Vert_{L^{q^\prime}}\lesssim t^{\frac{m}{2}}B_1(t)(1+\log t)^{\chi_\delta}\big(\Vert u_0\Vert_{H^{r+1}_q}+\Vert u_1\Vert_{H^{r}_q}\big),
\end{align*}
where  $1<q\leq2$, $\frac{1}{q}+\frac{1}{q^\prime}=1$, $r=n(\frac{1}{q}-\frac{1}{q^\prime})$,  $\chi_\delta=\left\{\begin{aligned}
&1,&\quad&\text{if}\ \ \delta=0,\\
&0,&\quad&\text{if}\ \ \delta\neq0,
\end{aligned}\right.$
and $B_1(t)=t^{-\frac{\mu}{2}-\frac{(m+1)(n-1)}{2}(\frac{1}{q}-\frac{1}{q^\prime})},$
\begin{equation*}
\begin{aligned}
B_0(t)&=\left\{\begin{aligned}
&t^{-\frac{\mu}{2}+\max\{\frac{m+1+\sqrt{\delta}}{2}-(m+1)n(\frac{1}{q}-\frac{1}{q^\prime}),
-\frac{(m+1)(n-1)}{2}(\frac{1}{q}-\frac{1}{q^\prime})\}},&\quad&\text{if}\  \ 0<\delta<\min\{(m+1)^2,1\},\\
&t^{-\frac{\mu}{2}+\max\{\frac{m+1}{2}-(m+1)n(\frac{1}{q}-\frac{1}{q^\prime})+\epsilon,
-\frac{(m+1)(n-1)}{2}(\frac{1}{q}-\frac{1}{q^\prime})\}},&\quad&\text{if}\ \ \delta=0,\\
&t^{-\frac{\mu}{2}+\max\{\frac{m+1}{2}-(m+1)n(\frac{1}{q}-\frac{1}{q^\prime}),
-\frac{(m+1)(n-1)}{2}(\frac{1}{q}-\frac{1}{q^\prime})\}},&\quad&\text{if}\ \ \delta<0
\end{aligned}
\right.
\end{aligned}
\end{equation*}
with $\epsilon>0$ sufficiently small.
\end{theorem}

\begin{theorem}\label{estimateu0=0}
Let us consider $m>-1, \mu, \nu^2\geq0$  such that $\delta\doteq(\mu-1)^2-4\nu^2<\min\{(m+1)^2,1\}$.  Let $u$  be the solution to the linear equation (\ref{lineareq}) with the initial data $u_0=0,u_1\in H^r_q$.
Then $u$ satisfies the estimates
\begin{align*}
&\Vert\nabla u(t,\cdot)\Vert_{L^{q^\prime}}\lesssim t^{-\frac{m}{2}}\tau^{-\frac{m}{2}}B_1(t)B^{-1}_1(\tau)(1+\log\frac{t}{\tau})^{\chi_\delta}
\big(\Vert u_1\Vert_{\dot{H}^{r}_q}+\tau^{-(m+1)r}\Vert u_1\Vert_{L^q}\big),\\
&\Vert\partial_tu(t,\cdot)\Vert_{L^{q^\prime}}\lesssim t^{\frac{m}{2}}\tau^{-\frac{m}{2}}B_1(t)B_1^{-1}(\tau)(1+\log\frac{t}{\tau})^{\chi_\delta}\big(
\Vert u_1\Vert_{\dot{H}^{r}_q}+\tau^{-(m+1)r}\Vert u_1\Vert_{L^q}\big),
\end{align*}
and for $r\geq1$,
\begin{align*}
&\Vert u(t,\cdot)\Vert_{L^{q^\prime}}\lesssim t^{-\frac{m}{2}}\tau^{-\frac{m}{2}}B_0(t)B_0^{-1}(\tau)(1+\log\frac{t}{\tau})^{\chi_\delta}
\big(\Vert u_1\Vert_{\dot{H}^{r-1}_q}+\tau^{(m+1)(1-r)}\Vert u_1\Vert_{L^q}\big),
\end{align*}
for $0\leq r<1$,
\begin{align*}
&\Vert u(t,\cdot)\Vert_{L^{q^\prime}}\lesssim t^{-\frac{m}{2}}\tau^{-\frac{m}{2}}B_0(t)B_0^{-1}(\tau)(1+\log\frac{t}{\tau})^{\chi_\delta}
\tau^{(m+1)(1-r)}\Vert u_1\Vert_{L^q},
\end{align*}
where $q, q^\prime, r, \chi_\delta, B_0(t)$ and $B_1(t)$ are the same as Theorem \ref{propuwithtau=1}.
\end{theorem}

Based on the above estimates along the conjugate line for the solution to the linear equation (\ref{lineareq}), we can apply the Duhamel principle combined with the contraction mapping principle to
establish  the global existence of solutions to the semilinear problem  (\ref{eq}) in the case where the mass term dominates the damping term.  In order to better describe the global existence of solutions to (\ref{eq}), we first list some indices and notations as follows.

For some $\kappa>1$ let us introduce the space
\begin{equation}\label{inditaldataspace}
\mathcal{A}_\kappa\doteq\big(H^1\cap W^{1,1+\frac{1}{\kappa}}\big)\times\big(L^2\cap L^{1+\frac{1}{\kappa}}\big)
\end{equation}
equipped with the norm
\begin{equation}\label{initialdatanorm}
\Vert(u_0,u_1)\Vert_{\mathcal{A}_\kappa}=\Vert u_0\Vert_{H^1}+\Vert u_0\Vert_{W^{1,1+\frac{1}{\kappa}}}
+\Vert u_1\Vert_{L^2}+\Vert u_1\Vert_{L^{1+\frac{1}{\kappa}}}.
\end{equation}
By the definition (\ref{p0ndef}) of   $p_0(n,\mu)$ , we see that
$p_0\big((m+1)n,\mu\big)>1$ is the positive root of
\begin{equation}\label{p0mnmu}
\big((m+1)n+\frac{\mu-1}{2}\big)(p-1)^2+\big((m+1)n+\mu-3\big)(p-1)-4=0.
\end{equation}
Let $\widetilde{p}_{1}(n,m,\mu)>1$ be the positive root of
\begin{align}
\frac{(m+1)n+\mu-1}{2}(p-1)^2+\big((m+1)n+m+\mu-2\big)(p-1)-4=0.\label{eqp1}
\end{align}
By the definition (\ref{p1ndef}) of  $p_1(n,\mu)$, we have $\widetilde{p}_1(n,0,\mu)=p_1(n,\mu)$.
Set
\begin{align}
\mu^{\star}(n,m)=n+\frac{1}{n+1}-mn(1+\frac{1}{n+1}),\ \ \widetilde{\mu}(n,m)=\frac{(1-m)n^2-(m+2)n+m+1}{n}.\label{mustarmutilde}
\end{align}

Before presenting the global existence of the solution to  (\ref{eq}), it is necessary to  establish the local existence,  that is
\begin{proposition}{\bf{\text{(Local solution).}}}\label{propositionlocalexistence}
Let   $n\geq1,m>-1$, $\mu>0,\nu^2\geq$ such that $\delta\doteq(\mu-1)^2-4\nu^2<\min\{(m+1)^2,1\}$. Assume
\begin{equation}\label{assumptionp}
\left\{\begin{aligned}
&p>1,&\quad&\text{if}\ n=1,\\
&1<p\leq1+\frac{2}{n-1},&\quad&\text{if}\ n\geq2
\end{aligned}\right.
\end{equation}
and the initial data  $(u_0,u_1)\in\mathcal{A}_p$, then
there exists a $T>1$ and a unique solution $$u\in C\big([1,T];H^1\cap L^{p+1}\big)\cap C^1\big([1,T]; L^2\big)$$ to (\ref{eq}).
\end{proposition}
\begin{remark}
It should be noted that since the proof of the local solution to  (\ref{eq}) is similar in method to that of the global solution, we present the detailed proof for the global solution in the main text, while placing the simpler proof framework  of the local solution, i.e., Proposition \ref{propositionlocalexistence},  in the appendix.
\end{remark}
\begin{theorem}{\bf{\text{(Global solution).}}}\label{mainresult}
Let $n\geq1$, $m>-1+\frac{1}{n}$
and $\mu,\nu^2\geq0$ such that $\delta\doteq(\mu-1)^2-4\nu^2\leq0$.
Suppose that  $p$ satisfies (\ref{assumptionp})
and
$\mu>\max\big\{0,\min\{\tilde{\mu}(n,m),\mu^\star(n,m)\}\big\}$. If
\begin{equation}\label{globalindex}
p>\max\big\{p_0\big((m+1)n,\mu\big),\widetilde{p}_1(n,m,\mu)\big\},
\end{equation}
then there exists a $\varepsilon>0$,  such that
for any $(u_0,u_1)\in\mathcal{A}_p$ satisfying  $\Vert(u_0,u_1)\Vert_{\mathcal{A}_p}\leq\varepsilon$, (\ref{eq}) admits a unique global solution
$$u\in C\big([1,\infty);H^1\cap L^{p+1}\big)\cap C^1\big([1,\infty); L^2\big).$$
Here $p_0\big((m+1)n, \mu\big), \widetilde{p}_1(n,m,\mu),\mu^\star(n,m),\tilde{\mu}(n,m)$ are defined by (\ref{p0mnmu}), (\ref{eqp1}) and (\ref{mustarmutilde}), respectively. Moreover, $u$ satisfies the estimates
\begin{align}
&\Vert \nabla u(t,\cdot)\Vert_{L^2}\lesssim t^{-\frac{m+\mu}{2}}(1+\log t)^{\chi_\delta}\Vert (u_0,u_1)\Vert_{\mathcal{A}_p},\label{thmes1}\\
&\Vert\partial_tu(t,\cdot)\Vert_{L^2}\lesssim t^{\frac{m-\mu}{2}}(1+\log t)^{\chi_\delta}\Vert (u_0,u_1)\Vert_{\mathcal{A}_p},\label{thmes2}\\
&\Vert u(t,\cdot)\Vert_{L^2}\lesssim t^{\frac{m-\mu}{2}+1}(1+\log t)^{\chi_\delta}\Vert (u_0,u_1)\Vert_{\mathcal{A}_p},\label{thmes3}\\
&\Vert u(t,\cdot)\Vert_{L^{p+1}}\lesssim t^{-\frac{m+\mu}{2}}A_p(t)(1+\log t)^{\chi_\delta}\Vert (u_0,u_1)\Vert_{\mathcal{A}_p},\label{thmes4}
\end{align}
where $\chi_\delta$ is defined in  Theorem \ref{propuwithtau=1} and $A_p(t)=t^{\max\{\frac{m+1}{2}-(m+1)n(1-\frac{2}{p+1})+\chi_\epsilon,-\frac{(m+1)(n-1)}{2}(1-\frac{2}{p+1})\}}$
with $\chi_\epsilon=\left\{\begin{aligned}
&\epsilon,&\quad&\text{if}\ \ \delta=0,\\
&0,&\quad&\text{if}\ \ \delta\neq0,
\end{aligned}\right.$
for sufficiently small $\epsilon>0$.

\end{theorem}

\begin{remark}\label{reasonablenessremark}
In the special case $m=0$,  (\ref{eq}) reduces to the semilinear wave equation with scale-invariant damping and mass (\ref{waveeq}).
Theorem \ref{mainresult} coincides exactly with the corresponding result in \cite{DaPa2021}  for (\ref{waveeq}).  Hence, our result can be interpreted as an extension from the wave operator $\partial_t^2-\Delta$ to the Gellerstedt operator $\partial_t^2-t^{2m}\Delta$, which confirms the validity of Theorem \ref{mainresult}.
\end{remark}

\begin{remark}\label{remarkprange}
From the structural comparison of (\ref{p0mnmu}) and (\ref{eqp1}), one observes that $p_0\big((m+1)n,\mu)=\widetilde{p}_1(n,m,\mu)$ solves $p-1=\frac{2}{n}.$
Moreover, combining this with Remark \ref{nonempty} below, one sees $p_0,p_1$ are injective with respect to $\mu$. Consequently,  we can compute that  $p_0\big((m+1)n,\mu)=\widetilde{p}_1(n,m,\mu)$  if and only if $\mu=\mu^{\star}(n,m)=n+\frac{1}{n+1}-mn(1+\frac{1}{n+1})$.
In addition,  an inspection of  (\ref{p0mnmu}) and (\ref{eqp1}) reveals that when $\mu$ is large,
\begin{align*}
p_0\big((m+1)n,\mu\big)\sim\frac{4}{(m+1)n+\mu-3},\ \ \widetilde{p}_1(n,m,\mu)\sim\frac{4}{(m+1)n+m+\mu-2},
\end{align*}
which yields
\begin{equation*}
\max\big\{p_0\big((m+1)n,\mu\big),\widetilde{p}_1(n,m,\mu)\big\}=
\left\{\begin{aligned}
&p_0\big((m+1)n,\mu\big),&\quad& \text{if}\ \mu\geq\mu^{\star}(n,m), \\
&\widetilde{p}_1(n,m,\mu), &\quad& \text{if}\ \mu<\mu^{\star}(n,m).
\end{aligned}
\right.
\end{equation*}

Based on the above analysis and the proof of Proposition \ref{proofofproposition}, it follows that
\begin{equation*}
A_p(t)=
\left\{\begin{aligned}
&t^{\frac{m+1}{2}-(m+1)n(1-\frac{2}{p+1})},&\quad& \text{if}\ \mu\geq\mu^{\star}(n,m), \\
&t^{-\frac{(m+1)(n-1)}{2}(1-\frac{2}{p+1})}, &\quad& \text{if}\ \mu<\mu^{\star}(n,m),
\end{aligned}
\right.
\end{equation*}
which means $\mu^\star(n,m)$ represents the threshold between the two type of time exponents appearing in
the
$L^{1+\frac{1}{p}}-L^p$ estimate (\ref{thmes4}) for the solution.

\end{remark}

Remark \ref{nonempty} will show that the condition $\mu>\max\big\{0,\min\{\tilde{\mu}(n,m),\mu^\star(n,m)\}\big\}$  in Theorem \ref{mainresult}  is imposed to ensure that the range of  $p$ is nonempty.
Before proceeding with the explanation, we first state a basic algebraic result.
\begin{lemma}\label{monolemma}
(\cite{DaPa2021}) Let $a_0,c,\alpha,\gamma$ be positive numbers and $b_0\in\mathbb{R}$. Consider the equation
\begin{equation}\label{eqquad}
ax^2+2bx-c=0,\ \ a=a_0+\gamma\alpha, \ \ b=b_0+\alpha.
\end{equation}
Then, the positive solution $x=x(\alpha)$ to (\ref{eqquad}) is strictly decreasing with respect to $\alpha$.
\end{lemma}

\begin{proof}
Taking the derivative of (\ref{eqquad})  with respect to $\alpha$ yields
\begin{align}
2x'(ax+b)+\gamma x^2+2x=0.\label{root}
\end{align}
For the positive root $x$ of (\ref{eqquad}), it follows from the equation that $2ax^2+2bx=c+ax^2>0$, which implies $ax+b>0$. Hence, by (\ref{root}), $x'(\alpha)<0$, which proves that $x(\alpha)$ is decreasing.
\end{proof}

\begin{remark}\label{nonempty}
By Lemma \ref{monolemma}, we have $m>-1+\frac{1}{n}$,  and from the form of (\ref{p0mnmu}) and (\ref{eqp1}), it follows that $p_0\big((m+1)n,\mu\big)$ and $\widetilde{p}_1(n,m,\mu)$ are decreasing with respect to $\mu$. To satisfy condition (2.8), based on  Remark \ref{remarkprange} and the condition $p\leq1+\frac{2}{n-1} (n\geq2)$, we need to ensure that for $n\geq2$,
\begin{equation*}
\max\big\{p_0\big((m+1)n,\mu\big),\ \widetilde{p}_1(n,m,\mu)\big\}<1+\frac{2}{n-1}.
\end{equation*}
Note that $p_0\big((m+1)n,\mu^\star(n,m)\big)=1+\frac{2}{n}$. Hence, by the monotonicity of $p_0\big((m+1)n,\mu\big)$ with respect to $\mu$,  we have
$$p_0\big((m+1)n,\mu\big)\leq p_0\big((m+1)n,\mu^\star(n,m)\big)=1+\frac{2}{n}<1+\frac{2}{n-1}$$
for any $\mu\geq\mu^\star(n,m,\mu)$. That is, $p_0\big((m+1)n,\mu\big)<1+\frac{2}{n-1}$ holds automatically.  Substituting $\widetilde{p}_1(n,m,\mu)=1+\frac{2}{n-1}$ into (\ref{eqp1}) yields $\mu=\frac{(1-m)n^2-(m+2)n+m+1}{n}=\tilde{\mu}(n,m)$. By the  monotonicity of  $\widetilde{p}_1(n,m,\mu)$ with respect to $\mu$, we obtain
$$
\widetilde{p}_1(n,m,\mu)<\widetilde{p}_1(n,m,\tilde{\mu}(n,m))=1+\frac{2}{n-1}, \ \ \text{for any}\ \ \mu>\tilde{\mu}.
$$
\end{remark}

Remarks \ref{remarkprange} and \ref{nonempty} establish the admissible range of
$p$ in Theorem \ref{mainresult} by analyzing the relations among the indices.  Nevertheless, a concrete example better illustrates the conclusion of Theorem \ref{mainresult}.

\begin{example}\label{ex2}
For $n=3$, $m=1>-1+\frac{1}{n}=-\frac{2}{3}$,  it can be  computed  that $\mu^\star(3,1)=-\frac{1}{2},\tilde{\mu}(3,1)=-\frac{7}{3}$, choosing $\mu=1>\max\big\{0,\min\{\mu^\star(3,1),\tilde{\mu}(3,1)\}\big\}=0$, $\nu=2$ such that $\delta=-16<0$ and $p_0=1+\frac{\sqrt{112}-4}{12}\approx 1.549>1.528\approx1+\frac{-6+\sqrt{84}}{6}=\widetilde{p}_1$. Then for
$\max\{p_0,\widetilde{p}_1\}=1+\frac{\sqrt{112}-4}{12}<p\leq 2=1+\frac{2}{n-1}$, (\ref{eq}) has small initial data  global solution.
\end{example}

\begin{remark}\label{deltadayu0case}
Since this paper does not provide blow-up results for (\ref{eq}) in the mass-dominated case, we cannot determine the critical exponent for   (\ref{eq}) in this regime. For the sake of simplicity in both the presentation and the proof, we restrict our detailed analysis to the case $\delta\leq0$. In fact, following the arguments in Section \ref{sectionproofmainresult}  and the calculations in Remarks \ref{remarkprange}, \ref{nonempty}, we can likewise establish the global existence of solutions to (\ref{eq}) in the case  $0<\delta<\min\{\frac{(m+1)^2}{n^2},1\}$.    In this situation, however,  the key  condition (\ref{globalindex}) needs to be modified to
\begin{equation*}
p>\max\big\{p_0\big((m+1)n,\mu-\sqrt{\delta}\big),\widetilde{p}_1(n,m,\mu)\big\}.
\end{equation*}
Moreover,   $\tilde{\mu}(n,m)$ and the admissible range of $m$ are identical to those described in Theorem \ref{mainresult},  whereas $\mu^\star$ becomes
\begin{equation*}
\begin{aligned}
&\mu^\star=\mu^\star(n,m,\delta)\\
&\resizebox{0.97\hsize}{!}{$
\begin{aligned}
&=\frac{2\big((m+1)n-\sqrt{\delta}\big)^2
-\big((m+1)n+m-2\big)(m+1+\sqrt{\delta})\big((m+1)n-\sqrt{\delta}\big)-\big((m+1)n-1\big)(m+1+\sqrt{\delta})^2}{(m+1+\sqrt{\delta})(m+1)(n+1)},
\end{aligned}$}
\end{aligned}
\end{equation*}
which satisfies $\mu^\star(n,m,\mu,0)=\mu^\star(n,m,\mu).$  As in Example \ref{ex2}, one can also consider the more illustrative case, such as $n=3,m=1,\mu=2,\nu=\frac{\sqrt{6}}{6}$.

\end{remark}

\begin{remark}\label{tableanboutblouandglobal}
After establishing the global existence of solutions to (\ref{eq}) in the case where the mass term dominates over the damping term, the study of its critical exponent is further advanced. For a more intuitive understanding of (\ref{eq}), we summarize the related results in Table 4 and Table 5.
\begin{table}[h]\label{eqblowuptable4}
\centering
\setlength{\tabcolsep}{0.5mm}{
  \setlength{\abovecaptionskip}{0pt} % 设置标题与表格之间的垂直间距
  \begin{tabular}{c|c|c}
    \toprule
    The range of $\delta$ & The range of $p$ & References  \\
    \hline
    $\delta\geq0$ & $1<p<\max\{p_S(n+\frac{\mu}{m+1},m),p_F((m+1)n+\frac{\mu-1-\sqrt\delta}{2})\}$ &Thm.1.2 \cite{Pa2025}  \\
    \hline
    $\delta\geq0$ & $p=p_S(n+\frac{\mu}{m+1},m)$ & Thm.1.2 \cite{LaiPa2024}\\
    \hline
    $\delta\geq0$ & $p= p_F((m+1)n+\frac{\mu-1-\sqrt\delta}{2})$ & Thm.1.3 \cite{Pa2025} \\
    \bottomrule
  \end{tabular}}
  \vspace{5pt} % 手动添加一些额外的垂直间距
    \caption{Blow-up results of (\ref{eq})}
\end{table}

\begin{table}[h]\label{eqglobaltable5}
\renewcommand\arraystretch{1.2}
 \centering
\setlength{\tabcolsep}{0.5mm}{
  \setlength{\abovecaptionskip}{0pt} % 设置标题与表格之间的垂直间距
  \begin{tabular}{c|c|c}
    \toprule
    The range of $\delta$ & The range of $p$  & References\\
    \hline
    $\delta\geq(m+1)^2(n-1)^2$ & $p>p_F((m+1)n+\frac{\mu-1-\sqrt\delta}{2})$ &Thm.2.1,2.2 \cite{LiGuo2025}  \\
    \hline
    $\delta\leq0$ & $p>\max\big\{p_0\big((m+1)n,\mu\big),\widetilde{p}_1(n,m,\mu)\big\}$ & This paper \\
    \bottomrule
  \end{tabular}}
  \vspace{5pt} % 手动添加一些额外的垂直间距
    \caption{Global existence of (\ref{eq})}
\end{table}
\vspace{10pt}
\end{remark}

\section{The estimate for solution to the linear EPDT equation on the conjugate line}\label{linearepdtestimate}

\subsection{The expression for the solution to the linear equation}\label{expression}

Denote by $E_0(t,\tau,x)$ and $E_1(t,\tau,x)$ the fundamental solutions to (\ref{lineareq}) with the initial data $(u_0,u_1)=(\delta_0,0)$ and $(u_0,u_1)=(0,\delta_0)$, respectively, where $\delta_0$ is the Dirac function. By Duhamel's principle, the solution to (\ref{eq}) can be expressed as
\begin{align}
u(t,x)=u^l(t,x)+u^n(t,x),\label{expressionultx}
\end{align}
where
\begin{align}
&u^l(t,x)\doteq E_0(t,1,x)\ast_xu_0(x)+E_1(t,1,x)\ast_xu_1(x),\label{ultxdef}\\
&u^n(t,x)\doteq\int_1^tE_1(t,\tau,x)\ast_x\vert u(\tau,x)\vert^pd\tau.\label{untxdef}
\end{align}

Through the variable transform
\begin{equation}\label{sub11}
u(t,x)\doteq t^{-\frac{\mu}{2}}w(t,x),
\end{equation}
(\ref{lineareq}) can be transformed into
\begin{equation}\label{eqw2}
 \left\{
\begin{aligned}
&\partial_t^2w-t^{2m}\Delta w+\frac{1-\delta}{4t^2}w=0,&\quad &t>\tau\geq1,\ x\in\mathbb{R}^n,\\
&w(\tau,x)= w_0(x), \ \partial_tw(\tau,x)= w_1(x),              &\quad &x\in\mathbb{R}^n,
\end{aligned}
 \right.
  \end{equation}
where
\begin{equation*}
\begin{aligned}
w_0(x)\doteq\tau^{\frac{\mu}{2}}u_0(x), \ w_1(x)\doteq\tau^{\frac{\mu}{2}}u_1(x)+\frac{\mu}{2}\tau^{\frac{\mu}{2}-1}u_0(x).
\end{aligned}
\end{equation*}
Let
\begin{equation}\label{sub21}
s\doteq\frac{t^{m+1}}{m+1},\ \eta\doteq\frac{\tau^{m+1}}{m+1}\ \ \text{and} \ \ v(s,x)\doteq w\Bigl(\bigl((m+1)s\bigr)^{\frac{1}{m+1}}, x\Bigr),
\end{equation}
(\ref{eqw2}) is converted into
\begin{equation}\label{eqv2}
 \left\{
\begin{aligned}
&\partial_s^2v-\Delta v+\frac{m}{m+1}\frac{\partial_sv}{s}+\frac{1-\delta}{4(m+1)^2}\frac{v}{s^2}=0,&\quad &s>\eta,\ x\in\mathbb{R}^n, \\
&v(\eta,x)= v_0(x), \ \partial_sv(\eta,x)= v_1(x),               &\quad &x\in\mathbb{R}^n,
\end{aligned}
 \right.
  \end{equation}
where
\begin{equation*}
\begin{aligned}
v_0(x)&\doteq w_0(x)=\bigl((m+1)\eta\bigr)^{\frac{\mu}{2(m+1)}}u_0(x), \\
v_1(x)&\doteq\bigl((m+1)\eta\bigr)^{-\frac{m}{m+1}}w_1(x)=\bigl((m+1)\eta\bigr)^{\frac{1}{m+1}(\frac{\mu}{2}-m)}u_1(x)+\frac{\mu}{2}\big((m+1)\eta\big)^{\frac{1}{m+1}(\frac{\mu}{2}-1-m)}u_0(x)
\end{aligned}
\end{equation*}
and $\delta$ is defined by (\ref{deltadefine}).
Put
\begin{equation}\label{sub31}
\mathcal{U}(s,x)\doteq s^{\frac{m}{2(m+1)}}v(s,x),
\end{equation}
then (\ref{eqv2}) becomes
\begin{equation}\label{masseq}
\left\{
\begin{aligned}
&\partial_s^2\mathcal{U}-\Delta \mathcal{U}+\frac{(m+1)^2-\delta}{4(m+1)^2s^2}\mathcal{U}=0,&\quad &s>\eta, \ x\in\mathbb{R}^n,\\
&\mathcal{U}(\eta,x)= \mathcal{U}_0(x), \ \partial_s\mathcal{U}(\eta,x)= \mathcal{U}_1(x),               &\quad &x\in\mathbb{R}^n,
\end{aligned}
 \right.
\end{equation}
where
\begin{equation*}
\begin{aligned}
\mathcal{U}_0(x)&\doteq\eta^{\frac{m}{2(m+1)}}v_0(x)=\eta^{\frac{m}{2(m+1)}}\big((m+1)\eta\big)^{\frac{\mu}{2(m+1)}}u_0(x),\\
\mathcal{U}_1(x)&\doteq\frac{m}{2(m+1)}\eta^{\frac{m}{2(m+1)}-1}v_0(x)+\eta^{\frac{m}{2(m+1)}}v_1(x)\\
&=\Big(\frac{m}{2(m+1)}\eta^{\frac{m}{2(m+1)}-1}\big((m+1)\eta\big)^{\frac{\mu}{2(m+1)}}+\frac{\mu}{2}\eta^{\frac{m}{2(m+1)}}
\big((m+1)\eta\big)^{\frac{1}{m+1}(\frac{\mu}{2}-1-m)}\Big)u_0(x)\\
&\quad+\eta^{\frac{m}{2(m+1)}}\big((m+1)\eta\big)^{\frac{1}{m+1}(\frac{\mu}{2}-m)}u_1(x).
\end{aligned}
\end{equation*}

After applying the variable transformations (\ref{sub11}),  (\ref{sub21}) and (\ref{sub31}), we convert (\ref{lineareq}) into a mass-type wave equation (\ref{masseq}).  A notable feature of (\ref{masseq}) is that it contains only a mass term without  damping term.  Therefore, we refer to  (\ref{sub11}),  (\ref{sub21}) and (\ref{sub31}) as the \textit{mass-dominated} transformations.

Taking the Fourier transform with respect to the spatial variable $x$ on both sides of (\ref{masseq}) yields
\begin{equation}\label{fouriereq}
\left\{
\begin{aligned}
&\partial_s^2\widehat{\mathcal{U}}+\Big(\vert \xi\vert^2+\frac{(m+1)^2-\delta}{4(m+1)^2s^2}\Big)\widehat{\mathcal{U}}=0,\ &\quad&s>\eta,\ \xi\in\mathbb{R}^n,\\
&\widehat{\mathcal{U}}(\eta,\xi)= \widehat{\mathcal{U}}_0(\xi), \ \partial_s\widehat{\mathcal{U}}(\eta, \xi)= \widehat{\mathcal{U}}_1(\xi),               &\quad&\xi\in\mathbb{R}^n.
\end{aligned}
\right.
\end{equation}
Before establishing the expression for the solution to (\ref{fouriereq}),  we collect some relevant properties of confluent hypergeometric functions, one can see \cite{BoEr1953,OlLo2010}.

The second-order  ordinary differential equation
\begin{equation}\label{confluent}
z\frac{d^2}{dz^2}\mathscr{W}+(c-z)\frac{d}{dz}\mathscr{W}-a\mathscr{W}=0,\ a,c\in\mathbb{C}
\end{equation}
is called the confluent hypergeometric equation. The standard  solution  to (\ref{confluent})  with  a pair of parameters $(a,c)$ are given by

$\bullet$ Kummer's function
\begin{equation}\label{kummer}
M(a,c;z)\doteq\sum\limits_{k=0}\limits^{\infty}\frac{(a)_k}{(c)_kk!}z^k,\ \ z\in\mathbb{C};
\end{equation}

$\bullet$  Tricomi's confluent hypergeometric function
\begin{equation}\label{tricomicon}
U(a,c;z)\sim z^{-a}\sum\limits_{k=0}^{\infty}\frac{(a)_k(a-c+1)_k}{k!}(-z)^{-k}, \ \ \vert argz\vert<\frac{3}{2}\pi,
\end{equation}
where
\begin{equation}\label{ak}
(a)_k\doteq\left\{\begin{aligned}
&a(a+1)\cdots(a+k-1),\ &\quad&k\geq1,\\
&1,\ &\quad& k=0.
\end{aligned}\right.
\end{equation}
If $c\notin\mathbb{Z}$,  a pair of independent solutions to (\ref{confluent}) are
\begin{equation}\label{nointe}
M_1(z)\doteq M(a,c;z), \ \ M_2(z)\doteq z^{1-c}M(a-c+1,2-c;z);
\end{equation}
and if $c\in\mathbb{Z}$, a pair of independent solutions to (\ref{confluent}) are
\begin{equation}\label{cinz}
U_1(z)\doteq U(a,c;z), \ \ U_2(z)\doteq e^zU(c-a,c;-z).
\end{equation}

About $M(a,c;z)$ and $U(a,c;z)$, the following two Propositions are valid, which can be found in Chapter 6 of \cite{BoEr1953} and  Chapter 13 of \cite{OlLo2010}.
\begin{proposition}\label{kummerpro}
\begin{align}
&M(a,c;z) \ \ \text{is entire in}\ \ z\ \text{and}\ \ a; \label{entire}\\
&\text{Kummer's transformation:}\ \ M(a,c;z)=e^zM(c-a,c;-z);\label{Kumtrans1}\\
&\frac{d}{dz}M(a,c;z)=\frac{a}{c}M(a+1,c+1;z);\label{deterive}\\
&(c-a)M(a-1,c;z)+(2a-c+z)M(a,c;z)-aM(a+1,c;z)=0;\label{rela1}\\
&\frac{d}{dz}M(a,c;z)=\frac{a}{z}\big(M(a+1,c;z)-M(a,c;z)\big);\label{rela2}\\
&(c-a)M(a-1,c;z)=(c-a-z)M(a,c;z)+z\frac{d}{dz}M(a,c;z);\label{rela3}\\
&\text{Wronskian determinant:}
\ \ W\big(M_1(z),M_2(z)\big)=(1-c)z^{-c}e^z;\label{Wron1}\\
&M\big(\frac{1}{2}+\nu,1+2\nu;2ix\big)=\Gamma(1+\nu)(\frac{1}{2}x)^{-\nu}e^{ix}J_\nu(x),\label{besselconnection}
\end{align}
where $\Gamma(\cdot)$ denotes the Gamma function, and  $J_\nu(x)$ is the first kind of  Bessel function satisfies  the derivative recurrence relation
\begin{align}
&J_\nu^\prime(z)=J_{\nu-1}(z)-\frac{\nu}{z}J_\nu(z)\label{10.6.2}
\end{align}
and the asymptotic behavior
\begin{align}\label{besselasy}
&J_\nu\big(\frac{z}{2i}\big)\sim\sin \big(\frac{z}{2i}\big)a_\nu\big(\frac{z}{2i}\big)\ \ \text{as}\ \ \vert z\vert\rightarrow\infty,
\ \ \text{and}\ \ \big\vert\frac{d^k}{dz^k}a_\nu\big(\frac{z}{2i}\big)\big\vert\lesssim\vert z\vert^{-\frac{1}{2}-k}\ \ \text{for any}\  k\in\mathbb{N}.
\end{align}
\end{proposition}

\begin{proposition}\label{tricomipro}
\begin{align}
&\text{Kummer's transformation:}\ \ U(a,c;z)=z^{1-c}U(a-c+1,2-c;z);\label{Kumtrans2}\\
&\frac{d}{dz}U(a,c;z)=-aU(a+1,c+1;z);\label{detiveu}\\
&U(a-1,c;z)=(a-c+z)U(a,c;z)-z\frac{d}{dz}U(a,c;z);\label{rela4}\\
&U(a,1;z)=-\frac{1}{\Gamma(a)}\big(\log z+\frac{\Gamma^\prime(a)}{\Gamma(a)}+2\gamma\big)+O( z\log z)\ \ \text{as}\ \ \vert z\vert\rightarrow 0;\label{rela5}\\
&U(a,2;z)=\frac{1}{\Gamma(a)}z^{-1}+O( \log z)\ \ \text{as}\ \ \vert z\vert\rightarrow 0;\label{rela6}\\
&U(a,c;z)=\frac{\Gamma(c-1)}{\Gamma(a)}z^{1-c}+O(z^{c-2})\ \ \text{as}\ \ \vert z\vert\rightarrow 0\ \ \text{for}\ \ \text{Re}c\geq2\ \ \text{and}\ \ c\neq2;\label{rela7}\\
&\vert U(a,c;z)\vert\lesssim\vert z\vert^{-\text{Re}a}\ \ \text{as}\ \ \vert z\vert\rightarrow \infty\ \ \text{for}\ \ \vert \text{arg}z\vert<\frac{3}{2}\pi;\label{rela8}\\
&\text{Wronskian determinant:}
\ \ W\big(U_1(z),U_2(z)\big)=e^{i(c-a)\pi \text{sgn}(\text{Im}z)}z^{-c}e^{z},\label{Wron2}
\end{align}
where $\gamma\doteq\lim\limits_{n\rightarrow\infty}\big(\sum\limits_{k=1}^n\frac{1}{k}-\log n\big)$ is the well-known Euler constant.
\end{proposition}
\begin{remark}\label{uni}
By applying Kummer's transformations (\ref{Kumtrans1}) and (\ref{Kumtrans2}), we see
\begin{equation}\label{kumer}
M_2(z)=z^{1-c}e^zM(1-a,2-c;-z),\  \ U_2(z)=(-1)^{1-c}z^{1-c}e^zU(1-a,2-c;-z).
\end{equation}
Therefore, the two independent  solutions (3.15), (3.16) to (\ref{confluent}) can be expressed in a unified form
\begin{equation}\label{unifiedform}
\mathscr{W}_{01}(z)=\mathscr{W}_0(a,c;z)\ \ \text{and}\ \ \mathscr{W}_{02}(z)=z^{1-c}e^z\mathscr{W}_0(1-a,2-c,-z),
\end{equation}
where
\begin{equation}\label{W0}
\mathscr{W}_0(a,c;z)\doteq\left\{\begin{aligned}&M(a,c,z), &\quad& c\notin\mathbb{Z},\\
&U(a,c;z),&\quad&c\in\mathbb{Z}.\end{aligned}\right.
\end{equation}
 And according to  (\ref{Wron1}) and (\ref{Wron2}), the following Wronskian determinant holds
\begin{equation}\label{Wron23}
 W\big(\mathscr{W}_{01}(z),\mathscr{W}_{02}(z)\big)=\left\{
\begin{aligned}
&(1-c)z^{-c}e^z,&\quad&c\notin\mathbb{Z},\\
&(-1)^{c-1}e^{i(c-a)\pi\text{sgn}(\text{Im}z)}z^{-c}e^z,&\quad&c\in\mathbb{Z}.
\end{aligned}
\right.
\end{equation}
\end{remark}
\begin{lemma}\label{sulutionfou}
Let $\widehat{H}_0(s,\eta,\xi)$ and $\widehat{H}_1(s,\eta,\xi)$ be a pair of independent solutions to (\ref{fouriereq}) satisfying the initial condition
\begin{equation}\label{initial10012}
\begin{pmatrix}
\widehat{H}_0(\eta,\eta,\xi)&\widehat{H}_1(\eta,\eta,\xi)\\
\partial_s\widehat{H}_0(\eta,\eta,\xi)&\partial_s\widehat{H}_1(\eta,\eta,\xi)
\end{pmatrix}
=\begin{pmatrix}
1&0\\
0&1
\end{pmatrix}.
\end{equation}
Then $\widehat{H}_0(s,\eta,\xi)$ and $\widehat{H}_1(s,\eta,\xi)$  can be expressed as
\begin{equation}\label{solutionfunde}
\widehat{H}_0(s,\eta,\xi)=\frac{\left| \begin{array}{cc}
E_1(s,\xi) & E_2(s,\xi) \\
\partial_sE_1(\eta,\xi) & \partial_sE_2(\eta,\xi)
\end{array} \right|}{\left| \begin{array}{cc}
E_1(\eta,\xi) & E_2(\eta,\xi) \\
\partial_sE_1(\eta,\xi) & \partial_sE_2(\eta,\xi)
\end{array} \right|},\ \
\widehat{H}_1(s,\eta,\xi)=\frac{\left| \begin{array}{cc}
E_1(\eta,\xi) & E_2(\eta,\xi) \\
E_1(s,\xi) & E_2(s,\xi)
\end{array} \right|}{\left| \begin{array}{cc}
E_1(\eta,\xi) & E_2(\eta,\xi) \\
\partial_sE_1(\eta,\xi) & \partial_sE_2(\eta,\xi)
\end{array} \right|},
\end{equation}
where
\begin{equation}\label{fundamental}
\begin{aligned}
&E_1(s,\xi)\doteq(2i)^{-\rho}z(s)^\rho e^{-\frac{z(s)}{2}}\mathscr{W}_0(\rho,2\rho;z(s)),\\ &E_2(s,\xi)\doteq(2i)^{-\rho}z(s)^{1-\rho}e^{\frac{z(s)}{2}}\mathscr{W}_0(1-\rho,2-2\rho;-z(s)),
\end{aligned}
\end{equation}
$\mathscr{W}_0$  defined by (\ref{unifiedform}), $z(s)=2is\vert\xi\vert$  and
\begin{equation}\label{rho}
\rho\doteq\left\{
\begin{aligned}
&\frac{1}{2}+\frac{\sqrt{\delta}}{2(m+1)}\in(\frac{1}{2},1), \quad&\quad & 0<\delta<\min\{(m+1)^2,1\},\\
&\frac{1}{2},\quad&\quad & \delta=0,\\
&\frac{1}{2}+i\frac{\sqrt{-\delta}}{2(m+1)},\quad&\quad & \delta<0.
\end{aligned}
\right.
\end{equation}
Furthermore, regarding the time derivatives, we have
\begin{equation}\label{partoiale1e2}
\begin{aligned}
&\partial_sE_1(s,\xi)=(2i)^{1-\rho}\vert\xi\vert z(s)^{\rho-1}e^{-\frac{z(s)}{2}}\mathscr{W}_1(\rho,2\rho;z(s)),\\
&\partial_sE_2(s,\xi)=(2i)^{1-\rho}\vert\xi\vert z(s)^{-\rho}e^{\frac{z(s)}{2}}\mathscr{W}_1(1-\rho,2-2\rho;-z(s)),
\end{aligned}
\end{equation}
where
\begin{equation}\label{W1}
\mathscr{W}_1(a,c;z)\doteq\left\{\begin{aligned}&\frac{z}{2}M(a,c,z)+(c-a)M(a-1,c;z), &\quad& c\notin\mathbb{Z},\\
&\frac{z}{2}U(a,c;z)-U(a-1,c;z),&\quad&c\in\mathbb{Z}.\end{aligned}\right.
\end{equation}
\end{lemma}

\begin{proof}
Let $r=s\vert\xi\vert$ and $\widehat{\mathcal{U}}(s,\xi)=r^\rho V(r)$, where $\rho$  is defined by (\ref{rho}), then  (\ref{fouriereq}) can be reduced to
\begin{equation}\label{eqV2}
rV^{\prime\prime}(r)+2\rho V^{\prime}(r)+rV=0.
\end{equation}
Put $z=2ir$ and $\mathscr{W}(z)=e^{ir}V(r)$, we can transfer (\ref{eqV2}) into
\begin{equation}\label{solutionhuaw}
z\mathscr{W}^{\prime\prime}(z)+(2\rho-z)\mathscr{W}^{\prime}(z)-\rho\mathscr{W}(z)=0,
\end{equation}
which is a confluent  hypergeometric equation with  parameters $(2\rho,\rho)$.  According to Remark \ref{uni},  the two fundamental solutions to (\ref{solutionhuaw}) are
\begin{equation}
\begin{aligned}
\mathscr{W}_{01}(z)=\mathscr{W}_0(\rho,2\rho;z)\ \ \text{and}\ \ \mathscr{W}_{02}(z)=z^{1-2\rho}e^z\mathscr{W}_0(1-\rho,2-2\rho;-z).
\end{aligned}
\end{equation}
Thus, the two independent fundamental  solutions to (\ref{fouriereq})  are
\begin{equation}\label{nextsolution}
\begin{aligned}
E_h(s,\xi)=(s\vert\xi\vert)^\rho e^{-is\vert\xi\vert}\mathscr{W}_{0h}(z(s)), \ \ h=1,2,
\end{aligned}
\end{equation}
which are indeed  (\ref{fundamental}).

We can assume that
\begin{equation}
\begin{aligned}
&\widehat{H}_0(s,\eta,\xi)=D_{01}(\eta,\xi)E_1(s,\xi)+D_{02}(\eta,\xi)E_2(s,\xi),\\
&\widehat{H}_1(s,\eta,\xi)=D_{11}(\eta,\xi)E_1(s,\xi)+D_{12}(\eta,\xi)E_2(s,\xi),
\end{aligned}
\end{equation}
where $D_{01}, D_{02},D_{11}, D_{12}$  can be determined by
\begin{equation}\label{pmatrixcondition}
\begin{pmatrix}
E_1(\eta,\xi)&E_{2}(\eta,\xi)\\
\partial_sE_1(\eta,\xi)&\partial_sE_{2}(\eta,\xi)
\end{pmatrix}
\begin{pmatrix}
D_{01}&D_{11}\\
D_{02}&D_{12}
\end{pmatrix}
=
\begin{pmatrix}
1&0\\
0&1
\end{pmatrix}
\end{equation}
according to the initial conditions (\ref{initial10012}).  A direct computation immediately yields
\begin{equation*}
\begin{pmatrix}
D_{01}&D_{11}\\
D_{02}&D_{12}
\end{pmatrix}
=\frac{\begin{pmatrix}
 \partial_sE_2(\eta,\xi)&-E_2(\eta,\xi)\\
-\partial_sE_1(\eta,\xi)&E_1(\eta,\xi)
\end{pmatrix}}{\left| \begin{array}{cc}
E_1(\eta,\xi) & E_2(\eta,\xi) \\
\partial_sE_1(\eta,\xi) & \partial_sE_2(\eta,\xi)
\end{array} \right|},
\end{equation*}
therefore we get (\ref{solutionfunde}).

About the derivatives,  it can be
straightforward to calculate that if $\delta\neq0$, i.e., $2\rho\neq 1$, then
\begin{align}
\partial_sE_1(s,\xi)=&\rho(s\vert\xi\vert)^{\rho-1}\vert\xi\vert e^{-is\vert\xi\vert}M(\rho,2\rho;z(s))-i\vert\xi\vert(s\vert\xi\vert)^\rho e^{-is\vert\xi\vert}M(\rho,2\rho;z(s))\notag\\
&\quad+2i\vert\xi\vert(s\vert\xi\vert)^\rho e^{-is\vert\xi\vert}z^{-1}\big(\rho M(\rho-1,2\rho;z(s))-(\rho-z)M(\rho,2\rho;z(s))\big)\notag\\
=&\vert\xi\vert(s\vert\xi\vert)^{\rho-1}e^{-is\vert\xi\vert}\Big(\frac{z(s)}{2}M(\rho,2\rho;z(s))+\rho M(\rho-1,2\rho;z(s))\Big),\label{partialse1notinter}
\end{align}
here we used (\ref{rela3}) in the first equality;  similarly,
\begin{align}
&\partial_sE_2(s,\xi)=\vert\xi\vert(s\vert\xi\vert)^{\rho-1}z(s)^{1-2\rho}e^{is\vert\xi\vert}\Big(
-\frac{z(s)}{2}M(1-\rho,2-2\rho;-z(s))+(1-\rho)M(-\rho,2-2\rho;-z(s))\Big).\label{partialse2notinter}
\end{align}
If $\delta=0$, i.e., $2\rho=1,$ using (\ref{rela4}) we have
\begin{align}
&\partial_sE_1(s,\xi)=\rho(s\vert\xi\vert)^{\rho-1}\vert\xi\vert e^{-is\vert\xi\vert}U(\rho,2\rho;z(s))-i\vert\xi\vert(s\vert\xi\vert)^\rho e^{-is\vert\xi\vert}U(\rho,2\rho;z(s))\notag\\
&\quad\quad+2i\vert\xi\vert(s\vert\xi\vert)^\rho e^{-is\vert\xi\vert}\frac{(-\rho+z(s))U(\rho,2\rho;z(s))-U(\rho-1,2\rho;z(s))}{2is\vert\xi\vert}\notag\\
&=\vert\xi\vert(s\vert\xi\vert)^{\rho-1}e^{-is\vert\xi\vert}\big(\frac{z(s)}{2}U(\rho,2\rho;z(s))
-U(\rho-1,2\rho;z(s))\big)\label{partialse1inter}
\end{align}
and

\begin{align}
&\partial_sE_2(s,\xi)=\vert\xi\vert(s\vert\xi\vert)^{\rho-1}e^{is\vert\xi\vert}z(s)^{1-2\rho}\Big(-\frac{z(s)}{2}U(1-\rho,2-2\rho;-z(s))-U(-\rho,2-2\rho;-z(s))\Big).\label{partialse2inter}
\end{align}
Combining (\ref{partialse1notinter})-(\ref{partialse2inter}) yields (\ref{partoiale1e2}) immediately.

\end{proof}

The above lemma provides the explicit expression of the fundamental solutions to  (\ref{fouriereq}). In what follows, we unify the expressions of $\widehat{H}_0(s,\eta,\xi), \widehat{H}_1(s,\eta,\xi)$ and their time derivatives into one form for the convenience of subsequent computations.
\begin{lemma}\label{expressionandderivatives}
The expressions of $\widehat{H}_k(s,\eta,\xi)$, $k=0,1$ and their time derivatives can be written by
\begin{equation}\label{solutionmartix1}
\partial_s^\iota \widehat{H}_k(s,\eta,\xi)=(-1)^kK_\rho(2i\vert\xi\vert)^{\iota-k}\det\big(\Lambda_k^\iota(s,\eta,\xi)\big),
\end{equation}
where
\begin{equation*}
\begin{aligned}
&\Lambda_k^\iota(s,\eta,\xi)\doteq\begin{pmatrix}
z(\eta)^{k-\rho}e^{\frac{z(\eta)}{2}}\mathscr{W}_{1-k}\big(1-\rho,2-2\rho;-z(\eta)\big)
&z(s)^{1-\rho-\iota}e^{\frac{z(s)}{2}}\mathscr{W}_\iota\big(1-\rho,2-2\rho;-z(s)\big)\\
z(\eta)^{\rho-1+k}e^{-\frac{z(\eta)}{2}}\mathscr{W}_{1-k}\big(\rho,2\rho;z(\eta)\big)&
z(s)^{\rho-\iota}e^{-\frac{z(s)}{2}}\mathscr{W}_\iota\big(\rho,2\rho;z(s)\big)
\end{pmatrix}
\end{aligned}
\end{equation*}
and
\begin{equation*}K_\rho\doteq\left\{
\begin{aligned}
&(1-2\rho)^{-1},&\quad& \delta\neq 0,\\
&e^{-\frac{\pi}{2}i},&\quad& \delta=0
\end{aligned}
\right.
\ \ \text{with}\  \rho \ \text{defined by (\ref{rho}).}
\end{equation*}
\end{lemma}

\begin{proof}
According to the expressions of $\widehat{H}_0(s,\eta,\xi)$, $\widehat{H}_1(s,\eta,\xi)$ in Lemma \ref{sulutionfou}, our aim is to get the Wronskian   $W\big(E_1(s,\xi),E_2(s,\xi)\big)$. From (\ref{nextsolution}) we know
\begin{equation}\label{equalexpressions}
E_h(s,\xi)=(2i)^{-\rho}z(s)^\rho e^{-\frac{z(s)}{2}}\mathscr{W}_{0h}(z(s)),\ \ z(s)=2is\vert\xi\vert,\  h=1,2.
\end{equation}
Then for $h=1,2,$
\begin{align}
&\partial_sE_h(s,\xi)=(2i)^{-\rho}\rho z(s)^{\rho-1}2i\vert\xi\vert e^{-\frac{z(s)}{2}}\mathscr{W}_{0h}(z(s))\notag\\
&\ \ +(2i)^{-\rho}z(s)^{\rho}(-i\vert\xi\vert)e^{-\frac{z(s)}{2}}\mathscr{W}_{0h}(z(s))+(2i)^{-\rho}z(s)^{\rho}2i\vert\xi\vert e^{-\frac{z(s)}{2}}\frac{d}{dz}\mathscr{W}_{0h}(z(s))\notag\\
&=(2i)^{1-\rho}z(s)^{\rho-1}\vert\xi\vert e^{-\frac{z(s)}{2}}\Big((\rho-\frac{z(s)}{2})\mathscr{W}_{0h}(z(s))+z(s)\frac{d}{dz}
\mathscr{W}_{0h}(z(s))\Big).\label{partialseh}
\end{align}
Combining (\ref{equalexpressions}), (\ref{partialseh})  and (\ref{Wron23}) we calculate that
\begin{align}
&W\big(E_1(s,\xi),E_2(s,\xi)\big)=E_1(s,\xi)\partial_sE_2(s,\xi)-E_2(s,\xi)\partial_sE_1(s,\xi)\notag\\
&\ \ =(2i)^{1-2\rho}\vert\xi\vert z(s)^{2\rho}e^{-z(s)}W\big(\mathscr{W}_{01}(z(s)),\mathscr{W}_{02}(z(s))\big)\notag\\
&\ \ =\left\{\begin{aligned}
&(2i)^{1-2\rho}(1-2\rho)\vert\xi\vert, &\quad&\text{if}\  \delta\neq0,\ \text{i.e.,}\  2\rho\neq1,\\
&(-1)^{2\rho-1}(2i)^{1-2\rho}\vert\xi\vert e^{i\rho\pi}, &\quad&\text{if}\  \delta=0,\ \text{i.e.,}\ 2\rho=1.
\end{aligned}\right.\label{Wron4}
\end{align}
Substituting (\ref{fundamental}),  (\ref{partoiale1e2}) and (\ref{Wron4})  into (\ref{solutionfunde}) indicates (\ref{solutionmartix1})  immediately.

\end{proof}

$\mathscr{W}_h(\rho,2\rho;z)$, $h=0,1$ satisfies the asymptotic properties  stated in the following Lemma.
\begin{lemma}
For any $k\in\mathbb{N}$, $\mathscr{W}_h(\rho,2\rho;z)$, $h=0,1$ satisfy the following  properties:
\begin{align}
&\text{For small}\ \vert z\vert,\ \mathscr{W}_h\big(\frac{1}{2},1;z\big)=C_h\log z+O(1); \label{propop1}\\
&\text{If} \  2\rho\notin\mathbb{Z}, \ \text{for small}\  \vert z\vert \ \text{we have}\ \ \vert\mathscr{W}_h(\rho,2\rho;z)\vert\leq C;\label{bounded1}\\
&\text{For small}\ \vert z\vert\, \ \text{it holds}\ \big\vert\frac{d^k}{dz^k}\mathscr{W}_h\big(\frac{1}{2},1;z\big) \big\vert\lesssim \vert z\vert^{-k}; \ \label{de1}\\
&\text{For large}\ \vert z\vert\,\ \mathscr{W}_h\big(\frac{1}{2},1;z\big)=c_ha_h(z),  \text{where}\ \big\vert\frac{d^k}{dz^k}a_h(z)\big\vert\lesssim \vert z\vert^{h-\frac{1}{2}-k}; \label{bounded2}\\
&\text{If}\   2\rho\notin\mathbb{Z}, \ \text{for large}\ \vert z\vert, \ \mathscr{W}_h\big(\rho,2\rho;z\big)=C_he^{\frac{z}{2}}\sin\big(\frac{z}{2i}\big)A_h(z),\ \text{where}\ \big\vert\frac{d^k}{dz^k}A_h(z)\big\vert\lesssim \vert z\vert^{h-\text{Re}\rho-k}.\label{bounded3}
\end{align}

\end{lemma}

\begin{proof}
From the definition (\ref{W1}) of $\mathscr{W}_1$  and using (\ref{detiveu}), (\ref{rela4}), we know
\begin{equation}
\begin{aligned}
\mathscr{W}_1\big(\frac{1}{2},1;z\big)&=\frac{1}{2}U\big(\frac{1}{2},1;z\big)-\frac{z}{2}U\big(\frac{1}{2},1;z\big)+z\frac{d}{dz}U\big(\frac{1}{2},2;z\big)=\frac{1-z}{2}U\big(\frac{1}{2},1;z\big)-\frac{z}{2}U\big(\frac{3}{2},2;z\big),
\end{aligned}
\end{equation}
and note  $\mathscr{W}_0\big(\frac{1}{2},1;z\big)=U\big(\frac{1}{2},1;z\big)$, then the asymptotic behaviors (\ref{rela5}) and (\ref{rela6}) gives (\ref{propop1}). (\ref{bounded1}) can be derived by the analyticity of Kummer's function  and Tricomi's confluent hypergeometric function in $z$. By the definitions (\ref{W0}), (\ref{W1}) of $\mathscr{W}_0,\mathscr{W}_1$ and using (\ref{detiveu}), we get
\begin{align}
&\frac{d^k}{dz^k}\mathscr{W}_0\big(\frac{1}{2},1;z\big)=(-1)^k\big(\frac{1}{2}\big)_kU
\big(\frac{1}{2}+k,1+k;z\big)\label{derive1}\\
&\frac{d^k}{dz^k}\mathscr{W}_1\big(\frac{1}{2},1;z\big)=(-1)^{k-1}\frac{k}{2}(\frac{1}{2})_{k-1}U\big(k-\frac{1}{2},k;z\big)
+(-1)^k\frac{z}{2}(\frac{1}{2})_kU\big(\frac{1}{2}+k,1+k;z\big)\notag\\
&\quad\quad\quad\quad\quad\quad\quad\quad-(-1)^k(-\frac{1}{2})_kU\big(-\frac{1}{2}+k,1+k;z\big),\label{derive2}
\end{align}
where $(a)_k$ is defined by (\ref{ak}). According to the asymptotic behaviors (\ref{rela5})-(\ref{rela7}) of $U\big(a,c;z\big)$ when $\vert z\vert$ is small and (\ref{rela8}) when $\vert z\vert$  is large, we can derive (\ref{de1}) and (\ref{bounded2}), respectively. From (\ref{besselconnection}) follows that
\begin{align}
\mathscr{W}_0(\rho,2\rho;z)&=M(\rho,2\rho;z)=\Gamma\big(\rho+\frac{1}{2}\big)(4i)^{\rho-\frac{1}{2}}z^{-\rho+\frac{1}{2}}e^{\frac{z}{2}}
J_{\rho-\frac{1}{2}}\big(\frac{z}{2i}\big)\notag\\
&\sim e^{\frac{z}{2}}\sin\big(\frac{z}{2i}\big)z^{-\rho+\frac{1}{2}}a_{\rho-\frac{1}{2}}\big(\frac{z}{2i}\big)\doteq e^{\frac{z}{2}}\sin\big(\frac{z}{2i}\big)
A_0(z).\label{w0m}
\end{align}
By (\ref{besselasy}) we get
\begin{equation}\label{dkdzk1}
\begin{aligned}
\big\vert\frac{d^k}{dz^k}A_0(z)\big\vert\sim\big\vert\sum\limits_{j=0}^kz^{-\rho+\frac{1}{2}-j}\frac{d^{k-j}}{dz^{k-j}}
a_{\rho-\frac{1}{2}}\big(\frac{z}{2i}\big)\big\vert\lesssim\vert z\vert^{-\text{Re}\rho-k}.
\end{aligned}
\end{equation}
Furthermore, from~(\ref{W1}),   by~(\ref{10.6.2}), (\ref{w0m}), (\ref{besselasy})  and (\ref{rela3}), we have
\begin{align}
&\mathscr{W}_1(\rho,2\rho;z)=\frac{z}{2}M(\rho,2\rho;z)+(\rho-z)M(\rho,2\rho;z)+ z\frac{d}{dz}M(\rho,2\rho;z)\notag\\
&\sim zM(\rho,2\rho;z)+M(\rho,2\rho;z)+e^{\frac{z}{2}}z^{-\rho+\frac{1}{2}}\Big(J_{\rho-\frac{1}{2}}\big(\frac{z}{2i}\big)
+zJ_{\rho-\frac{1}{2}}\big(\frac{z}{2i}\big)+z\frac{d}{dz}J_{\rho-\frac{1}{2}}\big(\frac{z}{2i}\big)\Big)\notag\\
&\sim e^{\frac{z}{2}}\sin\big(\frac{z}{2i}\big)\big(zA_0(z)+A_0(z)\big)+e^{\frac{z}{2}}z^{-\rho+\frac{1}{2}}\Big(J_{\rho-\frac{1}{2}}\big(\frac{z}{2i}\big)
+zJ_{\rho-\frac{1}{2}}\big(\frac{z}{2i}\big)+zJ_{\rho-\frac{3}{2}}\big(\frac{z}{2i}\big)\Big)\notag\\
&\sim e^{\frac{z}{2}}\sin\big(\frac{z}{2i}\big)\Big(\underbrace{zA_0(z)+A_0(z)+z^{-\rho+\frac{3}{2}}a_{\rho-\frac{3}{2}}\big(\frac{z}{2i}\big)}_{A_1(z)}\Big).\label{w1m}
\end{align}
Proceeding as in (\ref{dkdzk1}),  we have
\begin{equation}\label{dkdzk2}
\begin{aligned}
&\big\vert\frac{d^k}{dz^k}A_1(z)\big\vert\lesssim \vert z\vert^{1-\text{Re}\rho-k}.
\end{aligned}
\end{equation}
Combining (\ref{w0m}) and (\ref{dkdzk2}) yields (\ref{bounded3}).

\end{proof}

\subsection{$L^{q}-L^q{^{\prime}}$ estimates for the solution to  linear equation (\ref{masseq}) on the conjugate line}\label{sec3.1}

In this section, we will make full use of the properties of the Kummer's function and Tricomi's confluent hypergeometric function  to establish the $L^q-L^{q^\prime}$ estimates for the solution $\mathcal{U}(s,x)$ to (3.9) with the parameters $q,\ q'$ lying on the conjugate line, i.e., $\frac{1}{q}+\frac{1}{q^{\prime}}=1$. Then by the corresponding  inverse change of variables (\ref{sub11}), (\ref{sub21}) and (\ref{sub31}), we can establish the $L^q-L^{q^\prime}$ estimates for the solution $u(t,x)$ to the linear EPDT equation (\ref{lineareq}), i.e., Theorem \ref{propuwithtau=1} and Theorem \ref{estimateu0=0}.

From the analysis of Section \ref{expression},
\begin{equation}\label{masseqsolution}
\mathcal{U}(s,x)\doteq\mathcal{F}^{-1}\Big(\widehat{H}_0(s,\eta,\xi)\widehat{\mathcal{U}}_0(\xi)+
\widehat{H}_1(s,\eta,\xi)\widehat{\mathcal{U}}_1(\xi)\Big)
\end{equation}
is the solution to (\ref{masseq}). It is worth noting that $M(a,c;z)$ and $U(a,c;z)$  exhibit different asymptotic behaviors near the origin and at infinity, we need to use Littlewood-Paley theory to decompose the frequency space into high, medium and low frequency parts and discuss separately.

Specifically,  choosing the non-increasing  cutoff function
\begin{equation}\label{cutoff}
\psi\in C_0^\infty\big([0,\infty)\big), \ \ \psi\equiv1\ \ \text{on}\ \ [0,1]\ \ \text{and}\ \ \text{supp}\psi\subset[0,2].
\end{equation}
For small sufficiently small $K>0$,  define
\begin{equation*}
\begin{aligned}
&\psi_1(s,\eta,\xi)\doteq\psi\big(\frac{s\vert\xi\vert}{K}\big)\psi\big(\frac{\eta\vert\xi\vert}{K}\big),\\
&\psi_2(s,\eta,\xi)\doteq\psi\big(\frac{\eta\vert\xi\vert}{K}\big)\Big(1-\psi\big(\frac{s\vert\xi\vert}{K}\big)\Big),\\
&\psi_3(s,\eta,\xi)\doteq1-\psi\big(\frac{\eta\vert\xi\vert}{K}\big),
\end{aligned}
\end{equation*}
where $\psi_1+\psi_2+\psi_3\equiv1$ for any $s\geq\eta\geq\frac{1}{m+1}$  and any $\xi\in\mathbb{R}^n$ in the frequency space.
It is obvious that
\begin{align}
&\text{supp}\psi_1(s,\eta,\cdot)\subset\Omega_1(s)\doteq\big\{\xi\in\mathbb{R}^n: \vert\xi\vert\leq\frac{2K}{s}\big\},\notag\\
&\text{supp}\psi_2(s,\eta,\cdot)\subset\Omega_2(s,\eta)\doteq\big\{
\xi\in\mathbb{R}^n: \frac{K}{s}\leq\vert\xi\vert\leq\frac{2K}{\eta}\big\},\notag\\
&\text{supp}\psi_3(s,\eta,\cdot)\subset\Omega_3(\eta)\doteq\big\{\xi\in\mathbb{R}^n: \vert\xi\vert\geq\frac{K}{\eta}\big\}.\label{3region}
\end{align}
From (\ref{masseqsolution}), we know
\begin{align}
&\nabla \mathcal{U}(s,x)\sim\mathcal{F}^{-1}\Big(\xi\widehat{H}_0(s,\eta,\xi)\widehat{\mathcal{U}}_0(\xi)+\xi
\widehat{H}_1(s,\eta,\xi)\widehat{\mathcal{U}}_1(\xi)\Big),\label{nablau}\\
&\partial_s\mathcal{U}(s,x)=\mathcal{F}^{-1}\Big(\partial_s\widehat{H}_0(s,\eta,\xi)\widehat{\mathcal{U}}_0(\xi)+\partial_s
\widehat{H}_1(s,\eta,\xi)\widehat{\mathcal{U}}_1(\xi)\Big).\label{partialsu}
\end{align}
Hence, in order to obtain the $L^q-L^{q^\prime}$ estimates of $\mathcal{U}$, $\partial_s\mathcal{U}$ and $\nabla\mathcal{U}$,   we need to study the following Fourier multipliers
\begin{equation}\label{anaylis}
\mathcal{F}^{-1}\Big(\psi_j(s,\eta,\xi)b(\xi)\partial_s^\iota\widehat{H}_k(s,\eta,\xi)\hat{f}(\xi)\Big),\ \ j=1,2,3, \ \iota=0,1,
\end{equation}
where the choice of $b(\xi)$ is related to  three regions (\ref{3region}) in our decomposition. In different regions, we utilize different asymptotic properties of the confluent hypergeometric function, which leads to varying regularity requirements on the initial data in each region.

\subsubsection{The required analytical techniques}
The main theoretical tools used in establishing the $L^q-L^{q^\prime}$ estimates of $\mathcal{U}$, $\partial_s\mathcal{U}$ and $\nabla\mathcal{U}$ are Littlewood-Paley theory and interpolation theory for linear operators. For convenience in the subsequent proofs, we list these classical analytical techniques separately in this section.

Similar to (\ref{cutoff}), choosing an non-increasing cutoff function $\varphi(y)\in C_0^\infty(\mathbb{R}^n)$ such that $\text{supp}\varphi\subset\{y\in\mathbb{R}^n:\vert y\vert\leq2\}$ and $\varphi(y)=1$ when $\vert y\vert\leq1$, let
\begin{equation}\label{psi0}
\begin{aligned}
&\varphi_0(y)\doteq\varphi(y)\ \ \text{and}\ \ \varphi_j(y)\doteq\varphi(2^{-j}y)-\varphi(2^{-j+1}y),\ \ \text{for}\ \ 1\leq j\in\mathbb{N}.
\end{aligned}
\end{equation}
Then
\begin{equation*}
\begin{aligned}
\text{supp}\varphi_j\subset\{y\in\mathbb{R}^n:2^{j-1}\leq\vert y\vert\leq2^{j+1}\}\ \ \text{and}\ \ \sum\limits_{k=0}^{\infty}\varphi_k(y)=1, \  \text{for any} \  y\in\mathbb{R}^n.
\end{aligned}
\end{equation*}
The decomposition (\ref{psi0}) also known as dyadic decomposition, its main purpose is to derive global estimates from local ones, that is
\begin{proposition}\label{littlewoodpaleythm} (\cite{Bren1975})
Assume $a\in L^\infty(\mathbb{R}^n)$ and $\Vert\mathcal{F}^{-1}\big(a\varphi_j\hat{f}\big)\Vert_{L^{q^\prime}(\mathbb{R}^n)}\leq C\Vert f\Vert_{L^q}(\mathbb{R}^n)$ uniformly holds for any $j\in\mathbb{N}$, where $\varphi_j$  is defined by (\ref{psi0}) and $1<q\leq 2$, $\frac{1}{q}+\frac{1}{q^\prime}=1$, then
\begin{equation*}
\Vert\mathcal{F}^{-1}\big(a\hat{f}\big)\Vert_{L^{q^\prime}(\mathbb{R}^n)}\leq C\Vert f\Vert_{L^q(\mathbb{R}^n)}.
\end{equation*}
\end{proposition}

\begin{proposition}\label{propopo} (\cite{PaPhd2018})
Let $\kappa$ be a positive number and $u$ be a smooth radially function on $\{\xi\in\mathbb{R}^n:\vert\xi\vert\geq1\}$ such that
$$\vert\partial_\xi^\alpha u(\xi)\vert\lesssim\vert\xi\vert^{-\kappa-\vert\alpha\vert}\ \ \text{for any}\ \vert\alpha\vert\leq n+1.$$ Then
$
\mathcal{F}^{-1}\Big(\big(1-\psi(\vert\xi\vert)\big)u(\xi)\Big)\in L^1(\mathbb{R}^n),
$
where $\psi$ is defined by (\ref{cutoff}).

\end{proposition}

Moreover, we also need the following result to obtain $L^1-L^\infty$ estimate.

\begin{proposition}\label{propositionpare} (\cite{PaPhd2018, ReYa2000})
Let $\tau\geq\tau_0>0$, where $\tau_0$ is a positive number.  Assume that the amplitude function $v=v(\eta)\in C_0^\infty(\mathbb{R}^n)$  and $\text{supp}v(\eta)\subset[\frac{1}{2},2]$, then the oscillating integral
$\mathcal{F}_{\eta\rightarrow x}^{-1}\big(e^{-i\tau\vert\eta\vert}v(\eta)\big)$ satisfies the following $L^\infty-L^\infty$ estimate
\begin{equation*}
\Vert\mathcal{F}_{\eta\rightarrow x}^{-1}\big(e^{-i\tau\vert\eta\vert}v(\eta)\big)\Vert_{L^\infty(\mathbb{R}^n_x)}\lesssim(1+\tau)^{-\frac{n-1}{2}}
\sum\limits_{\vert\alpha\vert\leq s}\Vert\partial_\eta^\alpha v(\eta)\Vert_{L^\infty(\mathbb{R}^n_\eta)},\ \ s>\frac{n+3}{2}.
\end{equation*}
\end{proposition}

To facilitate the application of Proposition \ref{propositionpare}, we make appropriate modifications based on  $\varphi(y)$. Let
\begin{equation*}
\begin{aligned}
\phi_0(\eta,\xi)\doteq\sum\limits_{j=-\infty}^0\tilde{\phi}_j\big(\frac{\eta\vert\xi\vert}{K}\big) \ \ \text{and}\ \ \phi_j(\eta,\xi)\doteq\tilde{\phi}_j\big(\frac{\eta\vert\xi\vert}{K}\big),\ \ \text{for}\ \ 1\leq j\in\mathbb{N},
\end{aligned}
\end{equation*}
where
$\tilde{\phi}_j(y)\doteq\varphi(2^{-j}y)-\varphi(2^{-j+1}y), \ \ \text{for any}\ j\in\mathbb{Z}.$ It is easy to see that $\text{supp}\tilde{\phi}_0(y)\subset[\frac{1}{2},2]$ and $\sum\limits_{j=-\infty}^{\infty}\tilde{\phi}_j(y)=1$ holds for any $y\neq0$.

\begin{proposition}\label{L1theorem} (\cite{PaPhd2018})
Let $f$  be a smooth radially symmetric function on $\{\xi:\vert\xi\vert\geq1\}$ such that
\begin{equation*}
\vert\partial_\xi^\alpha f(\xi)\vert\lesssim\vert\xi\vert^{-b-\vert\alpha\vert}, \ \ \text{where}\ \ b>0.
\end{equation*}
Let $\chi$ be a cutoff function such that $\chi=1$ on $[0,1]$ and $\text{supp}\chi\subset[0,2]$. Then
\begin{equation*}
\mathcal{F}^{-1}\big((1-\chi(\xi))f(\xi)\big)\in L^1(\mathbb{R}^n).
\end{equation*}
\end{proposition}

\begin{proposition}\label{prop251012} (\cite{PaPhd2018})
Let $f\in C_0^\infty(\mathbb{R}^n)$ be a radially function and $\alpha>0$, then $\mathcal{F}^{-1}\big(\vert\xi\vert^\alpha f(\xi)\big)\in L^1(\mathbb{R}^n)$.
\end{proposition}

\begin{proposition}\label{interpolationtheorem}(\cite{BeLo1976})
Given $0<\theta<1,$  $\ s_0\neq s_1$ and  $ 1<p_0,p_1<\infty$, put
\begin{align*}
s_\star&\doteq(1-\theta)s_0+\theta s_1,\\
\frac{1}{p_\star}&\doteq\frac{1-\theta}{p_0}+\frac{\theta}{p_1},
\end{align*}
the following interpolation space result is valid:
\begin{align*}
[\dot{H}^{s_0}_{p_0},\dot{H}^{s_1}_{p_1}]=\dot{H}^{s_\star}_{p_\star}, \ \ s_0\neq s_1, \ 1<p_0,p_1<\infty.
\end{align*}
In other words, if $f\in\dot{H}^{s_0}_{p_0}\cap\dot{H}^{s_1}_{p_1}$, then $f\in\dot{H}^{s_\star}$  and
$\Vert f\Vert_{\dot{H}^{s_\star}_{p_\star}}\lesssim \Vert f\Vert_{\dot{H}^{s_0}_{p_0}}^{1-\theta}\Vert f\Vert_{\dot{H}^{s_1}_{p_1}}^{\theta}.$
\end{proposition}

\subsubsection{Estimates in the low-frequency region}\label{propositionlow}
\begin{proposition}\label{lowfre}
Let $\delta<\min\{(m+1)^2,1\}$  and $\mathcal{U}(s,x)$ be the solution to (\ref{masseq}),  the initial data $\mathcal{U}_0,\mathcal{U}_1\in L^q$.   Then we have
\begin{align*}
&\Vert\mathcal{F}^{-1}\Big(\psi_1(s,\eta,\xi)\widehat{D\mathcal{U}}(s,\xi)\Big)\Vert_{L^{q^\prime}}\lesssim s^{\text{Re}\rho-1-n(\frac{1}{q}-\frac{1}{q^\prime})}\eta^{-\text{Re}\rho}(1+\log\frac{s}{\eta})^{\chi_\delta}\Big(\Vert\mathcal{U}_0\Vert_{L^q}
+\eta\Vert\mathcal{U}_1\Vert_{L^q}\Big)
\end{align*}
and
\begin{align*}
&\Vert\mathcal{F}^{-1}\Big(\psi_1(s,\eta,\xi)\widehat{\mathcal{U}}(s,\xi)\Big)\Vert_{L^{q^\prime}}\lesssim s^{\text{Re}\rho-n(\frac{1}{q}-\frac{1}{q^\prime})}\eta^{-\text{Re}\rho}(1+\log\frac{s}{\eta})^{\chi_\delta}\Big(\Vert\mathcal{U}_0\Vert_{L^q}
+\eta\Vert\mathcal{U}_1\Vert_{L^q}\Big),
\end{align*}
where $1< q\leq 2$, $\frac{1}{q}+\frac{1}{q^\prime}=1$ and   $\chi_\delta$, $\rho$  are defined by  Theorem \ref{propuwithtau=1} and Lemma \ref{sulutionfou}, respectively.
\end{proposition}

\begin{proof}
By the expressions (\ref{nablau}), (\ref{partialsu}) of $\nabla\mathcal{U}$ and $\partial_s\mathcal{U}$ we should consider $b(\xi)=\xi^{1-\iota}$ in (\ref{anaylis}).
Note that $\eta\vert\xi\vert\leq s\vert\xi\vert\leq 2K$ in the zone $\Omega_{1}(s)$,  then by Lemma \ref{expressionandderivatives} and (\ref{bounded1}), in the case of $\delta\neq0$ we have
\begin{align}
&\vert\partial_s^\iota\widehat{H}_k(s,\eta,\xi)\vert\sim\vert\xi\vert^{\iota-k}\vert
\det\Lambda_k^\iota(s,\eta,\xi)\vert\notag\\
&\lesssim \vert\xi\vert^{\iota-k}\big\vert z(\eta)^{k-\rho}z(s)^{\rho-\iota}e^{\frac{z(\eta)-z(s)}{2}}\mathscr{W}_{1-k}(1-\rho,2-2\rho;-z(\eta))
\mathscr{W}_\iota(\rho,2\rho;z(s))\big\vert\notag\\
&\quad+\vert\xi\vert^{\iota-k}\big\vert z(s)^{1-\rho-\iota}z(\eta)^{\rho-1+k}e^{\frac{z(s)-z(\eta)}{2}}
\mathscr{W}_\iota(1-\rho,2-2\rho;-z(s))\mathscr{W}_{1-k}(\rho,2\rho;z(\eta))\big\vert\notag\\
&\lesssim \vert\xi\vert^{\iota-k}(\eta\vert\xi\vert)^{k-\text{Re}\rho}(s\vert\xi\vert)^{\text{Re}\rho-\iota}+
\vert\xi\vert^{\iota-k}(s\vert\xi\vert)^{1-\text{Re}\rho-\iota}(\eta\vert\xi\vert)^{\text{Re}\rho-1+k}\notag\\
&\lesssim s^{\text{Re}\rho-\iota}\eta^{k-\text{Re}\rho},\label{casedeltaneq0}
\end{align}
where  we used the fact that $\eta^{\text{Re}\rho-1+k}s^{1-\text{Re}\rho-\iota}\leq\eta^{k-\text{Re}\rho}s^{\text{Re}\rho-\iota}$ if $\text{Re}\rho\geq\frac{1}{2}$ in the last inequality.
In the case of $\delta=0$,  by (\ref{propop1}) we deduce that
\begin{align}
&\vert\partial_s^\iota\widehat{H}_k(s,\eta,\xi)\vert\notag\\
&\sim s^{\frac{1}{2}-\iota}\eta^{k-\frac{1}{2}}\big\vert e^{\frac{z(s)-z(\eta)}{2}}\mathscr{W}_\iota(\frac{1}{2},1;-z(s))\mathscr{W}_{1-k}(\frac{1}{2},1;z(\eta))-
e^{-\frac{z(s)-z(\eta)}{2}}\mathscr{W}_\iota(\frac{1}{2},1;z(s))\mathscr{W}_{1-k}(\frac{1}{2},1;-z(\eta))\big\vert\notag\\
&\lesssim s^{\frac{1}{2}-\iota}\eta^{k-\frac{1}{2}}\big\vert e^{\frac{z(s)-z(\eta)}{2}}\log(-z(s))\log(z(\eta))-e^{-\frac{z(s)-z(\eta)}{2}}\log(z(s))\log(-z(\eta))\big\vert\notag\\
&= s^{\frac{1}{2}-\iota}\eta^{k-\frac{1}{2}}\Big\vert
\Big(\cos\big((s-\eta)\vert\xi\vert\big)+i\sin\big((s-\eta)\vert\xi\vert\big)\Big)\big(\log(2s\vert\xi\vert)-i\frac{\pi}{2}\big)
\big(\log(2\eta\vert\xi\vert)+i\frac{\pi}{2}\big)\notag\\
&\quad\quad\quad-\Big(\cos\big((s-\eta)\vert\xi\vert\big)-i\sin\big((s-\eta)\vert\xi\vert\big)\Big)
\big(\log(2s\vert\xi\vert)+i\frac{\pi}{2}\big)\big(\log(2\eta\vert\xi\vert)-i\frac{\pi}{2}\big)
\Big\vert\notag\\
&= s^{\frac{1}{2}-\iota}\eta^{k-\frac{1}{2}}\Big\vert  2i\sin\big((s-\eta)\vert\xi\vert\big)\big(\log(2s\vert\xi\vert)\log(2\eta\vert\xi\vert)+\frac{\pi^2}{4}\big)+i\pi
\cos\big((s-\eta)\vert\xi\vert\big)\log\frac{s}{\eta}\Big\vert\notag\\
&\lesssim s^{\frac{1}{2}-\iota}\eta^{k-\frac{1}{2}}(1+\log\frac{s}{\eta})=s^{\text{Re}\rho-\iota}\eta^{k-\text{Re}\rho}(1+\log\frac{s}{\eta}),
\label{casedelta=0}
\end{align}
where we use the  boundedness of $\big\vert\log(2s\vert\xi\vert)\log(2\eta\vert\xi\vert)\big\vert$ from above in $\Omega_1(s)$ in the last inequality. Combining (\ref{casedeltaneq0}) and (\ref{casedelta=0}) we get
\begin{equation}\label{partialshksetaxi}
\vert\partial_s^\iota\widehat{H}_k(s,\eta,\xi)\vert\lesssim s^{\text{Re}\rho-\iota}\eta^{k-\text{Re}\rho}(1+\log \frac{s}{\eta})^{\chi_\delta}.
\end{equation}
%where $\chi_\delta$ is defined by (\ref{chidelta}).
Since $\vert\xi\vert\leq 2Ks^{-1}$ in $\Omega_1(s)$, it follows from  (\ref{partialshksetaxi}) that
\begin{equation}\label{linfty}
\begin{aligned}
\vert\xi\vert^{1-\iota}\vert\partial_s^\iota\widehat{H}_k(s,\eta,\xi)\vert\lesssim s^{\text{Re}\rho-1}\eta^{k-\text{Re}\rho}(1+\log \frac{s}{\eta})^{\chi_\delta}.
\end{aligned}
\end{equation}
Hence by Hausdorff-Young inequality, % i.e., Proposition \ref{HausdorffYoung},
H$\ddot{\text{o}}$lder's inequality and (\ref{linfty}) we have
\begin{align}
\Vert\mathcal{F}^{-1}\Big(\vert\xi\vert^{1-\iota}&\psi_1(s,\eta,\xi)\partial_s^\iota\widehat{H}_k(s,\eta,\xi)\hat{f}(\xi)\Big)\Vert_{L^{q^\prime}}
\leq\Vert \vert\xi\vert^{1-\iota}\psi_1(s,\eta,\xi)\partial_s^\iota\widehat{H}_k(s,\eta,\xi)\hat{f}(\xi)\Vert_{L^q}\notag\\
&\leq \Vert \vert\xi\vert^{1-\iota}\partial_s^\iota\widehat{H}_k(s,\eta,\xi)\Vert_{L^\infty_\xi(\Omega_1(s))}\Vert \psi_1(s,\eta,\cdot)\hat{f}(\cdot)\Vert_{L^q}\notag\\
&\leq\Vert \vert\xi\vert^{1-\iota}\partial_s^\iota\widehat{H}_k(s,\eta,\xi)\Vert_{L^\infty_\xi(\Omega_1(s))}\Vert \psi_1(s,\eta,\cdot)\Vert_{L^{\frac{qq^\prime}{q^\prime-q}}}\Vert \hat{f}\Vert_{L^{q^\prime}}\notag\\
&\lesssim s^{\text{Re}\rho-1-n(\frac{1}{q}-\frac{1}{q^\prime})}\eta^{k-\text{Re}\rho}(1+\log \frac{s}{\eta})^{\chi_\delta}\Vert f\Vert_{L^q},\label{lowestimate1}
\end{align}
where we calculate that
\begin{equation*}
\begin{aligned}
&\Vert\psi_1(s,\eta,\cdot)\Vert_{L^{\frac{qq^\prime}{q^\prime-q}}}=\Big(\int_{\{\xi\in\mathbb{R}^n:\vert\xi\vert\leq\frac{2K}{s}\}}
\psi_1(s,\eta,\xi)^{\frac{qq^\prime}{q^\prime-q}}d\xi\Big)^{\frac{q^\prime-q}{qq^\prime}}\leq\Big(\int_{\{\xi\in\mathbb{R}^n:
\vert\xi\vert\leq\frac{2K}{s}\}}
d\xi\Big)^{\frac{q^\prime-q}{qq^\prime}}\lesssim s^{-n(\frac{1}{q}-\frac{1}{q^\prime})}.
\end{aligned}
\end{equation*}
Hence from  (\ref{nablau}), (\ref{partialsu})  and  (\ref{lowestimate1}) we obtain
\begin{align}
\Vert\mathcal{F}^{-1}\Big(\psi_1(s,\eta,\xi)\widehat{\nabla\mathcal{U}}(s,\xi)\Big)\Vert_{L^{q^\prime}}&\leq
\Vert \mathcal{F}^{-1}\Big(\psi_1(s,\eta,\xi)\xi\widehat{H}_0(s,\eta,\xi)\mathcal{U}_0(\xi)+\psi_1(s,\eta,\xi)
\widehat{H}_1(s,\eta,\xi)\xi\mathcal{U}_1(\xi)\Big)\Vert_{L^{q^\prime}}\notag\\
&\lesssim s^{\text{Re}\rho-1-n(\frac{1}{q}-\frac{1}{q^\prime})}\eta^{-\text{Re}\rho}(1+\log\frac{s}{\eta})^{\chi_\delta}
\Big(\Vert\mathcal{U}_0\Vert_{L^q}+\eta\Vert\mathcal{U}_1\Vert_{L^q}\Big),\label{lowestimatenablau}
\end{align}
and  similarly
\begin{equation}\label{lowestimatepartialsu}
\begin{aligned}
&\Vert\mathcal{F}^{-1}\Big(\psi_1(s,\eta,\xi)\partial_s\widehat{\mathcal{U}}(s,\xi)\Big)\Vert_{L^{q^\prime}}\lesssim s^{\text{Re}\rho-1-n(\frac{1}{q}-\frac{1}{q^\prime})}\eta^{-\text{Re}\rho}(1+\log\frac{s}{\eta})^{\chi_\delta}
\Big(\Vert\mathcal{U}_0\Vert_{L^q}+\eta\Vert\mathcal{U}_1\Vert_{L^q}\Big).
\end{aligned}
\end{equation}

Regarding the estimate for $\mathcal{F}^{-1}\Big(\psi_1(s,\eta,\xi)\widehat{\mathcal{U}}(s,\eta,\xi)\Big)$, it can be inferred from (\ref{partialshksetaxi}), that is
\begin{align}
&\Vert\mathcal{F}^{-1}\Big(\psi_1(s,\eta,\xi)\widehat{\mathcal{U}}(s,\eta,\xi)\Big)\Vert_{L^{q^\prime}}=\Vert\mathcal{F}^{-1}\Big(\psi_1(s,\eta,\xi)\widehat{H}_0(s,\eta,\xi)\widehat{\mathcal{U}}_0(\xi)
+\psi_1(s,\eta,\xi)\widehat{H}_1(s,\eta,\xi)\widehat{\mathcal{U}}_1(\xi)\Big)\Vert_{L^{q^\prime}}\notag\\
&\lesssim \Vert\psi_1(s,\eta,\xi)\Vert_{L^{\frac{qq^\prime}{q^\prime-q}}}\Big(\Vert\widehat{H}_0(s,\eta,\xi)\Vert_{L^\infty(\Omega_1(s))}\Vert
\widehat{\mathcal{U}}_0(\xi)\Vert_{L^{q^\prime}}+\Vert\widehat{H}_1(s,\eta,\xi)\Vert_{L^\infty(\Omega_1(s))}\Vert
\widehat{\mathcal{U}}_1(\xi)\Vert_{L^{q^\prime}}  \Big)\notag\\
&\lesssim s^{\text{Re}\rho-n(\frac{1}{q}-\frac{1}{q^\prime})}\eta^{-\text{Re}\rho}(1+\log\frac{s}{\eta})^{\chi_\delta}\big(
\Vert\mathcal{U}_0\Vert_{L^q}+\eta\Vert\mathcal{U}_1\Vert_{L^q}\big).\label{lowestimateu}
\end{align}

Therefore by (\ref{lowestimatenablau})- (\ref{lowestimateu}) we complete the proof.

\end{proof}

\subsubsection{Estimates in the high-frequency region}\label{propositionhigh}

In the region $\Omega_3(\eta)$,  the inequality $\vert\xi\vert\geq\frac{K}{\eta}$ holds, which implies that $\vert\xi\vert$  is unbounded from above. Consequently, the argument employed for $\Omega_1(s)$ in Proposition \ref{lowfre} cannot be applied here. Our approach is to further perform a dyadic decomposition of $\Omega_3(\eta)$ so that $\vert\xi\vert$ remains bounded in each dyadic shell. This enables us to derive uniform local estimates within each shell. By subsequently invoking Proposition \ref{littlewoodpaleythm},
which provides the fundamental link between local and global estimates,
we ultimately obtain the $L^1-L^\infty$ estimate and $L^2-L^2$ estimate  of  $\mathcal{F}^{-1}\Big(\psi_3(s,\eta,\xi)\vert\xi\vert^{k-\iota-r}
\partial_s^\iota\widehat{H}_k(s,\eta,\xi)\hat{f}(\xi)\Big)$.  Then, by applying the Riesz-Thorin interpolation theorem, we can get  the desired $L^q-L^{q^\prime}$ estimate.

\begin{proposition}\label{highfrequencyestimate}
Let $\delta<\min\{(m+1)^2,1\}$  and $\mathcal{U}(s,x)$ be the solution to (\ref{masseq}),  the initial data  $(\mathcal{U}_0,\mathcal{U}_1)\in H^{r+1}_q(\mathbb{R}^n)\times H^r_q(\mathbb{R}^n)$. Then we have
\begin{equation*}
\begin{aligned}
&\Vert\mathcal{F}^{-1}\big(\psi_3(s,\eta,\xi)\widehat{D\mathcal{U}}(s,\xi)\big)\Vert_{L^{q^\prime}(\mathbb{R}^n)}
\lesssim (\frac{s}{\eta})^{-\frac{n-1}{2}(\frac{1}{q}-\frac{1}{q^\prime})}\big(\Vert\mathcal{U}_0\Vert_{\dot{H}^{r+1}_q(\mathbb{R}^n)}+\Vert\mathcal{U}_1
\Vert_{\dot{H}^{r}_q(\mathbb{R}^n)}\big),\\
&\Vert\mathcal{F}^{-1}\big(\psi_3(s,\eta,\xi)\widehat{\mathcal{U}}(s,\xi)\big)\Vert_{L^{q^\prime}(\mathbb{R}^n)}
\lesssim (\frac{s}{\eta})^{-\frac{n-1}{2}(\frac{1}{q}-\frac{1}{q^\prime})}\big(\Vert\mathcal{U}_0\Vert_{\dot{H}^{r}_q(\mathbb{R}^n)}+\Vert\mathcal{U}_1
\Vert_{\dot{H}^{r-1}_q(\mathbb{R}^n)}\big),
\end{aligned}
\end{equation*}
where $q\in(1,2]$, $\frac{1}{q}+\frac{1}{q^\prime}=1$ and $r=n(\frac{1}{q}-\frac{1}{q^\prime})$. Furthermore, if $r<1$, we have
\begin{equation*}
\Vert\mathcal{F}^{-1}\big(\psi_3(s,\eta,\xi)\widehat{\mathcal{U}}(s,\xi)\big)\Vert_{L^{q^\prime}(\mathbb{R}^n)}
\lesssim (\frac{s}{\eta})^{-\frac{n-1}{2}(\frac{1}{q}-\frac{1}{q^\prime})}\big(\Vert\mathcal{U}_0\Vert_{\dot{H}^{r}_q(\mathbb{R}^n)}+\eta^{1-r}\Vert\mathcal{U}_1
\Vert_{L^q(\mathbb{R}^n)}\big).
\end{equation*}
\end{proposition}

\begin{proof}
By Lemma \ref{expressionandderivatives},  (\ref{bounded2}) and  (\ref{bounded3}), we have
\begin{equation}\label{matrix1}
\Lambda_k^\iota(s,\eta,\xi)\sim
\left\{
\begin{aligned}
&
\begin{pmatrix}
\sin (-\eta\vert\xi\vert)b_{11}(\eta\vert\xi\vert)
&\sin (-s\vert\xi\vert)b_{12}(s\vert\xi\vert)\\
\sin (\eta\vert\xi\vert)b_{21}(\eta\vert\xi\vert)&
\sin (s\vert\xi\vert)b_{22}(s\vert\xi\vert)
\end{pmatrix},&\quad&\text{if}\ \  \delta\neq0,\\
&
\begin{pmatrix}
e^{i\eta\vert\xi\vert}\widetilde{b}_{11}(\eta\vert\xi\vert)
&e^{is\vert\xi\vert}\widetilde{b}_{12}(s\vert\xi\vert)\\
e^{-i\eta\vert\xi\vert}\widetilde{b}_{21}(\eta\vert\xi\vert)&
e^{-is\vert\xi\vert}\widetilde{b}_{22}(s\vert\xi\vert)
\end{pmatrix},&\quad&\text{if}\ \  \delta=0,
\end{aligned}
\right.
\end{equation}
where
\begin{align*}
&b_{11}(z)\doteq(2iz)^{k-\rho}A_{1-k}(-2iz), &\ & \widetilde{b}_{11}(z)\doteq(2iz)^{k-\frac{1}{2}}a_{1-k}(-2iz),\\
&b_{12}(z)=(2iz)^{1-\rho-\iota}A_{\iota}(-2iz), &\ & \widetilde{b}_{12}(z)\doteq(2iz)^{\frac{1}{2}-\iota}a_{\iota}(-2iz),\\
&b_{21}(z)\doteq(2iz)^{\rho-1+k}A_{1-k}(2iz), &\ & \widetilde{b}_{21}(z)\doteq(2iz)^{-\frac{1}{2}+k}a_{1-k}(2iz),\\
&b_{22}(z)\doteq(2iz)^{\rho-\iota}A_{\iota}(2iz), &\ & \widetilde{b}_{22}(z)\doteq(2iz)^{\frac{1}{2}-\iota}a_{\iota}(2iz)
\end{align*}
and by Leibniz's rule we can get
\begin{align}
\vert\frac{d^n}{dz^n}b_{jl}(z)\big\vert\lesssim \vert z\vert^{-n},\ \ \vert\frac{d^n}{dz^n}\widetilde{b}_{jl}(z)\big\vert\lesssim \vert z\vert^{-n}, \ \ j,l=1,2,\ \ \forall n\in\mathbb{N}.\label{bounded4}
\end{align}
Then
\small{
\begin{align}
&\partial_s^\iota\widehat{H}_k(s,\eta,\xi)=(-1)^kK_\rho(2i\vert\xi\vert)^{\iota-k}\det\big(\Lambda_k^\iota(s,\eta,\xi)\big)
\sim\vert\xi\vert^{\iota-k}\notag\\
&\times
\left\{
\begin{aligned}
&\sin(-\eta\vert\xi\vert)\sin(s\vert\xi\vert)b_{11}(\eta\vert\xi\vert)b_{22}(s\vert\xi\vert)-
\sin(\eta\vert\xi\vert)\sin(-s\vert\xi\vert)b_{12}(s\vert\xi\vert)b_{21}(\eta\vert\xi\vert), &\ & \text{if}\ \delta\neq0,\\
&e^{-i(s-\eta)\vert\xi\vert}\widetilde{b}_{11}(\eta\vert\xi\vert)\widetilde{b}_{22}(s\vert\xi\vert)-
e^{i(s-\eta)\vert\xi\vert}\widetilde{b}_{12}(s\vert\xi\vert)\widetilde{b}_{21}(\eta\vert\xi\vert),&\ &  \text{if}\ \delta=0.
\end{aligned}
\right.\label{express111}
\end{align}
}

Our aim now is to estimate
$$\mathcal{F}^{-1}\Big(\psi_3(s,\eta,\xi)\vert\xi\vert^{k-\iota-r}\partial_s^\iota\widehat{H}_k(s,\eta,\xi)\hat{f}(\xi)\Big),$$
where $r>0$ is determined later. Note that
$\psi_3(s,\eta,\xi)\tilde{\phi}_j\big(\frac{\eta\vert\xi\vert}{K}\big)=0$ holds for $j<0$   and $\sum\limits_{j=0}^\infty\phi_j(\eta,\xi)=1$,
it is sufficient to derive the  estimates of the norm for the following  expressions  that hold uniformly with respect to $j\in\mathbb{N}$:
\begin{align}
&\mathcal{F}^{-1}\Big(\psi_3(s,\eta,\xi)\phi_j(\eta,\xi)\sin(-\eta\vert\xi\vert)\sin(s\vert\xi\vert)
b_{11}(\eta\vert\xi\vert)b_{22}(s\vert\xi\vert)\vert\xi\vert^{-r}\hat{f}(\xi)\Big),\label{desired1}\\
&\mathcal{F}^{-1}\Big(\psi_3(s,\eta,\xi)\phi_j(\eta,\xi)\sin(\eta\vert\xi\vert)\sin(-s\vert\xi\vert)
b_{12}(s\vert\xi\vert)b_{21}(\eta\vert\xi\vert)\vert\xi\vert^{-r}\hat{f}(\xi)\Big),\label{desired2}\\
&\mathcal{F}^{-1}\Big(\psi_3(s,\eta,\xi)\phi_j(\eta,\xi)e^{-i(s-\eta)\vert\xi\vert}
\widetilde{b}_{11}(\eta\vert\xi\vert)\widetilde{b}_{22}(s\vert\xi\vert)\vert\xi\vert^{-r}\hat{f}(\xi)\Big),\label{desired3}\\
&\mathcal{F}^{-1}\Big(\psi_3(s,\eta,\xi)\phi_j(\eta,\xi)e^{i(s-\eta)\vert\xi\vert}
\widetilde{b}_{12}(s\vert\xi\vert)\widetilde{b}_{21}(\eta\vert\xi\vert)\vert\xi\vert^{-r}\hat{f}(\xi)\Big).\label{desired4}
\end{align}
Here we should point out the purpose of introducing the additional factor $\vert\xi\vert^{-r}$. In fact, the estimates for the above four expressions inevitably depend on $j$, and we precisely use the boundedness of $\vert\xi\vert^{-r}$ from above to eliminate the dependence on $j$, as shown in (\ref{lplq}).

We now present the $L^1-L^\infty$ estimate and $L^2-L^2$ estimate for
$$\mathcal{F}^{-1}\Big(\psi_3(s,\eta,\xi)\phi_j(\eta,\xi)e^{-i(s-\eta)\vert\xi\vert}
\widetilde{b}_{11}(\eta\vert\xi\vert)\widetilde{b}_{22}(s\vert\xi\vert)\vert\xi\vert^{-r}\hat{f}(\xi)\Big),$$
then the $L^q-L^{q^\prime}$ estimates for  (\ref{desired3}) can derived by applying the interpolation theorem. The arguments and the results  for the remaining three terms (\ref{desired1}), (\ref{desired2}) and  (\ref{desired4}) are entirely analogous and will be omitted.

By performing the change of variables $\xi=\frac{2^jK}{\eta}\zeta$, and then applying Proposition \ref{propositionpare},  we get
\begin{equation*}
\begin{aligned}
&\Vert\mathcal{F}^{-1}\Big(\psi_3(s,\eta,\xi)\phi_j(\eta,\xi)e^{-i(s-\eta)\vert\xi\vert}
\widetilde{b}_{11}(\eta\vert\xi\vert)\widetilde{b}_{22}(s\vert\xi\vert)\vert\xi\vert^{-r}\Big)\Vert_{L^\infty}\\
&=\big(\frac{2^jK}{\eta}\big)^{n-r}\Vert\mathcal{F}^{-1}\Big(e^{-i(s-\eta)\frac{2^jK\vert\zeta\vert}{\eta}}
\big(1-\psi(2^j\vert\zeta\vert)\big)\widetilde{\phi}_0(\vert\zeta\vert)\widetilde{b}_{11}(2^jK\vert\zeta\vert)
\widetilde{b}_{22}(2^jK\frac{s}{\eta}\vert\zeta\vert)\vert\zeta\vert^{-r}\Big)(\frac{2^jK}{\eta}y)\Vert_{L^\infty_y}\\
&\lesssim \big(\frac{2^jK}{\eta}\big)^{n-r}\big(1+\frac{2^jK(s-\eta)}{\eta}\big)^{-\frac{n-1}{2}}\sum\limits_{\vert\alpha\vert\leq h}\Vert\partial_\zeta^\alpha\Big(\widetilde{\phi}_0(\vert\zeta\vert)\big(1-\psi(2^j\vert\zeta\vert)\big)\widetilde{b}_{11}
(2^jK\vert\zeta\vert)\widetilde{b}_{22}\big(2^jK\frac{s}{\eta}\vert\zeta\vert\big)\vert\zeta\vert^{-r}\Big)\Vert_{L^\infty_\zeta},
\end{aligned}
\end{equation*}
where $h>\frac{n+3}{2}$.
According to (\ref{bounded4}), we know
\begin{align*}
&\Vert\partial_\zeta^\alpha\Big(\widetilde{\phi}_0(\vert\zeta\vert)\big(1-\psi(2^j\vert\zeta\vert)\big)\widetilde{b}_{11}
(2^jK\vert\zeta\vert)\widetilde{b}_{22}\big(2^jK\frac{s}{\eta}\vert\zeta\vert\big)\vert\zeta\vert^{-r}\Big)\Vert_{L^\infty}\\
&=
\sup\limits_{\frac{1}{2}\leq\vert\zeta\vert\leq 2}\sum\limits_{\vert\alpha_1\vert+\vert\alpha_2\vert+\vert\alpha_3\vert+\vert\alpha_4\vert=\vert\alpha\vert\leq h}\Big\vert\partial_\zeta^{\alpha_1}\big(\widetilde{\phi}_0(\vert\zeta\vert)\big(1-\psi(2^jK\vert\zeta\vert)\big)\big)\big
\vert\\
&\quad\quad\times\big\vert\partial_\zeta^{\alpha_2}\big(\widetilde{b}_{11}(2^jK\vert\zeta\vert)\big)\big\vert\big\vert\partial_\zeta^{\alpha_3}\big(\widetilde{b}_{22}
\big(2^jK\frac{s}{\eta}\vert\zeta\vert\big)\big)\big\vert\big\vert\partial_\zeta^{\alpha_4}\big(\vert\zeta\vert^{-r}\big)\big\vert\\
&\lesssim\sup\limits_{\frac{1}{2}\leq\vert\zeta\vert\leq 2}\sum\limits_{\vert\alpha_1\vert+\vert\alpha_2\vert+\vert\alpha_3\vert+\vert\alpha_4\vert=\vert\alpha\vert\leq h}\vert\zeta\vert^{-r-\vert\alpha_2\vert-\vert\alpha_3\vert-\vert\alpha_4\vert}\lesssim 1,
\end{align*}
so
\begin{equation*}
\begin{aligned}
&\Vert\mathcal{F}^{-1}\Big(\psi_3(s,\eta,\xi)\phi_j(\eta,\xi)e^{-i(s-\eta)\vert\xi\vert}
\widetilde{b}_{11}(\eta\vert\xi\vert)\widetilde{b}_{22}(s\vert\xi\vert)\vert\xi\vert^{-r}\Big)
\Vert_{L^\infty}\lesssim %2^{j(n-r)}s^{-\frac{n-1}{2}}\eta^{r-n-\frac{n-1}{2}}=
2^{j(n-r)}s^{-\frac{n-1}{2}}\eta^{r-\frac{n+1}{2}}.
\end{aligned}
\end{equation*}
Hence by Young's inequality we have
\begin{equation*}
\begin{aligned}
&\Vert\mathcal{F}^{-1}\Big(\psi_3(s,\eta,\xi)\phi_j(\eta,\xi)e^{-i(s-\eta)\vert\xi\vert}
\widetilde{b}_{11}(\eta\vert\xi\vert)\widetilde{b}_{22}(s\vert\xi\vert)\vert\xi\vert^{-r}\hat{f}(\xi)\Big)\Vert_{L^\infty}\\
&=\Vert\mathcal{F}^{-1}\Big(\psi_3(s,\eta,\xi)\phi_j(\eta,\xi)e^{-i(s-\eta)\vert\xi\vert}
\widetilde{b}_{11}(\eta\vert\xi\vert)\widetilde{b}_{22}(s\vert\xi\vert)\vert\xi\vert^{-r}\Big)\ast f(x)\Vert_{L^\infty_x}\\
&\lesssim \Vert\mathcal{F}^{-1}\Big(\psi_3(s,\eta,\xi)\phi_j(\eta,\xi)e^{-i(s-\eta)\vert\xi\vert}
\widetilde{b}_{11}(\eta\vert\xi\vert)\widetilde{b}_{22}(s\vert\xi\vert)\vert\xi\vert^{-r}\Big)\Vert_{L^\infty}\Vert f\Vert_{L^1}\\
&\lesssim 2^{j(n-r)}s^{-\frac{n-1}{2}}\eta^{r-\frac{n+1}{2}}\Vert f\Vert_{L^1},
\end{aligned}
\end{equation*}
which gives
\begin{equation}\label{l1linfty}
\begin{aligned}
&\Vert\mathcal{F}^{-1}\Big(\psi_3(s,\eta,\xi)\phi_j(\eta,\xi)\vert\xi\vert^{k-\iota-r}\partial_s^\iota
\widehat{H}_k(s,\eta,\xi)\hat{f}(\xi)\Big)\Vert_{L^\infty}\lesssim 2^{j(n-r)}s^{-\frac{n-1}{2}}\eta^{r-\frac{n+1}{2}}\Vert f\Vert_{L^1}.
\end{aligned}
\end{equation}

Applying Parseval's identity  and the boundness  of $\widetilde{b}_{11}, \widetilde{b}_{22}$ we have
\begin{equation*}
\begin{aligned}
&\Vert\mathcal{F}^{-1}\Big(\psi_3(s,\eta,\xi)\phi_j(\eta,\xi)e^{-i(s-\eta)\vert\xi\vert}
\widetilde{b}_{11}(\eta\vert\xi\vert)\widetilde{b}_{22}(s\vert\xi\vert)\vert\xi\vert^{-r}\hat{f}(\xi)\Big)\Vert_{L^2}\\
&= \Vert\psi_3(s,\eta,\xi)\phi_j(\eta,\xi)e^{-i(s-\eta)\vert\xi\vert}
\widetilde{b}_{11}(\eta\vert\xi\vert)\widetilde{b}_{22}(s\vert\xi\vert)\vert\xi\vert^{-r}\hat{f}(\xi)\Vert_{L^2}\\
&\lesssim  \Vert\psi_3(s,\eta,\xi)\phi_j(\eta,\xi)e^{-i(s-\eta)\vert\xi\vert}
\widetilde{b}_{11}(\eta\vert\xi\vert)\widetilde{b}_{22}(s\vert\xi\vert)\vert\xi\vert^{-r}\Vert_{L^\infty}\Vert f\Vert_{L^2}\\
&=2^{-jr}\eta^{r}\sup\limits_{\frac{1}{2}\leq\vert\zeta\vert\leq2}\Big(\big(1-\psi(2^jK\vert\zeta\vert)\big)
\tilde{\phi}_0(\vert\zeta\vert)\vert\zeta\vert^{-r}\Big)\Vert f\Vert_{L^2}\lesssim 2^{-jr}\eta^{r}\Vert f\Vert_{L^2},
\end{aligned}
\end{equation*}
which implies
\begin{equation}\label{l2l2}
\begin{aligned}
&\Vert\mathcal{F}^{-1}\Big(\psi_3(s,\eta,\xi)\phi_j(\eta,\xi)\vert\xi\vert^{k-\iota-r}\partial_s^\iota
\widehat{H}_k(s,\eta,\xi)\hat{f}(\xi)\Big)\Vert_{L^2}\lesssim 2^{-jr}\eta^{r}\Vert f\Vert_{L^2}.
\end{aligned}
\end{equation}

From (\ref{l1linfty})-(\ref{l2l2}), and using Riesz-Thorin interpolation theorem, we have
\begin{equation}\label{lplq}
\begin{aligned}
&\Vert\mathcal{F}^{-1}\Big(\psi_3(s,\eta,\xi)\phi_j(\eta,\xi)\vert\xi\vert^{k-\iota-r}\partial_s^\iota
\widehat{H}_k(s,\eta,\xi)\hat{f}(\xi)\Big)\Vert_{L^{q^\prime}}\lesssim 2^{j\big(n(\frac{1}{q}-\frac{1}{q^\prime})-r\big)}\eta^{r-\frac{n+1}{2}(\frac{1}{q}-\frac{1}{q^\prime})}s^{-\frac{n-1}{2}(\frac{1}{q}-\frac{1}{q^\prime})}\Vert f\Vert_{L^q},
\end{aligned}
\end{equation}
where $\frac{1}{q}+\frac{1}{q^\prime}=1$ and $1<q\leq2$.  From (\ref{lplq}), it can be seen that only $r\geq n(\frac{1}{q}-\frac{1}{q^\prime})$  can help us to obtain the uniform-in-$j$ estimate
\begin{equation*}\label{lplq1}
\begin{aligned}
&\Vert\mathcal{F}^{-1}\Big(\psi_3(s,\eta,\xi)\phi_j(\eta,\xi)\vert\xi\vert^{k-\iota-r}\partial_s^\iota
\widehat{H}_k(s,\eta,\xi)\hat{f}(\xi)\Big)\Vert_{L^{q^\prime}}\lesssim \eta^{r-\frac{n+1}{2}(\frac{1}{q}-\frac{1}{q^\prime})}s^{-\frac{n-1}{2}(\frac{1}{q}-\frac{1}{q^\prime})}\Vert f\Vert_{L^q}.
\end{aligned}
\end{equation*}
It is better to require the lowest regularity of the initial data, so we choose  $r\doteq n(\frac{1}{q}-\frac{1}{q^\prime})$.  Based on  the properties of $\phi_j$ and by Proposition \ref{littlewoodpaleythm}, for $1<q\leq2$ and $\frac{1}{q}+\frac{1}{q^\prime}=1$ we have
\begin{equation}\label{lplq2}
\begin{aligned}
&\Vert\mathcal{F}^{-1}\Big(\psi_3(s,\eta,\xi)\vert\xi\vert^{k-\iota-r}\partial_s^\iota
\widehat{H}_k(s,\eta,\xi)\hat{f}(\xi)\Big)\Vert_{L^{q^\prime}}\lesssim (\frac{s}{\eta})^{-\frac{n-1}{2}(\frac{1}{q}-\frac{1}{q^\prime})}\Vert f\Vert_{L^q}.
\end{aligned}
\end{equation}
Therefore, if follows from (\ref{lplq2}) that
\begin{align}
&\Vert\mathcal{F}^{-1}\big(\psi_3(s,\eta,\xi)\widehat{\mathcal{U}}(s,\xi)\big)\Vert_{L^{q^\prime}}=\Vert\mathcal{F}^{-1}\Big(\psi_3(s,\eta,\xi)\widehat{H}_0(s,\eta,\xi)\widehat{\mathcal{U}}_0(\xi)
+\psi_3(s,\eta,\xi)\widehat{H}_1(s,\eta,\xi)\widehat{\mathcal{U}}_1(\xi)\Big)\Vert_{L^{q^\prime}}\notag\\
&\lesssim (\frac{s}{\eta})^{-\frac{n-1}{2}(\frac{1}{q}-\frac{1}{q^\prime})}\Big(
\Vert \mathcal{F}^{-1}\big(\vert\xi\vert^{r}\widehat{\mathcal{U}}_0(\xi)\big)\Vert_{L^q}+
\Vert \mathcal{F}^{-1}\big(\vert\xi\vert^{r-1}\widehat{\mathcal{U}}_1(\xi)\big)\Vert_{L^q}\Big)\notag\\
&=(\frac{s}{\eta})^{-\frac{n-1}{2}(\frac{1}{q}-\frac{1}{q^\prime})}\big(
\Vert\mathcal{U}_0\Vert_{\dot{H}^r_q}+\Vert\mathcal{U}_1\Vert_{\dot{H}^{r-1}_q}\big),\label{lplq3}
\end{align}
and similarly,
\begin{equation} \label{lplq4}
\begin{aligned}
\Vert\mathcal{F}^{-1}\Big(\psi_3(s,\eta,\xi)\widehat{D\mathcal{U}}(s,\xi)\Big)\Vert_{L^{q^\prime}}\lesssim (\frac{s}{\eta})^{-\frac{n-1}{2}(\frac{1}{q}-\frac{1}{q^\prime})}\big(
\Vert\mathcal{U}_0\Vert_{\dot{H}^{r+1}_q}+\Vert\mathcal{U}_1\Vert_{\dot{H}^{r}_q}\big).
\end{aligned}
\end{equation}

If $r=n(\frac{1}{q}-\frac{1}{q^\prime})<1$, we need to modify the estimate $ \Vert\mathcal{F}^{-1}\Big( \psi_3(s,\eta,\xi)\widehat{H}_1(s,\eta,\xi)\widehat{\mathcal{U}}_1(\xi)\Big)\Vert_{L^{q^\prime}}$ in (\ref{lplq3}).  Choosing $\chi\in C_0^\infty\big([ 0,\infty)\big)$ and $\text{supp}\chi\subset[0,\epsilon)$, where $0<\epsilon\ll1$. Then we have $(1-\chi)(1-\psi)=1-\psi$.  By (\ref{lplq2}) and Young's inequality we get
\begin{equation*}
\begin{aligned}
&\Vert\mathcal{F}^{-1}\Big( \psi_3(s,\eta,\xi)\widehat{H}_1(s,\eta,\xi)\widehat{\mathcal{U}}_1(\xi)\Big)\Vert_{L^{q^\prime}}=\Vert\mathcal{F}^{-1}\Big(\big(1-\psi(\frac{\eta\vert\xi\vert}{K})\big)\vert\xi\vert^{1-r}\widehat{H}_1(s,\eta,\xi)
\big(1-\chi(\frac{\eta\vert\xi\vert}{K}))\big)\vert\xi\vert^{r-1}\widehat{\mathcal{U}}_1(\xi)\Big)\Vert_{L^{q^\prime}}\\
&\lesssim (\frac{s}{\eta})^{-\frac{n-1}{2}(\frac{1}{q}-\frac{1}{q^\prime})}\Vert\mathcal{F}^{-1}\Big(\big(1-\chi(\frac{\eta\vert\xi\vert}{K})\big)
(\eta\vert\xi\vert)^{r-1}\eta^{1-r}\widehat{\mathcal{U}}_1(\xi)\Big)\Vert_{L^q}\\
&\lesssim (\frac{s}{\eta})^{-\frac{n-1}{2}(\frac{1}{q}-\frac{1}{q^\prime})}\eta^{1-r}\Vert\mathcal{F}^{-1}\Big(\big(
1-\chi(\frac{\eta\vert\xi\vert}{K})\big)(\eta\vert\xi\vert)^{r-1}\Big)\Vert_{L^1}\Vert\mathcal{U}_1\Vert_{L^q}.
\end{aligned}
\end{equation*}
Note that $\frac{d^\alpha}{d\zeta^\alpha}(\vert\zeta\vert^{r-1})\lesssim \vert\zeta\vert^{r-1-\vert\alpha\vert}$, by Proposition \ref{propopo}, we have $\mathcal{F}^{-1}\Big(\big(1-\chi(\vert\zeta\vert)\big)\vert\zeta\vert^{r-1}\Big)\in L^1$. Thus
\begin{equation*}
\Vert\mathcal{F}^{-1}\Big(\big(
1-\chi(\frac{\eta\vert\xi\vert}{K})\big)(\eta\vert\xi\vert)^{r-1})\Big)\Vert_{L^1}
=K^{r-1}\Vert\mathcal{F}^{-1}\Big(\big(
1-\chi(\vert\zeta\vert)\big)\vert\zeta\vert^{r-1}\Big)\Vert_{L^1}< \infty.
\end{equation*}
Then
\begin{equation}\label{lplq5}
\begin{aligned}
&\Vert\mathcal{F}^{-1}\Big( \psi_3(s,\eta,\xi)\widehat{H}_1(s,\eta,\xi)\widehat{\mathcal{U}}_1(\xi)\Big)\Vert_{L^{q^\prime}}\lesssim (\frac{s}{\eta})^{-\frac{n-1}{2}(\frac{1}{q}-\frac{1}{q^\prime})}\eta^{1-r}\Vert\mathcal{U}_1\Vert_{L^q}.
\end{aligned}
\end{equation}
Hence in the case of $r=n(\frac{1}{q}-\frac{1}{q^\prime})<1$,  from (\ref{lplq3}) and (\ref{lplq5}) we have
\begin{equation}\label{lplq6}
\begin{aligned}
&\Vert\mathcal{F}^{-1}\big(\psi_3(s,\eta,\xi)\widehat{\mathcal{U}}(s,\xi)\big)\Vert_{L^{q^\prime}}
\lesssim(\frac{s}{\eta})^{-\frac{n-1}{2}(\frac{1}{q}-\frac{1}{q^\prime})}\big(
\Vert\mathcal{U}_0\Vert_{\dot{H}^r_q}+\eta^{1-r}\Vert\mathcal{U}_1\Vert_{L^q}\big).
\end{aligned}
\end{equation}

As a result, we complete the proof by combining (\ref{lplq3}), (\ref{lplq4}), and (\ref{lplq6}).

\end{proof}

\subsubsection{Estimates in the  medium- frequency  region}\label{propositionmedium}

\begin{proposition}\label{propmedium}
Let $\delta<\min\{(m+1)^2,1\}$  and $\mathcal{U}(s,x)$ be the solution to (\ref{masseq}), the initial data  $\mathcal{U}_0,\mathcal{U}_1\in L^q$, $1<q\leq2$, $\frac{1}{q}+\frac{1}{q^\prime}=1$. Then we have
\begin{equation*}
\begin{aligned}
&\Vert\mathcal{F}^{-1}\Big(\psi_2(s,\eta,\xi)\widehat{D\mathcal{U}}(s,\eta,\xi)\Big)\Vert_{L^{q^\prime}}\lesssim s^{-\frac{n-1}{2}(\frac{1}{q}-\frac{1}{q^\prime})}\eta^{-1-\frac{n+1}{2}(\frac{1}{q}-\frac{1}{q^\prime})}\Big(\Vert\mathcal{U}_0\Vert_{L^q}+
\eta\Vert\mathcal{U}_1\Vert_{L^q}\Big)
\end{aligned}
\end{equation*}
and
\begin{equation*}
\begin{aligned}
&\Vert\mathcal{F}^{-1}\Big(\psi_2(s,\eta,\xi)\widehat{\mathcal{U}}(s,\eta,\xi)\Big)\Vert_{L^{q^\prime}}\\
&\lesssim s^{\max\{\text{Re}\rho-n(\frac{1}{q}-\frac{1}{q^\prime})+\chi_\epsilon,-\frac{n-1}{2}(\frac{1}{q}-\frac{1}{q^\prime})\}}\eta^{\min\{ -\text{Re}\rho-\chi_\epsilon,-\frac{n+1}{2}(\frac{1}{q}-\frac{1}{q^\prime})\}}\Big(\Vert\mathcal{U}_0\Vert_{L^q}+
\eta\Vert\mathcal{U}_1\Vert_{L^q}\Big),
\end{aligned}
\end{equation*}
where $\chi_\epsilon$ is defined by Theorem \ref{mainresult} and $\rho$ is defined by (\ref{rho}).

\end{proposition}

\begin{proof}
Note that $s\vert\xi\vert$ is bounded from below and $\eta\vert\xi\vert$ is bounded in $\Omega_2(s,\eta)$ so   by Lemma \ref{expressionandderivatives},  (\ref{propop1}), (\ref{bounded1}), (\ref{bounded2}) and  (\ref{bounded3}), we have
\begin{equation}\label{matrix2}
\Lambda_k^\iota(s,\eta,\xi)\sim
\left\{
\begin{aligned}
&
\begin{pmatrix}
c_{11}(\eta\vert\xi\vert)
&\sin (-s\vert\xi\vert)b_{12}(s\vert\xi\vert)\\
c_{21}(\eta\vert\xi\vert)&
\sin (s\vert\xi\vert)b_{22}(s\vert\xi\vert)
\end{pmatrix},&\quad&\text{if}\ \  \delta\neq0,\\
&
\begin{pmatrix}
\widetilde{c}_{11}(\eta\vert\xi\vert)
&e^{is\vert\xi\vert}\widetilde{b}_{12}(s\vert\xi\vert)\\
\widetilde{c}_{21}(\eta\vert\xi\vert)&
e^{-is\vert\xi\vert}\widetilde{b}_{22}(s\vert\xi\vert)
\end{pmatrix},&\quad&\text{if}\ \  \delta=0,
\end{aligned}
\right.
\end{equation}
where
\begin{align*}
&c_{11}(z)=(2iz)^{k-\rho}e^{iz}, &\ & \widetilde{c}_{11}(z)=(2iz)^{k-\frac{1}{2}}e^{iz}\log (-2iz),\\
&c_{21}(z)=(2iz)^{\rho-1+k}e^{-iz}, &\ & \widetilde{c}_{21}(z)=(2iz)^{k-\frac{1}{2}}e^{-iz}\log (2iz)
\end{align*}
and $b_{12}(z),b_{22}(z),\tilde{b}_{12}(z),\tilde{b}_{22}(z)$ are same as (\ref{matrix1}). Moreover, we can  verify that
\begin{equation}\label{prviousestimate}
\begin{aligned}
&\vert c_{11}(z)\vert\lesssim\vert z\vert^{k-\text{Re}\rho}, &\ &\vert c_{21}(z)\vert\lesssim \vert z\vert^{\text{Re}\rho-1+k}\big(\lesssim\vert z\vert^{k-\text{Re}\rho}\ \text{for small}\ \vert z\vert\big),\\
&\vert \widetilde{c}_{11}(z)\vert\lesssim \vert z\vert^{k-\frac{1}{2}}\vert\log z\vert=\vert z\vert^{k-\text{Re}\rho}\vert\log z\vert,&\ &\vert\widetilde{c}_{21}(z)\vert\lesssim \vert z\vert^{k-\frac{1}{2}}\vert\log z\vert=\vert z\vert^{k-\text{Re}\rho}\vert\log z\vert.
\end{aligned}
\end{equation}
By (\ref{bounded4}) we have
\begin{align}
&\vert\frac{d^n}{dz^n}b_{12}(z)\vert, \vert\frac{d^n}{dz^n}b_{22}(z)\vert, \vert\frac{d^n}{dz^n}\tilde{b}_{12}(z)\vert, \vert\frac{d^n}{dz^n}\tilde{b}_{22}(z)\vert\lesssim\vert z\vert^{-n},\ \text{for any}\ n\in\mathbb{N}.\label{estimate2df}
\end{align}
Then
\begin{equation}\label{express112}
\begin{aligned}
&\partial_s^\iota\widehat{H}_k(s,\eta,\xi)\sim\vert\xi\vert^{\iota-k}\det\big(\Lambda_k^\iota(s,\eta,\xi)\big)
\\ &\sim\vert\xi\vert^{\iota-k}
\left\{
\begin{aligned}
&\sin(s\vert\xi\vert)c_{11}(\eta\vert\xi\vert)b_{22}(s\vert\xi\vert)-
\sin(-s\vert\xi\vert)c_{21}(\eta\vert\xi\vert)b_{12}(s\vert\xi\vert), &\ & \text{if}\ \delta\neq0,\\
&e^{-is\vert\xi\vert}\widetilde{c}_{11}(\eta\vert\xi\vert)\widetilde{b}_{22}(s\vert\xi\vert)-
e^{is\vert\xi\vert}\widetilde{c}_{21}(\eta\vert\xi\vert)\widetilde{b}_{12}(s\vert\xi\vert),&\ &  \text{if}\ \delta=0.
\end{aligned}
\right.
\end{aligned}
\end{equation}
Since $s\vert\xi\vert\geq K$ in $\Omega_2(s,\eta,\xi)$, the dyadic decomposition applied in Proposition \ref{highfrequencyestimate} remains necessary. Note that $\psi_2(s,\eta,\xi)\tilde{\phi}_j(\frac{s\vert\xi\vert}{K})=0$  holds for $j<0$ and  $\sum\limits_{j=0}^\infty\phi_j\big(\frac{s\vert\xi\vert}{K}\big)=1$, combining the expressions of (\ref{nablau})-(\ref{partialsu}) and (\ref{express112}), our goal is to derive the uniform-in-$j$ estimates of the norm for the following  operators with $j\in\mathbb{N}$, i.e.,
\begin{align}
&\mathcal{F}^{-1}\Big(\sin(s\vert\xi\vert)c_{11}(\eta\vert\xi\vert)b_{22}(s\vert\xi\vert)\psi_2(s,\eta,\xi)
\phi_j(s,\xi)\vert\xi\vert^{1-k}\hat{f}(\xi)\Big),\label{case3op1}\\
&\mathcal{F}^{-1}\Big(\sin(s\vert\xi\vert)c_{21}(\eta\vert\xi\vert)b_{12}(s\vert\xi\vert)\psi_2(s,\eta,\xi)
\phi_j(s,\xi)\vert\xi\vert^{1-k}\hat{f}(\xi)\Big)\label{case3op2},\\
&\mathcal{F}^{-1}\Big(e^{-is\vert\xi\vert}\widetilde{c}_{11}(\eta\vert\xi\vert)\widetilde{b}_{22}(s\vert\xi\vert)  \psi_2(s,\eta,\xi)
\phi_j(s,\xi)\vert\xi\vert^{1-k}\hat{f}(\xi)\Big),\label{case3op3}\\
&\mathcal{F}^{-1}\Big(e^{is\vert\xi\vert}\widetilde{c}_{21}(\eta\vert\xi\vert)\widetilde{b}_{12}(s\vert\xi\vert) \psi_2(s,\eta,\xi)
\phi_j(s,\xi)\vert\xi\vert^{1-k}\hat{f}(\xi)\Big).\label{case3op4}
\end{align}

The $L^1-L^\infty$ estimate and $L^2-L^2$ estimate for (\ref{case3op3}) are derived first, while the remaining three terms (\ref{case3op1})-(\ref{case3op2}) and (\ref{case3op4}) can be handled analogously, yielding the same estimates.

By Young's inequality we have
\begin{align}
&\Vert\mathcal{F}^{-1}\Big(e^{-is\vert\xi\vert}\widetilde{c}_{11}(\eta\vert\xi\vert)\widetilde{b}_{22}(s\vert\xi\vert)  \psi_2(s,\eta,\xi)
\phi_j(s,\xi)\vert\xi\vert^{1-k}\hat{f}(\xi)\Big)\Vert_{L^\infty}\notag\\
&\leq \Vert\mathcal{F}^{-1}\Big(e^{-is\vert\xi\vert}\widetilde{c}_{11}(\eta\vert\xi\vert)\widetilde{b}_{22}(s\vert\xi\vert)  \psi(\frac{\eta\vert\xi\vert}{K})(1-\psi(\frac{s\vert\xi\vert}{K}))
\phi_j(s,\xi)\vert\xi\vert^{1-k}\Big)\Vert_{L^\infty}\Vert f\Vert_{L^1}\notag\\
&\leq\Vert\mathcal{F}^{-1}\Big(e^{-is\vert\xi\vert}\phi_j(s,\xi)(1-\psi(\frac{s\vert\xi\vert}{K}))\widetilde{b}_{22}(s\vert\xi\vert)
\vert\xi\vert^{1-\text{Re}\rho-\gamma}\Big)\Vert_{L^\infty}\notag\\
&\quad\quad\quad\quad\quad\quad\quad\quad\times\Vert\mathcal{F}^{-1}\Big(\vert\xi\vert^{\gamma-k+\text{Re}\rho}
\widetilde{c}_{11}(\eta\vert\xi\vert)\psi(\frac{\eta\vert\xi\vert}{K})\Big)\Vert_{L^1}\Vert f\Vert_{L^1},\label{bian1}
\end{align}
where $\gamma>0$ is determined later, which is introduced to eliminate the influence of the factor $\vert\log z\vert$. Using the change of variables ~$\zeta\doteq\eta\vert\xi\vert\leq2K$, $y\doteq\frac{x}{\eta}$ and applying~(\ref{prviousestimate}), we obtain from Proposition \ref{prop251012} that
\begin{align}
&\Vert\mathcal{F}^{-1}\Big(\vert\xi\vert^{\gamma-k+\text{Re}\rho}
\widetilde{c}_{11}(\eta\vert\xi\vert)\psi(\frac{\eta\vert\xi\vert}{K})\Big)\Vert_{L^1}
=\eta^{-\text{Re}\rho+k-\gamma}\Vert\mathcal{F}^{-1}\Big(\vert\zeta\vert^{\gamma-k+\text{Re}\rho}
\psi(\frac{\vert\zeta\vert}{K})\widetilde{c}_{11}(\vert\zeta\vert)\Big)\Vert_{L^1}\notag\\
&\lesssim \eta^{-\text{Re}\rho+k-\gamma} \Vert\mathcal{F}^{-1}\Big(\vert\zeta\vert^{\gamma}\psi(\frac{\vert\zeta\vert}{K})\big\vert\log\vert\zeta\vert\big\vert
\Big)\Vert_{L^1}\lesssim  \eta^{-\text{Re}\rho+k-\gamma} \Vert\mathcal{F}^{-1}\Big(\vert\zeta\vert^{\frac{\gamma}{2}}\psi(\frac{\vert\zeta\vert}{K})\Big)\Vert_{L^1}\lesssim  \eta^{-\text{Re}\rho+k-\gamma}.\label{bian2}
\end{align}
Moreover, through the  scaling~$\zeta\doteq2^{-j}\frac{\xi}{K}$, it follows from    Proposition \ref{propositionpare} that
\begin{align}
&\Vert\mathcal{F}^{-1}\Big(e^{-is\vert\xi\vert}\phi_j(s,\xi)(1-\psi(\frac{s\vert\xi\vert}{K}))\widetilde{b}_{22}(s\vert\xi\vert)
\vert\xi\vert^{1-\text{Re}\rho-\gamma}\Big)\Vert_{L^\infty}\notag\\
&=2^{j(n+1-\text{Re}\rho-\gamma)}s^{-n-1+\text{Re}\rho+\gamma}K^{1-\text{Re}\rho-\gamma+n}\notag\\
&\quad\quad\quad\quad\quad\times\Vert\mathcal{F}^{-1}\Big(e^{-i2^{j}K\vert\zeta\vert}
\widetilde{\phi}_0(\vert\zeta\vert)(1-\psi(2^j\vert\zeta\vert))\widetilde{b}_{22}(2^jK\vert\zeta\vert)
\vert\zeta\vert^{1-\text{Re}\rho-\gamma}\Big)(\frac{2^jK}{s}x)\Vert_{L^\infty_x}\notag\\
&\lesssim 2^{j(n+1-\text{Re}\rho-\gamma)}s^{-n-1+\text{Re}\rho+\gamma}(1+2^{j}K)^{-\frac{n-1}{2}}\notag\\
&\quad\quad\quad\quad\quad\times
\mathop{\sum}\limits_{\vert\alpha\vert\leq h}\Vert \partial_\zeta^\alpha\Big( \widetilde{\phi}_0(\vert\zeta\vert)(1-\psi(2^j\vert\zeta\vert))\widetilde{b}_{22}(2^jK\vert\zeta\vert)
\vert\zeta\vert^{1-\text{Re}\rho-\gamma}\Big)\Vert_{L^\infty},\label{bian4}
\end{align}
where $h>\frac{n+3}{2}$.  From  (\ref{estimate2df}) we have
\begin{align}
&\mathop{\sum}\limits_{\vert\alpha\vert\leq h}\Vert \partial_\zeta^\alpha\Big( \tilde{\phi}_0(\vert\zeta\vert)(1-\psi(2^jK\vert\zeta\vert))\widetilde{b}_{22}(2^jK\vert\zeta\vert)
\vert\zeta\vert^{1-\text{Re}\rho-\gamma}\Big)\Vert_{L^\infty}\notag\\
&=\sup\limits_{\frac{1}{2}\leq\vert\zeta\vert\leq2}\mathop{\sum}
\limits_{\vert\alpha_1\vert+\vert\alpha_2\vert+\vert\alpha_3\vert=\vert\alpha\vert\leq h}\big\vert\partial_\zeta^{\alpha_1}
\Big(\widetilde{\phi}_0(\vert\zeta\vert)(1-\psi(2^j\vert\zeta\vert))\Big)\big\vert\big\vert\partial_\zeta^{\alpha_2}
\vert\zeta\vert^{1-\text{Re}\rho-\gamma}\big\vert\big\vert\partial_\zeta^{\alpha_3}\widetilde{b}_{22}(2^jK
\vert\zeta\vert)\big\vert\notag\\
&\lesssim \sup\limits_{\frac{1}{2}\leq\vert\zeta\vert\leq2}\mathop{\sum}
\limits_{\vert\alpha_1\vert+\vert\alpha_2\vert+\vert\alpha_3\vert=\vert\alpha\vert\leq h}
\vert\zeta\vert^{1-\text{Re}\rho-\gamma-\vert\alpha_2\vert}(2^jK\vert\zeta\vert)^{-\vert\alpha_3\vert}\lesssim 1.
\end{align}
Hence, by combining (\ref{bian1})–(\ref{bian4}), we  arrive at
\begin{align}
&\Vert\mathcal{F}^{-1}\Big(e^{-is\vert\xi\vert}\widetilde{c}_{11}(\eta\vert\xi\vert)\widetilde{b}_{22}(s\vert\xi\vert)  \psi_2(s,\eta,\xi)
\phi_j(s,\xi)\vert\xi\vert^{1-k}\hat{f}(\xi)\Big)\Vert_{L^\infty}\notag\\
&\lesssim
2^{j(\frac{n+3}{2}-\text{Re}\rho-\gamma)}s^{-n-1+\text{Re}\rho+\gamma}\eta^{-\text{Re}\rho+k-\gamma}\Vert f\Vert_{L^1}.\label{l1linftycase3}
\end{align}

It follows from  Parseval's identity, (\ref{prviousestimate})-(\ref{estimate2df}) and the change of variables~$\zeta\doteq\frac{s\xi}{2^jK}$ that
\begin{align}
&\Vert\mathcal{F}^{-1}\Big(e^{-is\vert\xi\vert}\widetilde{c}_{11}(\eta\vert\xi\vert)\widetilde{b}_{22}(s\vert\xi\vert)  \psi_2(s,\eta,\xi)
\phi_j(s,\xi)\vert\xi\vert^{1-k}\hat{f}(\xi)\Big)\Vert_{L^2}\notag\\
&=\Vert e^{-is\vert\xi\vert}\widetilde{c}_{11}(\eta\vert\xi\vert)\widetilde{b}_{22}(s\vert\xi\vert)  \psi_2(s,\eta,\xi)
\phi_j(s,\xi)\vert\xi\vert^{1-k}\hat{f}(\xi) \Vert_{L^2}\notag\\
&\leq \Vert\big(\varphi(2^{-j}\frac{s\vert\xi\vert}{K})-\varphi(2^{-j+1}\frac{s\vert\xi\vert}{K})\big)(\eta\vert\xi\vert)^{-\text{Re}\rho+k-\gamma}
\vert\xi\vert^{1-k}(1-\psi(\frac{s\vert\xi\vert}{K}))\widetilde{b}_{22}(s\vert\xi\vert)\Vert_{L^\infty}\notag\\
&\quad\quad\quad\quad\quad\times
\Vert \psi(\frac{\eta\vert\xi\vert}{K})\widetilde{c}_{11}(\eta\vert\xi)(\eta\vert\xi\vert)^{\text{Re}\rho-k+\gamma}\Vert_{L^\infty}
\Vert f\Vert_{L^2}\notag\\
&=\eta^{-\text{Re}\rho+k-\gamma}\sup\limits_{\frac{1}{2}\leq\vert\zeta\vert\leq2}\big\vert\tilde{\phi}_0(\vert
\zeta\vert)(2^jK\vert\xi\vert s^{-1})^{1-\text{Re}\rho-\gamma}\big(1-\psi(2^j\vert\zeta\vert)\widetilde{b}_{22}(2^jK\vert\zeta\vert)\big)\big\vert\notag\\
&\quad\quad\quad\quad\quad\quad\quad\quad\quad\quad\times\Vert \psi(\frac{\eta\vert\xi\vert}{K})(\eta\vert\xi\vert)^{\gamma}\big\vert\log\eta\vert\xi\vert\big\vert\Vert_{L^\infty}\Vert f\Vert_{L^2}\notag\\
&\lesssim 2^{j(1-\text{Re}\rho-\gamma)}s^{\text{Re}\rho+\gamma-1}\eta^{-\text{Re}\rho+k-\gamma}\Vert f\Vert_{L^2},\label{l2l2linftycase3}
\end{align}
where in the last inequality we used  the boundedness  of $(\eta\vert\xi\vert)^{\gamma}\big\vert\log\eta\vert\xi\vert\big\vert$ as $\eta\vert\xi\vert\leq 2K$.

According to Riesz-Thorin interpolation theorem, from (\ref{l1linftycase3})-(\ref{l2l2linftycase3}) we get  \begin{equation*}
\begin{aligned}
&\Vert\mathcal{F}^{-1}\Big(e^{-is\vert\xi\vert}\widetilde{c}_{11}(\eta\vert\xi\vert)\widetilde{b}_{22}(s\vert\xi\vert)  \psi_2(s,\eta,\xi)
\phi_j(s,\xi)\vert\xi\vert^{1-k}\hat{f}(\xi)\Big)\Vert_{L^{q^\prime}} \\
&\lesssim 2^{j(\frac{n+1}{2}(\frac{1}{q}-\frac{1}{q^\prime})+1-\text{Re}\rho-\gamma)}s^{-n(\frac{1}{q}-\frac{1}{q^\prime})+\gamma+\text{Re}\rho-1}
\eta^{-\text{Re}\rho+k-\gamma}\Vert f\Vert_{L^q}.
\end{aligned}
\end{equation*}
Note that $\text{Re}\rho\in[\frac{1}{2},1)$,  so let $\gamma\doteq\frac{n+1}{2}(\frac{1}{q}-\frac{1}{q^\prime})+1-\text{Re}\rho>0$, then from Proposition \ref{littlewoodpaleythm},
\small{\begin{equation}\label{es20251012}
\begin{aligned}
&\Vert\mathcal{F}^{-1}\Big(e^{-is\vert\xi\vert}\widetilde{c}_{11}(\eta\vert\xi\vert)\widetilde{b}_{22}(s\vert\xi\vert)  \psi_2(s,\eta,\xi)
\vert\xi\vert^{1-k}\hat{f}(\xi)\Big)\Vert_{L^{q^\prime}} \lesssim s^{-\frac{n-1}{2}(\frac{1}{q}-\frac{1}{q^\prime})}
\eta^{k-1-\frac{n+1}{2}(\frac{1}{q}-\frac{1}{q^\prime})}\Vert f\Vert_{L^q}.
\end{aligned}
\end{equation}
}
Combining the definition (\ref{express112}) of $\partial_s^\iota\widehat{H}_k(s,\eta,\xi)$ and (\ref{es20251012}), we have
\begin{align}
&\Vert\mathcal{F}^{-1}\Big(\psi_2(s,\eta,\xi)
\vert\xi\vert^{1-\iota}\partial_s^\iota\widehat{H}_k(s,\eta,\xi)\hat{f}(\xi)\Big)\Vert_{L^{q^\prime}}\lesssim s^{-\frac{n-1}{2}(\frac{1}{q}-\frac{1}{q^\prime})}
\eta^{k-1-\frac{n+1}{2}(\frac{1}{q}-\frac{1}{q^\prime})}\Vert f\Vert_{L^q}.\label{lplqescase3}
\end{align}

Consequently, we conclude from  (\ref{lplqescase3}) that
\begin{align*}
&\Vert\mathcal{F}^{-1}\Big(\psi_2(s,\eta,\xi)\widehat{\nabla\mathcal{U}}(s,\xi)\Big)\Vert_{L^{q^\prime}}=
\Vert\mathcal{F}^{-1}\Big(\psi_2(s,\eta,\xi)\xi\big( \widehat{H}_0(s,\eta,\xi)\widehat{\mathcal{U}}_0(\xi)+\widehat{H}_1(s,\eta,\xi)
\widehat{\mathcal{U}}_1(\xi)\big)\Vert_{L^{q^\prime}}\\
&\lesssim s^{-\frac{n-1}{2}(\frac{1}{q}-\frac{1}{q^\prime})}\eta^{-1-\frac{n+1}{2}(\frac{1}{q}-\frac{1}{q^\prime})}\Big(\Vert\mathcal{U}_0\Vert_{L^q}+
\eta\Vert\mathcal{U}_1\Vert_{L^q}\Big)
\end{align*}
and
\begin{equation*}
\begin{aligned}
&\Vert\mathcal{F}^{-1}\Big(\psi_2(s,\eta,\xi)\partial_s\widehat{\mathcal{U}}(s,\xi)\Big)\Vert_{L^{q^\prime}}=
\Vert\mathcal{F}^{-1}\Big(\psi_2(s,\eta,\xi)\big( \partial_s\widehat{H}_0(s,\eta,\xi)\widehat{\mathcal{U}}_0(\xi)+\partial_s\widehat{H}_1(s,\eta,\xi)
\widehat{\mathcal{U}}_1(\xi)\big)\Vert_{L^{q^\prime}}\\
&\lesssim s^{-\frac{n-1}{2}(\frac{1}{q}-\frac{1}{q^\prime})}\eta^{-1-\frac{n+1}{2}(\frac{1}{q}-\frac{1}{q^\prime})}\Big(\Vert\mathcal{U}_0\Vert_{L^q}+
\eta\Vert\mathcal{U}_1\Vert_{L^q}\Big)
\end{aligned}
\end{equation*}

According to the representation (\ref{masseqsolution}) of $\mathcal{U}(s,\xi)$, the factor $\vert\xi\vert^{1-k}$ appearing in the operators (\ref{case3op1})-(\ref{case3op4}) is no longer applicable and thus needs to be modified. Consequently, in what follows, the estimate will again be illustrated using
$$\mathcal{F}^{-1}\Big(e^{-is\vert\xi\vert}\widetilde{c}_{11}(\eta\vert\xi\vert)\widetilde{b}_{22}(s\vert\xi\vert)  \psi_2(s,\eta,\xi)
\phi_j(s,\xi)\vert\xi\vert^{-k}\hat{f}(\xi)\Big).$$
Let $\gamma_1>0$ be a nonnegative number to be determined. Repeating the proof of (\ref{l1linftycase3}), it follows that
\begin{align}
&\Vert\mathcal{F}^{-1}\Big(e^{-is\vert\xi\vert}\widetilde{c}_{11}(\eta\vert\xi\vert)\widetilde{b}_{22}(s\vert\xi\vert)  \psi_2(s,\eta,\xi)
\phi_j(s,\xi)\vert\xi\vert^{-k}\hat{f}(\xi)\Big)\Vert_{L^\infty}\notag\\
&\lesssim \Vert\mathcal{F}^{-1}\Big( e^{-is\vert\xi\vert}\phi_j(s,\xi)(1-\psi(\frac{s\vert\xi\vert}{K}))\vert\xi\vert^{-\gamma_1-\text{Re}\rho}
\tilde{b}_{22}(s\vert\xi\vert)\Big)\Vert_{L^\infty} \Vert\mathcal{F}^{-1}\Big(\vert\xi\vert^{\gamma_1-k+\text{Re}\rho}\tilde{c}_{11}(\eta\vert\xi\vert)
\psi(\frac{\eta\vert\xi\vert}{K})\Big)\Vert_{L^1}\Vert f\Vert_{L^1}\notag\\
&\lesssim 2^{j(n-\gamma_1-\text{Re}\rho)}s^{-n+\gamma_1+\text{Re}\rho}(1+2^{j}K)^{-\frac{n-1}{2}}\eta^{k-\text{Re}\rho-\gamma_1}\Vert f\Vert_{L^1}\notag\\
&\lesssim 2^{j(\frac{n+1}{2}-\gamma_1-\text{Re}\rho)}s^{-n+\gamma_1+\text{Re}\rho}\eta^{k-\text{Re}\rho-\gamma_1}\Vert f\Vert_{L^1}.\label{riese1}
\end{align}
Repeating the argument of (\ref{l2l2linftycase3}) yields
\begin{align}
&\Vert\mathcal{F}^{-1}\Big(e^{-is\vert\xi\vert}\widetilde{c}_{11}(\eta\vert\xi\vert)\widetilde{b}_{22}(s\vert\xi\vert)  \psi_2(s,\eta,\xi)
\phi_j(s,\xi)\vert\xi\vert^{-k}\hat{f}(\xi)\Big)\Vert_{L^2}%\notag\\
%&\lesssim \Vert\big(\varphi(2^{-j}\frac{s\vert\xi\vert}{K})-\varphi(2^{-j+1}\frac{s\vert\xi\vert}{K})\big)(\eta\vert\xi\vert)^{-\text{Re}\rho+k-\gamma_1}
%\vert\xi\vert^{-k}(1-\psi(\frac{s\vert\xi\vert}{K}))\tilde{b}_{22}(s\vert\xi\vert)\Vert_{L^\infty}\notag\\
%&\quad\quad\quad\quad\quad\quad\times\Vert(\eta\vert\xi\vert)^{\text{Re}\rho-k+\gamma_1}\psi(\frac{\eta\vert\xi\vert}{K})
%\tilde{c}_{11}(\eta\vert\xi\vert)\Vert_{L^\infty}\Vert f\Vert_{L^2}\notag\\
\lesssim 2^{j(-\text{Re}\rho-\gamma_1)}s^{\text{Re}\rho+\gamma_1}\eta^{-\text{Re}\rho+k-\gamma_1}\Vert f\Vert_{L^2}.\label{riese2}
\end{align}

By applying the Riesz-Thorin interpolation theorem to (\ref{riese1}) and (\ref{riese2}), we  deduce that
\begin{equation*}
\begin{aligned}
&\Vert\mathcal{F}^{-1}\Big(e^{-is\vert\xi\vert}\widetilde{c}_{11}(\eta\vert\xi\vert)\widetilde{b}_{22}(s\vert\xi\vert)  \psi_2(s,\eta,\xi)
\phi_j(s,\xi)\vert\xi\vert^{-k}\hat{f}(\xi)\Big)\Vert_{L^{q^\prime}}\lesssim  2^{j(\frac{n+1}{2}(\frac{1}{q}-\frac{1}{q^\prime})-\text{Re}\rho-\gamma_1)}s^{\gamma_1+\text{Re}\rho-n(\frac{1}{q}
-\frac{1}{q^\prime})}\eta^{k-\text{Re}\rho-\gamma_1}.
\end{aligned}
\end{equation*}
Since $\frac{1}{q}-\frac{1}{q^\prime}\in[0,1)$  for $q\in(1,2]$ and $\text{Re}\rho\in[\frac{1}{2},1)$, which prevents us from ensuring the nonnegativity of $\frac{n+1}{2}(\frac{1}{q}-\frac{1}{q^\prime})-\text{Re}\rho$.
For a sufficiently small $\epsilon>0$, we accordingly select
$
\gamma_1\doteq\max\{\chi_\epsilon,\frac{n+1}{2}(\frac{1}{q}-\frac{1}{q^\prime})-\text{Re}\rho\},
$
where $\chi_\epsilon$ is defined by Theorem \ref{mainresult}. Hence we have
\begin{equation*}
\begin{aligned}
&\Vert\mathcal{F}^{-1}\Big(e^{-is\vert\xi\vert}\widetilde{c}_{11}(\eta\vert\xi\vert)\widetilde{b}_{22}(s\vert\xi\vert)  \psi_2(s,\eta,\xi)
\phi_j(s,\xi)\vert\xi\vert^{-k}\hat{f}(\xi)\Big)\Vert_{L^{q^\prime}}\\
&\lesssim  s^{\max\{\text{Re}\rho-n(\frac{1}{q}-\frac{1}{q^\prime})+\chi_\epsilon,-\frac{n-1}{2}(\frac{1}{q}-\frac{1}{q^\prime})\}}\eta^{\min\{ k-\text{Re}\rho-\chi_\epsilon,k-\frac{n+1}{2}(\frac{1}{q}-\frac{1}{q^\prime})\}}\Vert f\Vert_{L^q}.
\end{aligned}
\end{equation*}
Then follows from Proposition \ref{littlewoodpaleythm} that
\begin{equation*}
\begin{aligned}
&\Vert\mathcal{F}^{-1}\Big(e^{-is\vert\xi\vert}\widetilde{c}_{11}(\eta\vert\xi\vert)\widetilde{b}_{22}(s\vert\xi\vert)  \psi_2(s,\eta,\xi)
\vert\xi\vert^{-k}\hat{f}(\xi)\Big)\Vert_{L^{q^\prime}}\\
&\lesssim  s^{\max\{\text{Re}\rho-n(\frac{1}{q}-\frac{1}{q^\prime})+\chi_\epsilon,-\frac{n-1}{2}(\frac{1}{q}-\frac{1}{q^\prime})\}}\eta^{\min\{ k-\text{Re}\rho-\chi_\epsilon,k-\frac{n+1}{2}(\frac{1}{q}-\frac{1}{q^\prime})\}}\Vert f\Vert_{L^q},
\end{aligned}
\end{equation*}
which gives
\begin{equation*}
\begin{aligned}
&\Vert\mathcal{F}^{-1}\Big(\vert\xi\vert^{-\iota}\partial_s^\iota\widehat{H}_k(s,\eta,\xi)\hat{f}(\xi)\Big)\Vert_{L^{q^\prime}}\lesssim  s^{\max\{\text{Re}\rho-n(\frac{1}{q}-\frac{1}{q^\prime})+\chi_\epsilon,-\frac{n-1}{2}(\frac{1}{q}-\frac{1}{q^\prime})\}}\eta^{\min\{ k-\text{Re}\rho-\chi_\epsilon,k-\frac{n+1}{2}(\frac{1}{q}-\frac{1}{q^\prime})\}}\Vert f\Vert_{L^q}.
\end{aligned}
\end{equation*}
Consequently,
\begin{equation*}
\begin{aligned}
&\Vert\mathcal{F}^{-1}\Big(\psi_2(s,\eta,\xi)\widehat{\mathcal{U}}(s,\xi)\Big)\Vert_{L^{q^\prime}}=
\Vert\mathcal{F}^{-1}\Big(\psi_2(s,\eta,\xi)\big(\widehat{H}_0(s,\eta,\xi)\widehat{\mathcal{U}}_0(\xi)+
\widehat{H}_1(s,\eta,\xi)\widehat{\mathcal{U}}_1(\xi)\big)\Big)\Vert_{L^{q^\prime}}\\
&\lesssim s^{\max\{\text{Re}\rho-n(\frac{1}{q}-\frac{1}{q^\prime})+\chi_\epsilon,-\frac{n-1}{2}(\frac{1}{q}-\frac{1}{q^\prime})\}}\eta^{\min\{ -\text{Re}\rho-\chi_\epsilon,-\frac{n+1}{2}(\frac{1}{q}-\frac{1}{q^\prime})\}}\big(\Vert\mathcal{U}_0\Vert_{L^q}+
\eta\Vert\mathcal{U}_1\Vert_{L^q}\big).
\end{aligned}
\end{equation*}

\end{proof}

\subsection{The proof of Theorems \ref{propuwithtau=1}-\ref{estimateu0=0}} \label{section3.4}
To control the linear part of the solution to (\ref{eq}), it is necessary to derive estimates for the solution to (\ref{lineareq}) at the initial time $\tau=1$.  In order to apply the Duhamel's principle to control the nonlinear part of the solution to (\ref{eq}),  the  estimates for the solution to  (\ref{lineareq}) with
$u_0=0$ are also required.  Correspondingly, based on the estimates for the solution $\mathcal{U}(s,x)$ to (\ref{masseq}) established in Sections \ref{propositionlow}-\ref{propositionmedium}, we first derive the  estimates for the solution to (\ref{masseq}) at the initial time $\eta=\frac{1}{m+1}$.
\begin{proposition}\label{propmathcalu}
Let  $\delta<\min\{(m+1)^2,1\}$  and  $\mathcal{U}(s,x)$ be the solution to (\ref{masseq}) with $\eta=\frac{1}{m+1}$,  and the initial data $(\mathcal{U}_0,\mathcal{U}_1)\in H^{r+1}_q\times H^r_q$. Then for any $s\geq\frac{1}{m+1}$ we have
\begin{align*}
&\Vert\mathcal{U}(s,\cdot)\Vert_{L^{q^\prime}}\lesssim a_0(s)(1+\log(m+1)s)^{\chi_{\delta}}\Big(\Vert\mathcal{U}_0\Vert_{H^r_q}+\Vert\mathcal{U}_1\Vert_{H^{[r-1]_+}_q}\Big),\\
&\Vert D\mathcal{U}(s,\cdot)\Vert_{L^{q^\prime}}\lesssim a_1(s)(1+\log(m+1)s)^{\chi_{\delta}}\Big(\Vert\mathcal{U}_0\Vert_{H^{r+1}_q}+\Vert\mathcal{U}_1\Vert_{H^{r}_q}\Big),
\end{align*}
where $1<q\leq2$ with $\frac{1}{q}+\frac{1}{q^\prime}=1$,  $r=n(\frac{1}{q}-\frac{1}{q^\prime})\geq0$, and for sufficiently small $\epsilon>0$,
\begin{equation*}
\begin{aligned}
a_0(s)&=s^{\max\{\text{Re}\rho-n(\frac{1}{q}-\frac{1}{q^\prime})+\chi_\epsilon,-\frac{n-1}{2}(\frac{1}{q}-\frac{1}{q^\prime})\}},\\
a_1(s)&=s^{-\frac{n-1}{2}(\frac{1}{q}-\frac{1}{q^\prime})}.
\end{aligned}
\end{equation*}
Here $\rho, \chi_\delta,\chi_\epsilon$ are defined by Lemma \ref{sulutionfou}, Theorem \ref{propuwithtau=1} and Theorem \ref{mainresult},  respectively.
\end{proposition}

%\begin{equation*}
%a_0(s)=\left\{
%\begin{aligned}
%&s^{\max\{\frac{1}{2}+\frac{\sqrt{\delta}}{2(m+1)}-n(\frac{1}{p}-\frac{1}{q}),-\frac{n-1}{2}(\frac{1}{p}-\frac{1}{q})\}},&\quad&0<\delta<\max\{(m+1)^2,1\}\\
%&s^{\max\{\frac{1}{2}+\epsilon-n(\frac{1}{p}-\frac{1}{q}),-\frac{n-1}{2}(\frac{1}{p}-\frac{1}{q})\}},&\quad&\delta=0,\\
%&s^{\max\{\frac{1}{2}-n(\frac{1}{p}-\frac{1}{q}),-\frac{n-1}{2}(\frac{1}{p}-\frac{1}{q})\}},&\quad&\delta<0.
%\end{aligned}
%\right.
%\end{equation*}

\begin{proof}
The estimate for $\Vert D\mathcal{U}(s,\cdot)\Vert_{L^{q^\prime}}$  is  presented as an illustration here,  and the estimate for $\mathcal{U}(s,\cdot)$ can be derived analogously.  Note $\text{Re}\rho\in[\frac{1}{2},1)$,  we have
\begin{equation*}
\begin{aligned}
&s^{\text{Re}\rho-1-n(\frac{1}{q}-\frac{1}{q^\prime})}\eta^{-\text{Re}\rho}=s^{-\frac{n-1}{2}
(\frac{1}{q}-\frac{1}{q^\prime})}s^{\text{Re}\rho-1-\frac{n+1}{2}(\frac{1}{q}-\frac{1}{q^\prime})}\eta^{-\text{Re}\rho}\\
&\leq s^{-\frac{n-1}{2}(\frac{1}{q}-\frac{1}{q^\prime})}\eta^{-1-\frac{n+1}{2}(\frac{1}{q}-\frac{1}{q^\prime})}=
(m+1)^{1+\frac{n+1}{2}(\frac{1}{q}-\frac{1}{q^\prime})}s^{-\frac{n-1}{2}(\frac{1}{q}-\frac{1}{q^\prime})}
\end{aligned}
\end{equation*}
and $
s^{\text{Re}\rho-1-n(\frac{1}{p}-\frac{1}{q})}\eta^{1-\text{Re}\rho}\leq
(m+1)^{\frac{n+1}{2}(\frac{1}{p}-\frac{1}{q})}s^{-\frac{n-1}{2}(\frac{1}{p}-\frac{1}{q})}.
$
Then from Propositions \ref{lowfre}-\ref{propmedium},
\begin{equation*}
\begin{aligned}
&\Vert D\mathcal{U}(s,\cdot)\Vert_{L^{q^\prime}}\\
&\leq \Vert\mathcal{F}^{-1}\Big(\psi_1(s,\eta,\xi)\widehat{D\mathcal{U}}(s,\xi)\Big)\Vert_{L^{q^\prime}}
+\Vert\mathcal{F}^{-1}\Big(\psi_2(s,\eta,\xi)\widehat{D\mathcal{U}}(s,\xi)\Big)\Vert_{L^{q^\prime}}
+\Vert\mathcal{F}^{-1}\Big(\psi_3(s,\eta,\xi)\widehat{D\mathcal{U}}(s,\xi)\Big)\Vert_{L^{q^\prime}}\\
&\lesssim s^{\text{Re}\rho-1-n(\frac{1}{q}-\frac{1}{q'})}\eta^{-\text{Re}\rho}(1+\log(m+1)s)^{\chi_{\delta}}\Big(\Vert\mathcal{U}_0\Vert_{L^q}
+\eta\Vert\mathcal{U}_1\Vert_{L^q}\Big)\\
&\ +s^{-\frac{n-1}{2}(\frac{1}{q}-\frac{1}{q^\prime})}\eta^{\frac{n-1}{2}(\frac{1}{q}-\frac{1}{q^\prime})}
\Big(\Vert\mathcal{U}_0\Vert_{\dot{H}^{r+1}_q}
+\eta\Vert\mathcal{U}_1\Vert_{\dot{H}^{r}_q}\Big)+ s^{-\frac{n-1}{2}(\frac{1}{q}-\frac{1}{q^\prime})}\eta^{-1-\frac{n+1}{2}(\frac{1}{q}-\frac{1}{q^\prime})}
\Big(\Vert\mathcal{U}_0\Vert_{L^q}
+\eta\Vert\mathcal{U}_1\Vert_{L^q}\Big)\\
&\lesssim s^{-\frac{n-1}{2}(\frac{1}{q}-\frac{1}{q^\prime})} (1+\log(m+1)s)^{\chi_{\delta}} \Big(\Vert\mathcal{U}_0\Vert_{H^{r+1}_q}+
\Vert\mathcal{U}_1\Vert_{H^{r}_q}\Big).
\end{aligned}
\end{equation*}
\end{proof}

\begin{proposition}\label{propositionu0=0}
Let  $\delta<\min\{(m+1)^2,1\}$  and  $\mathcal{U}(s,x)$ be the solution to (\ref{masseq}),   the initial data $\mathcal{U}_0=0$ and $\mathcal{U}_1\in H^{r}_q$, where $r=n(\frac{1}{q}-\frac{1}{q^\prime})$ with $1<q\leq 2$ and $\frac{1}{q}+\frac{1}{q^\prime}=1$. Then for any $s\geq\eta\geq\frac{1}{m+1}$ we have
\begin{align*}
&\Vert D\mathcal{U}(s,\cdot)\Vert_{L^{q^\prime}}\lesssim a_1(s)a_1^{-1}(\eta)(1+\log\frac{s}{\eta})^{\chi_\delta}\Big(\Vert\mathcal{U}_1\Vert_{\dot{H}^r_q}+
\eta^{-r}\Vert\mathcal{U}_1\Vert_{L^q}\Big)
\end{align*}
and
\begin{equation*}
\Vert\mathcal{U}(s,\cdot)\Vert_{L^{q^\prime}}\lesssim\left\{
\begin{aligned}
&a_0(s)a_0^{-1}(\eta)(1+\log\frac{s}{\eta})^{\chi_\delta}\Big(\Vert\mathcal{U}_1\Vert_{\dot{H}^{r-1}_q}+
\eta^{1-r}\Vert\mathcal{U}_1\Vert_{L^q}\Big), &\quad&\text{if}\ r\geq1,\\
&a_0(s)a_0^{-1}(\eta)(1+\log\frac{s}{\eta})^{\chi_\delta}
\eta^{1-r}\Vert\mathcal{U}_1\Vert_{L^q},&\quad&\text{if}\ 0\leq r\leq1,
\end{aligned}
\right.
\end{equation*}
where $a_0(s),a_1(s),\chi_\delta$ are the same as Proposition \ref{propmathcalu}.
\end{proposition}

\begin{proof}
We restrict attention to the case $0<r<1$, for which the estimate of $\Vert\mathcal{U}(s,\cdot)\Vert_{L^{q^\prime}}$ is derived. The remaining cases can be obtained in a similar manner. Note $r=n(\frac{1}{q}-\frac{1}{q^\prime})$, we have
\begin{equation*}
\begin{aligned}
s^{\text{Re}\rho-n(\frac{1}{q}-\frac{1}{q^\prime})}\eta^{1-\text{Re}\rho}\leq a_0(s)a_0^{-1}(\eta)\eta^{1-r}%\\
%\iff& s^{\text{Re}\rho-n(\frac{1}{q}-\frac{1}{q^\prime})}\eta^{1-\text{Re}\rho} \leq s^{\max\{\text{Re}\rho-n(\frac{1}{q}-\frac{1}{q^\prime}),-\frac{n-1}{2}(\frac{1}{q}-\frac{1}{q^\prime})\}}\eta^{1-r+\min\{-\text{Re}\rho+
%n(\frac{1}{q}-\frac{1}{q^\prime}),\frac{n-1}{2}(\frac{1}{q}-\frac{1}{q^\prime})\}}\\
\iff \eta^{\max\{\chi_\epsilon,-\text{Re}\rho+\frac{n+1}{2}(\frac{1}{q}-\frac{1}{q^\prime})\}}\leq s^{\max\{\chi_\epsilon,-\text{Re}\rho+\frac{n+1}{2}(\frac{1}{q}-\frac{1}{q^\prime})\}},
\end{aligned}
\end{equation*}
\begin{equation*}
\begin{aligned}
s^{-\frac{n-1}{2}(\frac{1}{q}-\frac{1}{q^\prime})}\eta^{\frac{n-1}{2}(\frac{1}{q}-\frac{1}{q^\prime})}\leq a_0(s)a_0^{-1}(\eta)
\iff \eta^{\max\{0,\text{Re}\rho-\frac{n+1}{2}(\frac{1}{q}-\frac{1}{q^\prime})+\chi_\epsilon\}}\leq
s^{\max\{0,\text{Re}\rho-\frac{n+1}{2}(\frac{1}{q}-\frac{1}{q^\prime})+\chi_\epsilon\}}
\end{aligned}
\end{equation*}
and
\begin{equation*}
\begin{aligned}
&s^{\max\{\text{Re}\rho-n(\frac{1}{q}-\frac{1}{q^\prime})+\chi_\epsilon,-\frac{n-1}{2}(\frac{1}{q}-\frac{1}{q^\prime})\}}
\eta^{1+\min\{-\text{Re}\rho-\chi_\epsilon,-\frac{n+1}{2}(\frac{1}{q}-\frac{1}{q^\prime})\}}
=a_0(s)a_0^{-1}(\eta)\eta^{1-r}.
\end{aligned}
\end{equation*}
Since $s\geq\eta$, the above inequality follows immediately. Then  by Propositions \ref{lowfre}-\ref{propmedium} we have
\begin{equation*}
\begin{aligned}
&\Vert\mathcal{U}(s,\cdot)\Vert_{L^{q^\prime}}\\
&\leq \Vert\mathcal{F}^{-1}\Big(\psi_1(s,\eta,\xi)\widehat{\mathcal{U}}(s,\xi)\Big)\Vert_{L^{q^\prime}}+
\Vert\mathcal{F}^{-1}\Big(\psi_2(s,\eta,\xi)\widehat{\mathcal{U}}(s,\xi)\Big)\Vert_{L^{q^\prime}}+
\Vert\mathcal{F}^{-1}\Big(\psi_3(s,\eta,\xi)\widehat{\mathcal{U}}(s,\xi)\Big)\Vert_{L^{q^\prime}}\\
&\lesssim s^{\text{Re}\rho-n(\frac{1}{q}-\frac{1}{q^\prime})}\eta^{1-\text{Re}\rho}(1+\log \frac{s}{\eta})^{\chi_\delta}\Vert\mathcal{U}_1\Vert_{L^q}+
s^{-\frac{n-1}{2}(\frac{1}{q}-\frac{1}{q^\prime})}\eta^{\frac{n-1}{2}(\frac{1}{q}-\frac{1}{q^\prime})+1-r}
\Vert\mathcal{U}_1\Vert_{L^q}\\
&\quad\quad+s^{\max\{\text{Re}\rho-n(\frac{1}{q}-\frac{1}{q^\prime})+\chi_\epsilon,-\frac{n-1}{2}(\frac{1}{q}-\frac{1}{q^\prime})\}}
\eta^{1+\min\{-\text{Re}\rho-\chi_\epsilon,-\frac{n+1}{2}(\frac{1}{q}-\frac{1}{q^\prime})\}}
\Vert\mathcal{U}_1\Vert_{L^q}\\
&\lesssim a_0(s)a_0^{-1}(\eta)(1+\log\frac{s}{\eta})\eta^{1-r}\Vert\mathcal{U}_1\Vert_{L^q}.
\end{aligned}
\end{equation*}

\end{proof}

As the closing part of this section, and relying on Propositions  \ref{propmathcalu} and \ref{propositionu0=0}, Theorems \ref{propuwithtau=1} and \ref{estimateu0=0} are derived through successive applications of the variable transformations  (\ref{sub11}), (\ref{sub21}), and (\ref{sub31}).

\begin{proof}[Proof of Theorem \ref{propuwithtau=1}]
Substituting
\begin{align*}
&v(s,x)=s^{-\frac{m}{2(m+1)}}\mathcal{U}(s,x),\ \ \nabla v(s,x)=s^{-\frac{m}{2(m+1)}}\nabla\mathcal{U}(s,x),\\
&\partial_sv(s,x)=-\frac{m}{2(m+1)}s^{-\frac{m}{2(m+1)}-1}\mathcal{U}(s,x)+s^{-\frac{m}{2(m+1)}}\partial_s\mathcal{U}(s,x)
\end{align*}
and
\begin{align*}
\mathcal{U}_0(x)&=\eta^{\frac{m}{2(m+1)}}v_0(x)=(m+1)^{-\frac{m}{2(m+1)}}v_0(x),\\
\mathcal{U}_1(x)&=\frac{m}{2(m+1)}
\eta^{\frac{m}{2(m+1)}-1}v_0(x)+\eta^{\frac{m}{2(m+1)}}v_1(x)\\
&=\frac{m}{2(m+1)}(m+1)^{1-\frac{m}{2(m+1)}}v_0(x)+(m+1)^{-\frac{m}{2(m+1)}}v_1(x)
\end{align*}
into Proposition \ref{propmathcalu}, we conclude that
\begin{align*}
\Vert v(s,\cdot)\Vert_{L^{q^\prime}}&\lesssim s^{-\frac{m}{2(m+1)}}a_0(s)(1+\log(m+1)s)^{\chi_\delta}\Big(\Vert v_0\Vert_{H^{r}_q}+\Vert v_1\Vert_{H^{[r-1]_+}_q}\Big),\\
\Vert \nabla v(s,\cdot)\Vert_{L^{q^\prime}}&\lesssim s^{-\frac{m}{2(m+1)}}a_1(s)(1+\log(m+1)s)^{\chi_\delta}\Big(\Vert v_0\Vert_{H^{r+1}_q}+\Vert v_1\Vert_{H^{r}_q}\Big),\\
\Vert \partial_sv(s,\cdot)\Vert_{L^{q^\prime}}&\lesssim s^{-\frac{m}{2(m+1)}-1}a_0(s)(1+\log(m+1)s)^{\chi_\delta}\Big(\Vert v_0\Vert_{H^{r}_q}+\Vert v_1\Vert_{H^{[r-1]_+}_q}\Big)\\
&\quad+ s^{-\frac{m}{2(m+1)}}a_1(s)(1+\log(m+1)s)^{\chi_\delta}\Big(\Vert v_0\Vert_{H^{r+1}_q}+\Vert v_1\Vert_{H^{r}_q}\Big).
\end{align*}
Note $s\geq\eta=\frac{1}{m+1}$ and $\text{Re}\rho<1$,  we can choose a  sufficiently small $\epsilon>0$ such that $-1+\text{Re}\rho-\frac{n+1}{2}(\frac{1}{q}-\frac{1}{q'})+\chi_\epsilon>0$, then   it follows that
\begin{align*}
s^{-1}a_0(s)&=s^{\max\{-1+\text{Re}\rho-n(\frac{1}{q}-\frac{1}{q^\prime})+\chi_\epsilon,
-1-\frac{n-1}{2}(\frac{1}{q}-\frac{1}{q^\prime})\}}=s^{-\frac{n-1}{2}(\frac{1}{q}-\frac{1}{q^\prime})}s^{\max\{-1+\text{Re}\rho-\frac{n+1}{2}(\frac{1}{q}-\frac{1}{q^\prime})+
\chi_\epsilon,-1\}}\\
&\leq (m+1)^{\min\{1-\text{Re}\rho+\frac{n+1}{2}(\frac{1}{q}-\frac{1}{q^\prime})-
\chi_\epsilon,1\}}s^{-\frac{n-1}{2}(\frac{1}{q}-\frac{1}{q^\prime})}.
\end{align*}
Hence,
\begin{equation}\label{transvw}
\begin{aligned}
\Vert v(s,\cdot)\Vert_{L^{q^\prime}}&\lesssim b_0(s)(1+\log(m+1)s)^{\chi_\delta}\Big(\Vert v_0\Vert_{H^{r}_q}+\Vert v_1\Vert_{H^{[r-1]_+}_q}\Big),\\
\Vert D v(s,\cdot)\Vert_{L^{q^\prime}}&\lesssim  b_1(s)(1+\log(m+1)s)^{\chi_\delta}\Big(\Vert v_0\Vert_{H^{r+1}_q}+\Vert v_1\Vert_{H^{r}_q}\Big),
\end{aligned}
\end{equation}
where $b_0(s)=s^{-\frac{m}{2(m+1)}}a_0(s), b_1(s)=s^{-\frac{m}{2(m+1)}}a_1(s)$.

When $\eta=\frac{1}{m+1}$, we have  $v_0(x)=w_0(x)$, $v_1(x)=w_1(x)$ and
\begin{align*}
&\partial_sv(s,x)=\big((m+1)s\big)^{-\frac{m}{m+1}}\partial_sw\big(((m+1)s)^{\frac{1}{m+1}},x\big),\ \ \nabla v(s,x)=\nabla w\big(((m+1)s)^{\frac{1}{m+1}},x\big).
\end{align*}
From (\ref{transvw}),  we get
\begin{equation*}
\begin{aligned}
&\Vert w\big(((m+1)s)^{\frac{1}{m+1}},\cdot\big)\Vert_{L^{q^\prime}}\lesssim b_0(s)(1+\log(m+1)s)^{\chi_\delta}\big(\Vert w_0\Vert_{H^r_q}+\Vert w_1\Vert_{H^{[r-1]_+}_q}\big),\\
&\Vert \nabla w\big(((m+1)s)^{\frac{1}{m+1}},\cdot\big)\Vert_{L^{q^\prime}}\lesssim b_1(s)(1+\log(m+1)s)^{\chi_\delta}\big(\Vert w_0\Vert_{H^{r+1}_q}+\Vert w_1\Vert_{H^r_q}\big),\\
&\Vert \partial_s w\big(((m+1)s)^{\frac{1}{m+1}},\cdot\big)\Vert_{L^{q^\prime}}\lesssim s^{\frac{m}{m+1}}b_1(s)(1+\log(m+1)s)^{\chi_\delta}\big(\Vert w_0\Vert_{H^{r+1}_q}+\Vert w_1\Vert_{H^r_q}\big).
\end{aligned}
\end{equation*}
Since  $t=((m+1)s)^{\frac{1}{m+1}}$,  it follows that
\begin{align}
&\Vert w(t,\cdot)\Vert_{L^{q^\prime}}\lesssim t^{-\frac{m}{2}}\widetilde{b}_0(t)(1+\log t )^{\chi_\delta}\big(\Vert w_0\Vert_{H^r_q}+\Vert w_1\Vert_{H^{[r-1]_+}_q}\big),\label{transforwu1}\\
&\Vert \nabla w(t,\cdot)\Vert_{L^{q^\prime}}\lesssim t^{-\frac{m}{2}}\widetilde{b}_1(t)(1+\log t)^{\chi_\delta}\big(\Vert w_0\Vert_{H^{r+1}_q}+\Vert w_1\Vert_{H^r_q}\big),\label{transforwu2}\\
&\Vert \partial_tw(t,\cdot)\Vert_{L^{q^\prime}}\lesssim t^{\frac{m}{2}}\widetilde{b}_1(t)(1+\log t)^{\chi_\delta}\big(\Vert w_0\Vert_{H^{r+1}_q}+\Vert w_1\Vert_{H^r_q}\big),\label{transforwu3}
\end{align}
where
\begin{equation}\label{transforwu}
\begin{aligned}
&\widetilde{b}_0(t)=t^{\max\{(m+1)\text{Re}\rho-(m+1)n(\frac{1}{q}-\frac{1}{q^\prime})+\chi_\epsilon,
-\frac{(m+1)(n-1)}{2}(\frac{1}{q}-\frac{1}{q^\prime})\}},\\
&\widetilde{b}_1(t)=t^{-\frac{(m+1)(n-1)}{2}(\frac{1}{q}-\frac{1}{q^\prime})}.
\end{aligned}
\end{equation}

To conclude, from
\begin{align*}
&u(t,x)=t^{-\frac{\mu}{2}}w(t,x), \ \ \nabla u(t,x)=t^{-\frac{\mu}{2}}w(t,x),\\
&\partial_tu(t,x)=-\frac{\mu}{2}t^{-\frac{\mu}{2}-1}w(t,x)+t^{-\frac{\mu}{2}}\partial_tw(t,x),\\
&w_0(x)=u_0(x), \ \ w_1(x)=u_1(x)+\frac{\mu}{2}u_0(x)
\end{align*}
and (\ref{transforwu1})-(\ref{transforwu3}), we have
\begin{align*}
&\Vert u(t,\cdot)\Vert_{L^{q^\prime}}\lesssim t^{-\frac{\mu+m}{2}}\widetilde{b}_0(t)(1+\log t)^{\chi_\delta}\big(\Vert u_0\Vert_{H^{r}_q}+\Vert u_1\Vert_{H^{[r-1]_+}_q}\big)\\
&\Vert \nabla u(t,\cdot)\Vert_{L^{q^\prime}}\lesssim t^{-\frac{\mu+m}{2}}\widetilde{b}_1(t)(1+\log t)^{\chi_\delta}\big(\Vert u_0\Vert_{H^{r+1}_q}+\Vert u_1\Vert_{H^{r}_q}\big),\\
&\Vert \partial_t u(t,\cdot)\Vert_{L^{q^\prime}}\lesssim t^{-\frac{\mu+m}{2}-1}\widetilde{b}_0(t)(1+\log t)^{\chi_\delta}\big(\Vert u_0\Vert_{H^{r}_q}+\Vert u_1\Vert_{H^{[r-1]_+}_q}\big)\\
&\quad\quad\quad\quad\quad\quad\quad+t^{-\frac{\mu-m}{2}}\widetilde{b}_1(t)(1+\log t)^{\chi_\delta}\big(\Vert u_0\Vert_{H^{r+1}_q}+\Vert u_1\Vert_{H^{r}_q}\big)\\
&\quad\ \quad\quad\quad\quad\lesssim t^{-\frac{\mu-m}{2}}\widetilde{b}_1(t)(1+\log t)^{\chi_\delta}\big(\Vert u_0\Vert_{H^{r+1}_q}+\Vert u_1\Vert_{H^{r}_q}\big),
\end{align*}
where in the last inequality we use the fact that $t^{-1-\frac{m}{2}}\widetilde{b}_0(t)\leq t^{\frac{m}{2}}\widetilde{b}_1(t)$. In fact, in view of $\text{Re}\rho<1$, $\max\{(m+1)(\text{Re}\rho-1)-\frac{(m+1)(n+1)}{2}(\frac{1}{q}-\frac{1}{q'})+\chi_\epsilon,-(m+1)\}<0$ holds for any $m>-1$ and sufficiently small $\epsilon$,  then
\begin{align*}
&t^{\max\{(m+1)(\text{Re}\rho-1)-\frac{(m+1)(n+1)}{2}(\frac{1}{q}-\frac{1}{q'})+\chi_\epsilon,-(m+1)\}}\leq1\\
\implies&t^{\max\{(m+1)\text{Re}\rho-(m+1)n(\frac{1}{p}-\frac{1}{q})+\chi_\epsilon,
-\frac{(m+1)(n-1)}{2}(\frac{1}{q}-\frac{1}{q'})\}}\leq t^{m+1-\frac{(m+1)(n-1)}{2}(\frac{1}{q}-\frac{1}{q'})}\\
\implies&t^{-1-\frac{m}{2}}\tilde{b}_0(t)\leq t^{\frac{m}{2}}\tilde{b}_1(t).
\end{align*}

Therefore, the proof of Theorem   \ref{propuwithtau=1} is concluded.
\end{proof}

\begin{proof}[Proof of Theorem \ref{estimateu0=0}]
With the notations introduced in the proof of Theorem \ref{propuwithtau=1}, we begin by stating  a fact: $b_0(s)b_0^{-1}(\eta)s^{-1}\eta\leq b_1(s)b_1^{-1}(\eta)$. In fact,
\begin{align*}
&b_0(s)b_0^{-1}(\eta)s^{-1}\eta\leq b_1(s)b_1^{-1}(\eta)\iff a_0(s)a_1^{-1}(s)s^{-1}\leq a_0(\eta)a_1^{-1}(\eta)s\eta^{-1}\\
\iff&s^{\max\{\text{Re}\rho-1-\frac{n+1}{2}(\frac{1}{q}-\frac{1}{q'})+\chi_\epsilon,-1\}}\leq \eta^{\max\{\text{Re}\rho-1-\frac{n+1}{2}(\frac{1}{q}-\frac{1}{q'})+\chi_\epsilon,-1\}},
\end{align*}
the last inequality holds as $ \max\{\text{Re}\rho-1-\frac{n+1}{2}(\frac{1}{p}-\frac{1}{q})+\chi_\epsilon,-1\}<0$ under the condition $\text{Re}\rho<1$ and for sufficiently small $\epsilon$. When $r\geq1$, by  substituting
\begin{align*}
&v(s,x)=s^{-\frac{m}{2(m+1)}}\mathcal{U}(s,x),\\
&\nabla v(s,x)=s^{-\frac{m}{2(m+1)}}\nabla\mathcal{U}(s,x),\\
&\partial_sv(s,x)=-\frac{m}{2(m+1)}s^{-\frac{m}{2(m+1)}-1}\mathcal{U}(s,x)+s^{-\frac{m}{2(m+1)}}\partial_s\mathcal{U}(s,x)
\end{align*}
and
$\mathcal{U}_1(x)=\eta^{\frac{m}{2(m+1)}}v_1(x)$
into Proposition \ref{propositionu0=0},  we have
\begin{align*}
\Vert\partial_sv(s,\cdot)\Vert_{L^{q^\prime}}&\lesssim b_0(s)b_0^{-1}(\eta)s^{-1}\eta(1+\log\frac{s}{\eta})^{\chi_\delta}\big(\eta^{-1}\Vert v_1\Vert_{\dot{H}^{r-1}_q}+\eta^{-r}\Vert v_1\Vert_{L^q}\big)\\
&\quad\quad+ b_1(s)b_1^{-1}(\eta)(1+\log\frac{s}{\eta})^{\chi_\delta}\big(\Vert v_1\Vert_{\dot{H}^{r}_q}+\eta^{-r}\Vert v_1\Vert_{L^q}\big)\\
&\lesssim b_1(s)b_1^{-1}(\eta)(1+\log\frac{s}{\eta})^{\chi_\delta}\big(\Vert v_1\Vert_{\dot{H}^{r}_q}+\eta^{-1}\Vert v_1\Vert_{\dot{H}^{r-1}_q}+\eta^{-r}\Vert v_1\Vert_{L^q}\big)
\end{align*}
and similarly,
\begin{align*}
\Vert\partial_sv(s,\cdot)\Vert_{L^{q^\prime}}&\lesssim b_1(s)b_1^{-1}(\eta)(1+\log\frac{s}{\eta})^{\chi_\delta}\big(
\Vert v_1\Vert_{\dot{H}^{r}_q}+\eta^{-r}\Vert v_1\Vert_{L^q}\big),&\quad&\text{if}\ 0\leq r<1,\\
\Vert v(s,\cdot)\Vert_{L^{q^\prime}}&\lesssim b_0(s)b^{-1}_0(\eta)(1+\log\frac{s}{\eta})^{\chi_\delta}\big(\Vert v_1\Vert_{\dot{H}^{r-1}_q}+\eta^{1-r}\Vert v_1\Vert_{L^q}\big),&\quad& \text{if}\ r\geq1,\\
\Vert v(s,\cdot)\Vert_{L^{q^\prime}}&\lesssim b_0(s)b^{-1}_0(\eta)(1+\log\frac{s}{\eta})^{\chi_\delta}\eta^{1-r}\Vert v_1\Vert_{L^q},&\quad& \text{if}\ 0\leq r<1,\\
\Vert\nabla v(s,\cdot)\Vert_{L^{q^\prime}}&\lesssim b_1(s)b_1^{-1}(\eta)(1+\log\frac{s}{\eta})^{\chi_\delta}\big(
\Vert v_1\Vert_{\dot{H}^{r}_q}+\eta^{-r}\Vert v_1\Vert_{L^q}\big),&\quad& \text{for any}\ r\geq1.
\end{align*}
In addition, by interpolation theorem (Proposition \ref{interpolationtheorem}) and Young's inequality, for $r\geq1$ we have
\begin{align*}
\eta^{-1}\Vert v_1\Vert_{\dot{H}^{r-1}_q}&\lesssim\eta^{-1}\Vert v_1\Vert_{L^q}^{\frac{1}{r}}\Vert v_1\Vert_{\dot{H}^r_q}^{1-\frac{1}{r}}=\big(\eta^{-r}\Vert v_1\Vert_{L^q}\big)^{\frac{1}{r}}\Vert v_1\Vert_{\dot{H}^r_q}^{1-\frac{1}{r}}\\
&\leq \frac{1}{r}\big(\eta^{-r}\Vert v_1\Vert_{L^q}\big)^{\frac{1}{r}r}+\frac{r-1}{r}\big(\Vert v_1\Vert_{\dot{H}^r_q}^{1-\frac{1}{r}}\big)^{\frac{r}{r-1}}\lesssim \Vert v_1\Vert_{\dot{H}^r_q}+\eta^{-r}\Vert v_1\Vert_{L^q},
\end{align*}
which implies that
\begin{align*}
\Vert\partial_sv(s,\cdot)\Vert_{L^{q^\prime}}\lesssim b_1(s)b_1^{-1}(\eta)(1+\log\frac{s}{\eta})^{\chi_\delta}\big(\Vert v_1\Vert_{\dot{H}^{r}_q}+\eta^{-r}\Vert v_1\Vert_{L^q}\big)
\end{align*}
holds for any $r\geq0.$  By performing the scaling transformation   (\ref{sub21}) and $v_1(x)=((m+1)\eta)^{-\frac{m}{m+1}}w_1(x)$, we arrive at
\begin{align*}
&\Vert\partial_tw(t,\cdot)\Vert_{L^{q^\prime}}\lesssim t^{\frac{m}{2}}\widetilde{b}_1(t)\widetilde{b}_1^{-1}(\tau)\tau^{-\frac{m}{2}}(1+\log\frac{t}{\tau})^{\chi_\delta}
\big(\Vert w_1\Vert_{\dot{H}^r_q}+\tau^{-(m+1)r}\Vert w_1\Vert_{L^q}\big),\  \text{for any}\ r\geq0,\\
&\Vert\nabla w(t,\cdot)\Vert_{L^{q^\prime}}\lesssim t^{-\frac{m}{2}}\widetilde{b}_1(t)\widetilde{b}_1^{-1}(\tau)\tau^{-\frac{m}{2}}(1+\log\frac{t}{\tau})^{\chi_\delta}
\big(\Vert w_1\Vert_{\dot{H}^r_q}+\tau^{-(m+1)r}\Vert w_1\Vert_{L^q}\big),\  \text{for any}\ r\geq0,\\
&\Vert w(t,\cdot)\Vert_{L^{q^\prime}}\lesssim t^{-\frac{m}{2}}\widetilde{b}_0(t)\widetilde{b}_0^{-1}(\tau)(1+\log\frac{t}{\tau})^{\chi_\delta}\tau^{-\frac{m}{2}}\big(\Vert w_1\Vert_{\dot{H}^{r-1}_q}+\tau^{(m+1)(1-r)}\Vert w_1\Vert_{L^q}\big),\ \text{if}\ r\geq1,\\
&\Vert w(t,\cdot)\Vert_{L^{q^\prime}}\lesssim t^{-\frac{m}{2}}\widetilde{b}_0(t)\widetilde{b}_0^{-1}(\tau)
(1+\log\frac{t}{\tau})^{\chi_\delta}\tau^{-\frac{m}{2}+(m+1)(1-r)}\Vert w_1\Vert_{L^q},\ \text{if}\ 0\leq r<1,
\end{align*}
where $\widetilde{b}_0(t), \widetilde{b}_1(t)$ are defined by  (\ref{transforwu}).

Under the condition $\text{Re}\rho<1$ and for sufficiently small $\epsilon>0$, we have
\begin{align*}
&t^{\max\{(m+1)(\text{Re}\rho-1)-\frac{(m+1)(n-1)}{2}(\frac{1}{q}-\frac{1}{q'})+\chi_\epsilon,-(m+1)\}}\leq \tau^{\max\{(m+1)(\text{Re}\rho-1)-\frac{(m+1)(n+1)}{2}(\frac{1}{q}-\frac{1}{q'})+\chi_\epsilon,-(m+1)\}}\\
&\iff t^{-(m+1)}\widetilde{b}_0(t)\widetilde{b}_1^{-1}(t)\leq\tau^{-(m+1)}\widetilde{b}_0(\tau)\widetilde{b}_1^{-1}(\tau)
\iff t^{-\frac{m}{2}-1}\widetilde{b}_0(t)\widetilde{b}^{-1}_0(\tau)\tau^{1+\frac{m}{2}}\leq t^{\frac{m}{2}}\widetilde{b}_1(t)\widetilde{b}^{-1}_1(\tau)\tau^{-\frac{m}{2}}.
\end{align*}
It follows from
$u(t,x)=t^{-\frac{\mu}{2}}w(t,x)$, $\partial_tu(t,x)=-\frac{\mu}{2}t^{-\frac{\mu}{2}-1}w(t,x)+t^{-\frac{\mu}{2}}\partial_tw(t,x)$,
$w_1(x)=\tau^{\frac{\mu}{2}}u_1(x)$ and $r\geq1$ that
\begin{align*}
&\Vert\partial_tu(t,\cdot)\Vert_{L^{q^\prime}}\lesssim t^{-\frac{\mu}{2}-1}\Vert w(t,\cdot)\Vert_{L^{q^\prime}}+t^{-\frac{\mu}{2}}\Vert\partial_tw(t,\cdot)\Vert_{L^{q^\prime}}\\
&\lesssim t^{-1-\frac{\mu+m}{2}}\widetilde{b}_0(t)\widetilde{b}^{-1}_0(\tau)\tau^{\frac{\mu-m}{2}}(1+\log\frac{t}{\tau})^{\chi_\delta}
\big(\Vert u_1\Vert_{\dot{H}^{r-1}_q}+\tau^{(m+1)(1-r)}\Vert u_1\Vert_{L^q}\big)\\
&\quad+t^{\frac{m-\mu}{2}}\widetilde{b}_1(t)\widetilde{b}^{-1}_1(\tau)\tau^{\frac{\mu-m}{2}}(1+\log\frac{t}{\tau})^{\chi_\delta}
\big(\Vert u_1\Vert_{\dot{H}^r_q}+\tau^{-(m+1)r}\Vert u_1\Vert_{L^q}\big)\\
&\lesssim t^{\frac{m-\mu}{2}}\widetilde{b}_1(t)\widetilde{b}_1^{-1}(\tau)\tau^{\frac{\mu-m}{2}}(1+\log\frac{t}{\tau})^{\chi_\delta}\big(
\Vert u_1\Vert_{\dot{H}^{r}_q}+\tau^{-(m+1)}\Vert u_1\Vert_{\dot{H}^{r-1}_q}+\tau^{-(m+1)r}\Vert u_1\Vert_{L^q}\big).
\end{align*}
Similarly,
\begin{align*}
&\Vert\partial_tu(t,\cdot)\Vert_{L^{q^\prime}}\lesssim t^{\frac{m-\mu}{2}}\widetilde{b}_1(t)\widetilde{b}_1^{-1}(\tau)\tau^{\frac{\mu-m}{2}}(1+\log\frac{t}{\tau})^{\chi_\delta}\big(
\Vert u_1\Vert_{\dot{H}^{r}_q}+\tau^{-(m+1)r}\Vert u_1\Vert_{L^q}\big),&\ &\text{if}\ 0\leq r<1, \\
&\Vert\nabla u(t,\cdot)\Vert_{L^{q^\prime}}\lesssim t^{-\frac{\mu+m}{2}}\widetilde{b}_1(t)\widetilde{b}^{-1}_1(\tau)\tau^{\frac{\mu-m}{2}}(1+\log\frac{t}{\tau})^{\chi_\delta}
\big(\Vert u_1\Vert_{\dot{H}^{r}_q}+\tau^{-(m+1)r}\Vert u_1\Vert_{L^q}\big),&\ &\text{for any}\ r\geq0,\\
&\Vert u(t,\cdot)\Vert_{L^{q^\prime}}\lesssim t^{-\frac{\mu+m}{2}}\widetilde{b}_0(t)\widetilde{b}_0^{-1}(\tau)\tau^{\frac{\mu-m}{2}}(1+\log\frac{t}{\tau})^{\chi_\delta}
\big(\Vert u_1\Vert_{\dot{H}^{r-1}_q}+\tau^{(m+1)(1-r)}\Vert u_1\Vert_{L^q}\big)&\ &\text{if}\ r\geq1,\\
&\Vert u(t,\cdot)\Vert_{L^{q^\prime}}\lesssim t^{-\frac{\mu+m}{2}}\widetilde{b}_0(t)\widetilde{b}_0^{-1}(\tau)\tau^{\frac{\mu-m}{2}}(1+\log\frac{t}{\tau})^{\chi_\delta}
\tau^{(m+1)(1-r)}\Vert u_1\Vert_{L^q},&\ &\text{if}\ 0\leq r<1.
\end{align*}
A further application of the interpolation theorem and Young's inequality yields
\begin{align*}
\tau^{-(m+1)}\Vert u_1\Vert_{\dot{H}^{r-1}_q}&\lesssim \tau^{-(m+1)}\Vert u_1\Vert_{L^q}^{\frac{1}{r}}
\Vert u_1\Vert_{\dot{H}^r_q}^{1-\frac{1}{r}}=\Vert u_1\Vert_{\dot{H}^r_q}^{1-\frac{1}{r}}\big(\tau^{-(m+1)r}\Vert u_1\Vert_{L^q}\big)^{\frac{1}{r}}\\
&\lesssim \frac{r-1}{r}\Vert u_1\Vert_{\dot{H}^r_q}+\frac{1}{r}\tau^{-(m+1)r}\Vert u_1\Vert_{L^q},
\end{align*}
which gives
\begin{align*}
&\Vert\partial_tu(t,\cdot)\Vert_{L^{q^\prime}}\lesssim t^{\frac{m-\mu}{2}}\widetilde{b}_1(t)\widetilde{b}_1^{-1}(\tau)\tau^{\frac{\mu-m}{2}}(1+\log\frac{t}{\tau})^{\chi_\delta}\big(
\Vert u_1\Vert_{\dot{H}^{r}_q}+\tau^{-(m+1)r}\Vert u_1\Vert_{L^q}\big)
\end{align*}
holds for any $r\geq0$.

Thus the proof of Theorem \ref{estimateu0=0} is now complete.
\end{proof}

\section{Global existence of solutions to semilinear  EPDT equation}\label{sectionproofmainresult}

Relying on the two estimates for the linear EPDT  equation, namely, Theorems \ref{propuwithtau=1}, \ref{estimateu0=0}, the global existence of solutions to (\ref{eq}) is established through the use of the Duhamel's principle and the contraction mapping principle. In view of the fact that Theorem \ref{mainresult} addresses the case
$\delta\leq0$, we make use of the following two corollaries, which are immediate consequences of Theorems \ref{propuwithtau=1}, \ref{estimateu0=0}.

\begin{corollary}\label{energyspacethm1}
Let us consider $\mu, \nu^2\geq0$  such that $\delta=(\mu-1)^2-4\nu^2\leq0$.  Let $u$  be the solution to the linear equation (\ref{lineareq}) with $\tau=1$,  the initial data $(u_0,u_1)\in H^{r_\kappa+1}_{1+\frac{1}{\kappa}}\times H^{r_\kappa}_{1+\frac{1}{\kappa}}$ for some $\kappa\geq1$,  and $r_\kappa=n(1-\frac{2}{\kappa+1})$. Then $u$ satisfies
\begin{align*}
&\Vert \nabla u(t,\cdot)\Vert_{L^2}\lesssim t^{-\frac{m+\mu}{2}}(1+\log t)^{\chi_\delta}\big(\Vert u_0\Vert_{H^1}+\Vert u_1\Vert_{L^2}\big),\\
&\Vert\partial_tu(t,\cdot)\Vert_{L^2}\lesssim t^{\frac{m-\mu}{2}}(1+\log t)^{\chi_\delta}\big(\Vert u_0\Vert_{H^1}+\Vert u_1\Vert_{L^2}\big)\\
&\Vert u(t,\cdot)\Vert_{L^{\kappa+1}}\lesssim t^{-\frac{m+\mu}{2}}A_\kappa(t)(1+\log t)^{\chi_\delta}
\big(\Vert u_0\Vert_{H^{r_\kappa}_{1+\frac{1}{\kappa}}}+\Vert u_1\Vert_{H^{[r_\kappa-1]_+}_{1+\frac{1}{\kappa}}}\big),
\end{align*}
where $A_\kappa(t)$ is defined by  Theorem \ref{mainresult}.

\end{corollary}

\begin{corollary}\label{energyspacethm2}
Let us consider $\mu, \nu^2\geq0$  such that $\delta=(\mu-1)^2-4\nu^2\leq0$.  Let $u$  be the solution to the linear equation (\ref{lineareq}), the initial data $u_0=0$, $u_1\in L^{1+\frac{1}{\kappa}}$ for some $\kappa\geq1$.  Assume $r_k=n(1-\frac{2}{\kappa+1})\leq1$, then
$u$ satisfies
\begin{align*}
&\Vert \nabla u(t,\cdot)\Vert_{L^2}\lesssim t^{-\frac{m+\mu}{2}}\tau^{\frac{\mu-m}{2}}(1+\log\frac{t}{\tau})^{\chi_\delta}\Vert u_1\Vert_{L^2},\\
&\Vert\partial_tu(t,\cdot)\Vert_{L^2}\lesssim t^{\frac{m-\mu}{2}}\tau^{\frac{\mu-m}{2}}(1+\log\frac{t}{\tau})^{\chi_\delta}\Vert u_1\Vert_{L^2}\\
&\Vert u(t,\cdot)\Vert_{L^{\kappa+1}}\lesssim t^{-\frac{m+\mu}{2}}\tau^{\frac{\mu-m}{2}}A_\kappa(t)A_\kappa^{-1}(\tau)
\tau^{(m+1)(1-r_\kappa)}(1+\log\frac{t}{\tau})^{\chi_\delta}\Vert u_1\Vert_{L^{1+\frac{1}{\kappa}}},
\end{align*}
where $A_\kappa(t)$ is defined by  Theorem \ref{mainresult}.
\end{corollary}

For any $T>1$, we introduce the function
\begin{align*}
\mathcal{X}(T):=C\big([1,T];H^1\cap L^{p+1}\big)\cap C^1\big([1,T]; L^2\big)
\end{align*}
equipped  with the norm
\begin{equation}\label{normmathcalx}
\begin{aligned}
&\Vert u(t,\cdot)\Vert_{\mathcal{X}(T)}=\sup\limits_{t\in[1,T]}\Big\{t^{\frac{\mu}{2}}(1+\log t)^{-\chi_\delta}\Big(
t^{\frac{m}{2}}\Vert\nabla u(t,\cdot)\Vert_{L^2}+t^{-\frac{m}{2}}\Vert\partial_tu(t,\cdot)\Vert_{L^2}\\
&\quad\quad\quad\quad\quad\quad\quad+t^{\frac{m}{2}}A_p(t)^{-1}\Vert u(t,\cdot)\Vert_{L^{p+1}}+t^{-\frac{m}{2}-1}\Vert u(t,\cdot)\Vert_{L^2}\Big)\Big\},
\end{aligned}
\end{equation}
where $A_p(t)$ is defined by Theorem \ref{mainresult}.
Based on the integral representation of the solution to (\ref{eq}) given by  (\ref{expressionultx}), we define the operator $\mathcal{N}$ as
\begin{equation}\label{mathcalmudef}
\mathcal{N}u(t,x)=u^l(t,x)+u^n(t,x),
\end{equation}
where $u^l(t,x)$ and $u^n(t,x)$ are defined by (\ref{ultxdef}) and (\ref{untxdef}), respectively.
Consider the subset of $\mathcal{X}(T)$
\begin{equation*}
X(T,M)=\big\{u\in\mathcal{X}(T): \Vert u\Vert_{\mathcal{X}(T)}\leq M\big\},
\end{equation*}
then we have

\begin{proposition}\label{keypropositionmass}
Under the conditions of Theorem \ref{mainresult}, there exists a $C>0$ such that for any $T>1$ and any $u,\tilde{u}\in\mathcal{X}(T)$,
\begin{align}
&\Vert \mathcal{N}u\Vert_{\mathcal{X}(T)}\leq C\big(\Vert (u_0,u_1)\Vert_{\mathcal{A}_p}+\Vert u\Vert_{\mathcal{X}(T)}^p\big),\label{prokey1}\\
&\Vert \mathcal{N}u-\mathcal{N}\tilde{u}\Vert_{\mathcal{X}(T)}\leq C\Vert u-\tilde{u}\Vert_{\mathcal{X}(T)}\big(\Vert u\Vert_{\mathcal{X}(T)}^{p-1}+\Vert \tilde{u}\Vert_{\mathcal{X}(T)}^{p-1}\big),\label{prokey2}
\end{align}
where $\mathcal{A}_p$ is defined by (\ref{inditaldataspace}).
\end{proposition}

The proof of Proposition \ref{keypropositionmass} is based on Corollaries
\ref{energyspacethm1}, \ref{energyspacethm2}, and will be presented in detail in Section \ref{proofofproposition}.  For now, we employ Proposition \ref{keypropositionmass}  to establish the existence of global solutions.

\subsection{The proof of Theorem \ref{mainresult}}

Choose $M=3C\Vert(u_0,u_1)\Vert_{\mathcal{A}_p}$ in $X(T,M)$, where  $C$ is the constant mentioned in Proposition \ref{keypropositionmass}.   Then for any $u\in \mathcal{X}(T,M)$,  as long as $\Vert(u_0,u_1)\Vert_{\mathcal{A}_p}\leq\bigl(\frac{2}{3^pC^p}\bigr)^{\frac{1}{p-1}}$, it holds that
\begin{equation*}
\begin{aligned}
\Vert \mathcal{N}u\Vert_{\mathcal{X}(T)}&\leq C\Vert(u_0,u_1)\Vert_{\mathcal{A}_p}+C\Vert u\Vert_{\mathcal{X}(T)}^p
\leq\bigl(C+3^pC^{p+1}\Vert(u_0,u_1)\Vert_{\mathcal{A}_p}^{p-1}\bigr)\Vert(u_0,u_1)\Vert_{\mathcal{A}_p}
\leq M.
\end{aligned}
\end{equation*}
Moreover, provided $\Vert(u_0,u_1)\Vert_{D^\sigma}\leq\bigl(\frac{1}{4C^p3^{p-1}}\bigr)^{\frac{1}{p-1}}$, we have
\begin{equation*}
\begin{aligned}
\Vert \mathcal{N}u-\mathcal{N}\tilde{u}\Vert_{\mathcal{X}(T)}&\leq C \Vert u-\tilde{u}\Vert_{\mathcal{X}(T)}\bigl(\Vert u\Vert_{\mathcal{X}(T)}^{p-1}+\Vert \tilde{u}\Vert_{\mathcal{X}(T)}^{p-1}\bigr)\\
                                        &\leq 2C\Vert u-\tilde{u}\Vert_{\mathcal{X}(T)}\bigl(3C\Vert(u_0,u_1)\Vert_{\mathcal{A}_p}\bigr)^{p-1}\leq \frac{1}{2}\Vert u-\tilde{u}\Vert_{\mathcal{X}(T)}.
\end{aligned}
\end{equation*}
Hence, if $\Vert(u_0,u_1)\Vert_{\mathcal{A}_p}\leq\varepsilon\doteq \bigl(\frac{1}{4C^p3^{p-1}}\bigr)^{\frac{1}{p-1}}$, we can claim that $\mathcal{N}$ is a contraction mapping from
$X(T,M)$ into itself.
According to the  contraction mapping principle, there exists a unique $u\in X(T,M)$ such that $\mathcal{N}u=u$,  which means  $u$ is the unique solution to (\ref{eq}).  Furthermore,  $u$ is  actually  the global solution since the constant $C$ don't depend on the choice of $T$. And $\Vert u\Vert_{\mathcal{X}(T)}\leq 3C\Vert (u_0,u_1)\Vert_{\mathcal{A}_p}$, which gives the estimates (\ref{thmes1})-(\ref{thmes4}).

\subsection{The proof of Proposition \ref{keypropositionmass}}\label{proofofproposition}

\begin{proof}[The proof of (\ref{prokey1})]

In view of the expression (\ref{mathcalmudef}) of $\mathcal{N}u(t,x)$, the first aim is to estimate $\Vert u^l\Vert_{\mathcal{X}(T)}$.
Since $\epsilon$ is sufficiently small and $A_1(t)=t^{\frac{m+1}{2}+\chi_\epsilon}$, it holds that $t^{-\frac{m+\mu}{2}}A_1(t)\leq t^{\frac{m-\mu}{2}+1}$. The condition (\ref{assumptionp}) ensures that $r_p\leq1$, by Corollary \ref{energyspacethm1}, the following estimates are valid
\begin{align*}
&\Vert\nabla u^l(t,\cdot)\Vert_{L^2}\lesssim t^{-\frac{m+\mu}{2}}(1+\log t)^{\chi_\delta}\big(\Vert u_0\Vert_{H^1}+\Vert u_1\Vert_{L^2}\big),\\
&\Vert\partial_tu^l(t,\cdot)\Vert_{L^2}\lesssim t^{\frac{m-\mu}{2}}(1+\log t)^{\chi_\delta}\big(\Vert u_0\Vert_{H^1}+\Vert u_1\Vert_{L^2}\big),\\
&\Vert u^l(t,\cdot)\Vert_{L^2}\lesssim t^{-\frac{m+\mu}{2}}A_1(t)(1+\log t)^{\chi_\delta}\big(\Vert u_0\Vert_{L^2}+\Vert u_1\Vert_{L^2}\big)\lesssim t^{\frac{m-\mu}{2}+1}(1+\log t)^{\chi_\delta}\big(\Vert u_0\Vert_{L^2}+\Vert u_1\Vert_{L^2}\big),\\
&\Vert u^l(t,\cdot)\Vert_{L^{p+1}}\lesssim t^{-\frac{m+\mu}{2}}A_p(t)(1+\log t)^{\chi_\delta}\big(\Vert u_0\Vert_{H^{r_p}_{1+\frac{1}{p}}}+\Vert u_1\Vert_{L^{1+\frac{1}{p}}}\big)\\
&\quad\quad\quad\quad\quad\ \lesssim  t^{-\frac{m+\mu}{2}}A_p(t)(1+\log t)^{\chi_\delta}\big(\Vert u_0\Vert_{W^{1,1+\frac{1}{p}}}+\Vert u_1\Vert_{L^{1+\frac{1}{p}}}\big).
\end{align*}
Then by the definition (\ref{normmathcalx}) of  the norm of  $\mathcal{X}(T)$, we have
\begin{equation}\label{ulmathcalX}
\begin{aligned}
&\Vert u^l\Vert_{\mathcal{X}(T)}\lesssim \Vert(u_0,u_1)\Vert_{\mathcal{A}_p}.
\end{aligned}
\end{equation}

It follows from  Gagliardo-Nirenberg inequality and the definition (\ref{normmathcalx}) of the norm of $\mathcal{X}(T)$ that
\begin{align}
&\Vert\vert u(\tau,\cdot)\vert^p\Vert_{L^2}=\Vert u(\tau,\cdot)\Vert_{L^{2p}}^p\lesssim \Vert \nabla u(\tau,\cdot)\Vert_{L^2}^{\theta p}\Vert u(\tau,\cdot)\Vert_{L^{p+1}}^{(1-\theta)p}\notag\\
&\quad\quad\quad\quad\quad\quad\quad\quad\quad\quad\quad\quad\lesssim \tau^{-\frac{m+\mu}{2}p}A_p(\tau)^{(1-\theta)p}(1+\log\tau)^{\chi_\delta p}\Vert u\Vert_{\mathcal{X}(T)}^p,\label{gnineq}\\
&\Vert\vert u(\tau,\cdot)\vert^p\Vert_{L^{1+\frac{1}{p}}}=\Vert u(\tau,\cdot)\Vert_{L^{p+1}}^p\lesssim \tau^{-\frac{m+\mu}{2}p}A_p(\tau)^p(1+\log\tau)^{\chi_\delta p}\Vert u\Vert_{\mathcal{X}(T)}^p,\label{gnineq1}
\end{align}
where $\theta\in[0,1]$  solves
\begin{align}
\frac{1}{2p}=(\frac{1}{2}-\frac{1}{n})\theta+\frac{1-\theta}{p+1}\label{thetaeq}
\end{align}
and $A_p(\tau)$ is given by Corollary \ref{energyspacethm1}. It should be noted that, under condition (\ref{assumptionp}), such a $\theta\in[0,1]$ exists.  According to  Duhamel's principle, by using   Corollary \ref{energyspacethm2} and  (\ref{gnineq}), we have
\begin{equation}\label{nablaunl2estimate}
\begin{aligned}
&\Vert \nabla u^n(t,\cdot)\Vert_{L^2}\lesssim t^{-\frac{m+\mu}{2}}\int_1^t\tau^{\frac{\mu-m}{2}}(1+\log\frac{t}{\tau})^{\chi_\delta}\Vert\vert u(\tau,\cdot)\vert^p\Vert_{L^2}d\tau\\
&\lesssim t^{-\frac{m+\mu}{2}}(1+\log t)^{\chi_\delta}\Vert u\Vert_{\mathcal{X}(T)}^p\int_1^t\tau^{-\frac{m}{2}(p+1)-\frac{\mu}{2}(p-1)}A_p(\tau)^{(1-\theta)p}
(1+\log\tau)^{\chi_\delta p}d\tau.
\end{aligned}
\end{equation}
Similarly,
\begin{align}
&\Vert\partial_tu^n(t,\cdot)\Vert_{L^2}\lesssim t^{\frac{m-\mu}{2}}(1+\log t)^{\chi_\delta}\Vert u\Vert_{\mathcal{X}(T)}^p\int_1^t\tau^{-\frac{m}{2}(p+1)-\frac{\mu}{2}(p-1)}A_p(\tau)^{(1-\theta)p}(1+\log\tau)^{\chi_\delta p}d\tau,\label{partialtunl2estimate}\\
&\Vert u^n(t,\cdot)\Vert_{L^{p+1}}\lesssim t^{-\frac{m+\mu}{2}}A_p(t)(1+\log t)^{\chi_\delta}\Vert
u\Vert_{\mathcal{X}(T)}^p\notag\\
&\quad\quad\quad\quad\quad\quad\quad\quad\quad\quad\times\int_1^t\tau^{-\frac{m}{2}(p+1)-\frac{\mu}{2}(p-1)+(m+1)\big(1-n(1-\frac{2}{p+1})\big)}
A_p(\tau)^{p-1}(1+\log\tau)^{\chi_\delta p}d\tau.\label{unlp+1estimate}
\end{align}
Note $A_1(t)=t^{\frac{m+1}{2}+\chi_\epsilon}$ and $\epsilon$ is sufficiently small, which yields
$\tau^{m+1}A_1(\tau)^{-1}=\tau^{\frac{m+1}{2}-\chi_\epsilon}\leq t^{\frac{m+1}{2}-\chi_\epsilon}.$
Hence
\begin{align}
&\Vert u^n(t,\cdot)\Vert_{L^2}\lesssim t^{-\frac{m+\mu}{2}}A_1(t)(1+\log t)^{\chi_\delta}\Vert u\Vert_{\mathcal{X}(T)}^p\notag\\
&\quad\quad\quad\quad\quad\quad\times\int_1^t\tau^{-\frac{m}{2}(p+1)-\frac{\mu}{2}(p-1)+m+1}A_1(\tau)^{-1}
A_p(\tau)^{(1-\theta)p}(1+\log\tau)^{\chi_\delta p}
d\tau\notag\\
&\lesssim t^{\frac{m-\mu}{2}+1}(1+\log t)^{\chi_\delta}\Vert u\Vert_{\mathcal{X}(T)}^p\int_1^t\tau^{-\frac{m}{2}(p+1)-\frac{\mu}{2}(p-1)}A_p(\tau)^{(1-\theta)p}(1+\log\tau)^{\chi_\delta p}
d\tau.\label{unl2estimate}
\end{align}
Since the integral expressions in (\ref{nablaunl2estimate})-(\ref{unl2estimate}) involve different powers of $\tau$, we need to compare the magnitudes of the corresponding exponents.

The following conclusion can be drawn
\begin{equation}\label{comparebigsmall}
A_p(\tau)^{(1-\theta)p}\leq\tau^{(m+1)\big(1-n(1-\frac{2}{p+1})\big)}A_p(\tau)^{p-1}.
\end{equation}
In fact, note that the definition of  $A_p(\tau)$
and $-\frac{(m+1)(n-1)}{2}(1-\frac{2}{p+1})<0$ for $n\geq2$.  Then, for sufficiently small $\epsilon>0$, if  $r_p\leq\frac{1}{2}+\frac{\chi_\epsilon}{m+1}$, we have $\frac{m+1}{2}-(m+1)r_p+\chi_\epsilon\geq0$,  i.e.,  $A_p(t)$ is increasing with respect to $t$, so
\begin{align*}
&A_p(\tau)^{-\theta p}\leq1\leq\tau^{\frac{m+1}{2}-\chi_\epsilon}\Longrightarrow A_p(\tau)^{-\theta p}\leq\tau^{(m+1)(1-r_p)}A_p(\tau)^{-1}\Longrightarrow (\ref{comparebigsmall}).
\end{align*}
While if $r_p>\frac{1}{2}+\frac{\chi_\epsilon}{m+1}$, we get $A_p(\tau)\leq1$.
From (\ref{thetaeq}) and the condition $p\leq 1+\frac{2}{n-1}(n\geq2)$,  it follows that $f(p)\doteq\theta p=\frac{n-np}{(n-2)p-n-2}$ is increasing in $p$ and  satisfies  $f(1+\frac{2}{n-1})=1$, which implies that $\theta p\leq1$  holds for any $1<p\leq1+\frac{2}{n-1}$.  Therefore,
\begin{align*}
A_p(\tau)^{-\theta p}\leq A_p(\tau)^{-1}\leq\tau^{(m+1)(1-r_p)} A_p(\tau)^{-1}\implies(\ref{comparebigsmall}).
\end{align*}
Moreover, (\ref{comparebigsmall}) remains valid for $n=1$ by a similar analysis.

The goal is to guarantee the integrability of the terms on the right-hand sides of  (\ref{nablaunl2estimate})-(\ref{unl2estimate}).  By (\ref{comparebigsmall}),  it suffices to ensure the convergence of the improper integral
\begin{align}
\int_1^\infty\tau^{-\frac{m}{2}(p+1)-\frac{\mu}{2}(p-1)+(m+1)\big(1-n(1-\frac{2}{p+1})\big)} A_p(\tau)^{p-1}(1+\log\tau)^{\chi_\delta p}d\tau.\label{inte1111}
\end{align}
Note the factor $(1+\log\tau)^{\chi_\delta p}$ does not affect the integrability of (\ref{inte1111}). Thus, in the case of  $A_p(\tau)=\tau^{\frac{m+1}{2}-(m+1)n(1-\frac{2}{p+1})+\chi_\epsilon}$, for sufficiently small $\epsilon>0$,  the integrability of (\ref{inte1111}) is equivalent to
\begin{equation*}\label{intergrity1}
\begin{aligned}
&-\frac{m}{2}(p+1)-\frac{\mu}{2}(p-1)+(m+1)\big(1-n(1-\frac{2}{p+1})\big)+\frac{m+1}{2}(p-1)-(m+1)n(p-1)
(1-\frac{2}{p+1})<-1,
\end{aligned}
\end{equation*}
which is equal to $p>p_0\big((m+1)n,\mu\big)$.  In the case of $A_p(\tau)=\tau^{-\frac{(m+1)(n-1)}{2}(1-\frac{2}{p+1})}$, we need to guarantee
\begin{align*}
&-\frac{m}{2}(p+1)-\frac{\mu}{2}(p-1)+(m+1)\big(1-n(1-\frac{2}{p+1})\big)-\frac{(m+1)(n-1)}{2}(p-1)(1-\frac{2}{p+1})<-1,
\end{align*}
which is equal to $p>\widetilde{p}_1(n,m,\mu)$. Consequently, from (\ref{nablaunl2estimate})-(\ref{unl2estimate}), we get
\begin{align*}
&\Vert \nabla u^n(t,\cdot)\Vert_{L^2}\lesssim t^{-\frac{m+\mu}{2}}(1+\log t)^{\chi_\delta}\Vert u\Vert_{\mathcal{X}(T)}^p,&\ & \Vert\partial_tu^n(t,\cdot)\Vert_{L^2}\lesssim t^{\frac{m-\mu}{2}}(1+\log t)^{\chi_\delta}\Vert u\Vert_{\mathcal{X}(T)}^p\\
&\Vert u^n(t,\cdot)\Vert_{L^{p+1}}\lesssim t^{-\frac{m+\mu}{2}}A_p(t)(1+\log t)^{\chi_\delta}\Vert u\Vert_{\mathcal{X}(T)}^p,&\ &\Vert u^n(t,\cdot)\Vert_{L^2}\lesssim t^{\frac{m-\mu}{2}+1}(1+\log t)^{\chi_\delta}\Vert u\Vert_{\mathcal{X}(T)}^p.
\end{align*}
Then it follows from the definition (\ref{normmathcalx}) of the norm of $\mathcal{X}(T)$ that
\begin{equation}\label{unmathcalX}
\Vert u^n\Vert_{\mathcal{X}(T)}\lesssim \Vert u\Vert_{\mathcal{X}(T)}^p.
\end{equation}

By combining (\ref{ulmathcalX}) and (\ref{unmathcalX}), we get the validness of  (\ref{prokey1}).

\end{proof}

\begin{proof}[The proof of (\ref{prokey2})]
By H$\ddot{\text{o}}$lder's inequality, Gagliardo-Nirenberg inequality and the definition (\ref{normmathcalx}) of the norm of $\mathcal{X}$, we get
\begin{align}
&\Vert\vert u(\tau,\cdot)\vert^{p}-\vert\tilde{u}(\tau,\cdot)\vert^p\Vert_{L^2}\lesssim
\Vert \vert u(\tau,\cdot)-\vert\tilde{u}(\tau,\cdot)\vert\big(\vert u(\tau,\cdot)\vert^{p-1}+\vert\tilde{u}(\tau,\cdot)\vert^{p-1}\big)\Vert_{L^{2}}\notag\\
&\leq \Vert u(\tau,\cdot)-\tilde{u}(\tau,\cdot)\Vert_{L^{2p}}\big(\Vert\vert u(\tau,\cdot)\vert^{p-1}\Vert_{L^{\frac{2p}{p-1}}}+\Vert\vert \tilde{u}(\tau,\cdot)\vert^{p-1}\Vert_{L^{\frac{2p}{p-1}}}\big)\notag\\
&=\Vert u(\tau,\cdot)-\tilde{u}(\tau,\cdot)\Vert_{L^{2p}}\big(\Vert u(\tau,\cdot)\Vert_{L^{2p}}^{p-1}+\Vert \tilde{u}(\tau,\cdot)\Vert_{L^{2p}}^{p-1}\big)\notag\\
&\lesssim \Vert \nabla\big(u(\tau,\cdot)-\tilde{u}(\tau,\cdot)\big)\Vert_{L^2}^\theta\Vert u(\tau,\cdot)-\tilde{u}(\tau,\cdot)\Vert_{L^{p+1}}^{1-\theta}\notag\\
&\quad\quad\quad\times\Big(\Vert\nabla u(\tau,\cdot)\Vert_{L^2}^{\theta(p-1)}\Vert u(\tau,\cdot)\Vert_{L^{p+1}}^{(1-\theta)(p-1)}+\Vert\nabla \tilde{u}(\tau,\cdot)\Vert_{L^2}^{\theta(p-1)}\Vert \tilde{u}(\tau,\cdot)\Vert_{L^{p+1}}^{(1-\theta)(p-1)}\Big)\notag\\
&\lesssim \tau^{-\frac{m+\mu}{2}p}A_p(\tau)^{(1-\theta)p}(1+\log\tau)^{\chi_\delta p}\Vert u-\tilde{u}\Vert_{\mathcal{X}(T)}\big(\Vert u\Vert_{\mathcal{X}(T)}^{p-1}+\Vert \tilde{u}\Vert_{\mathcal{X}(T)}^{p-1}\big)\label{gninequality221}
\end{align}
and
\begin{align}
&\Vert\vert u(\tau,\cdot)\vert^{p}-\vert\tilde{u}(\tau,\cdot)\vert^p\Vert_{L^{1+\frac{1}{p}}}\lesssim \Vert \vert u(\tau,\cdot)-\vert\tilde{u}(\tau,\cdot)\vert\big(\vert u(\tau,\cdot)\vert^{p-1}+\vert\tilde{u}(\tau,\cdot)\vert^{p-1}\big)\Vert_{L^{1+\frac{1}{p}}}\notag\\
&\lesssim \Vert u(\tau,\cdot)-\tilde{u}(\tau,\cdot)\Vert_{L^{p+1}}\big(\Vert\vert u(\tau,\cdot)\vert^{p-1}\Vert_{L^{\frac{p+1}{p-1}}}+\Vert\vert \tilde{u}(\tau,\cdot)\vert^{p-1}\Vert_{L^{\frac{p+1}{p-1}}}\big)\notag\\
&=\Vert u(\tau,\cdot)-\tilde{u}(\tau,\cdot)\Vert_{L^{p+1}}\big(\Vert u(\tau,\cdot)\Vert_{L^{p+1}}^{p-1}+\Vert \tilde{u}(\tau,\cdot)\Vert_{L^{p+1}}^{p-1}\big)\notag\\
&\lesssim \tau^{-\frac{m+\mu}{2}p}A_p(\tau)^p(1+\log\tau)^{\chi_\delta p}\Vert u-\tilde{u}\Vert_{\mathcal{X}(T)}\big(\Vert u\Vert_{\mathcal{X}(T)}^{p-1}+\Vert \tilde{u}\Vert_{\mathcal{X}(T)}^{p-1}\big),\label{gninequality222}
\end{align}
where $\theta$ is defined by (\ref{thetaeq}).

Then by Corollary \ref{energyspacethm2},
\begin{align}
&\Vert\nabla\big(\mathcal{N}u(t,\cdot)-\mathcal{N}\tilde{u}(t,\cdot)\big)\Vert_{L^2}\lesssim t^{-\frac{m+\mu}{2}}\int_1^t\tau^{\frac{\mu-m}{2}}(1+\log\frac{t}{\tau})^{\chi_\delta}\Vert\vert u(\tau,\cdot)\vert^{p}-\vert\tilde{u}(\tau,\cdot)\vert^p\Vert_{L^2}d\tau\notag\\
&\lesssim t^{-\frac{m+\mu}{2}}(1+\log t)^{\chi_\delta}\Vert u-\tilde{u}\Vert_{\mathcal{X}(T)}\big(\Vert u\Vert_{\mathcal{X}(T)}^{p-1}+\Vert \tilde{u}\Vert_{\mathcal{X}(T)}^{p-1}\big)\int_1^t\tau^{-\frac{m}{2}(p+1)-
\frac{\mu}{2}(p-1)}A_p(\tau)^{(1-\theta)p}(1+\log\tau)^{\chi_\delta p}d\tau.\label{key2key11}
\end{align}
Similarly,
\begin{align}
&\Vert\partial_t\big(\mathcal{N}u(t,\cdot)-\mathcal{N}\tilde{u}(t,\cdot)\big)\Vert_{L^2}\lesssim t^{\frac{m-\mu}{2}}(1+\log t)^{\chi_\delta}\Vert u-\tilde{u}\Vert_{\mathcal{X}(T)}\big(\Vert u\Vert_{\mathcal{X}(T)}^{p-1}+\Vert \tilde{u}\Vert_{\mathcal{X}(T)}^{p-1}\big)\notag\\
&\quad\quad\quad\quad\quad\times\int_1^t\tau^{-\frac{m}{2}(p+1)-
\frac{\mu}{2}(p-1)}A_p(\tau)^{(1-\theta)p}(1+\log\tau)^{\chi_\delta p}d\tau\label{key2key12},\\
&\Vert\mathcal{N}u(t,\cdot)-\mathcal{N}\tilde{u}(t,\cdot)\Vert_{L^{p+1}}\lesssim t^{-\frac{m+\mu}{2}}A_p(t)(1+\log t)^{\chi_\delta}\Vert u-\tilde{u}\Vert_{\mathcal{X}(T)}\big(\Vert u\Vert_{\mathcal{X}(T)}^{p-1}+\Vert \tilde{u}\Vert_{\mathcal{X}(T)}^{p-1}\big)\notag\\
&\quad\quad\times\int_1^t\tau^{-\frac{m}{2}(p+1)-\frac{\mu}{2}(p-1)+(m+1)(1-r_p)}
A_p(\tau)^{p-1}(1+\log\tau)^{\chi_\delta p}d\tau.\label{key2key13}
\end{align}
It is again observed that $\tau^{m+1}A_1(\tau)^{-1}\leq t^{\frac{m+1}{2}-\chi_\epsilon},$ we estimate
\begin{align}
&\Vert\mathcal{N}u(t,\cdot)-\mathcal{N}\tilde{u}(t,\cdot)\Vert_{L^2}
\lesssim t^{\frac{m-\mu}{2}+1}(1+\log t)^{\chi_\delta}\Vert u-\tilde{u}\Vert_{\mathcal{X}(T)}\big(\Vert u\Vert_{\mathcal{X}(T)}^{p-1}+\Vert \tilde{u}\Vert_{\mathcal{X}(T)}^{p-1}\big)\notag\\
&\quad\quad\quad\quad\quad\times\int_1^t\tau^{-\frac{m}{2}(p+1)-\frac{\mu}{2}(p-1)}A_p(\tau)^{(1-\theta)p}(1+\log \tau)^{\chi_\delta p}d\tau.\label{key2key14}
\end{align}
Note that the integrals appearing in (\ref{key2key11})-(\ref{key2key14}) coincide with those in  (\ref{nablaunl2estimate})-(\ref{unl2estimate}). Hence, under the condition $p>\max\big\{p_0\big((m+1)n,\mu\big),\widetilde{p}_1(n,m,\mu)\big\}$, we have
\begin{equation*}
\begin{aligned}
&\Vert\nabla\big(\mathcal{N}u(t,\cdot)-\mathcal{N}\tilde{u}(t,\cdot)\big)\Vert_{L^2}\lesssim t^{-\frac{m+\mu}{2}}(1+\log t)^{\chi_\delta}\Vert u-\tilde{u}\Vert_{\mathcal{X}(T)}\big(\Vert u\Vert_{\mathcal{X}(T)}^{p-1}+\Vert \tilde{u}\Vert_{\mathcal{X}(T)}^{p-1}\big),\\
&\Vert\partial_t\big(\mathcal{N}u(t,\cdot)-\mathcal{N}\tilde{u}(t,\cdot)\big)\Vert_{L^2}\lesssim t^{\frac{m-\mu}{2}}(1+\log t)^{\chi_\delta}\Vert u-\tilde{u}\Vert_{\mathcal{X}(T)}\big(\Vert u\Vert_{\mathcal{X}(T)}^{p-1}+\Vert \tilde{u}\Vert_{\mathcal{X}(T)}^{p-1}\big),\\
&\Vert\mathcal{N}u(t,\cdot)-\mathcal{N}\tilde{u}(t,\cdot)\Vert_{L^{p+1}}\lesssim t^{-\frac{m+\mu}{2}}A_p(t)(1+\log t)^{\chi_\delta}\Vert u-\tilde{u}\Vert_{\mathcal{X}(T)}\big(\Vert u\Vert_{\mathcal{X}(T)}^{p-1}+\Vert \tilde{u}\Vert_{\mathcal{X}(T)}^{p-1}\big),\\
&\Vert\mathcal{N}u(t,\cdot)-\mathcal{N}\tilde{u}(t,\cdot)\Vert_{L^2}\lesssim t^{\frac{m-\mu}{2}+1}(1+\log t)^{\chi_\delta}\Vert u-\tilde{u}\Vert_{\mathcal{X}(T)}\big(\Vert u\Vert_{\mathcal{X}(T)}^{p-1}+\Vert \tilde{u}\Vert_{\mathcal{X}(T)}^{p-1}\big),
\end{aligned}
\end{equation*}
which verify the validity of (\ref{prokey2}).
\end{proof}

\section{Final remarks and open problem}\label{finialremarks}

To determine whether the index $\max\{p_0\big((m+1)n,\mu\big),\widetilde{p}_1(n,m,\mu)\}$ established in this paper is the critical exponent for (\ref{eq}) when $\delta<0$, the blow-up behavior of solutions must also be investigated.

In the iterative argument employed by Palmieri \cite{Pa2025} to study the blow-up problem for (\ref{eq}), the iteration framework is constructed using the two real roots of the quadratic equation $r^2-(\mu-1)r+\nu^2=0$, which requires the discriminant to satisfy $\delta\doteq(\mu-1)^2-4\nu^2\geq0$.

Moreover, the test function method employed in our study \cite{LiGuo2025ii} of the critical curve for the weakly coupled system remains applicable to the single equation. Multiplying both sides of (\ref{eq}) by $h(t)$, we have
\begin{equation}\label{sandueq}
\partial_t^2(hu)-\Delta(t^{2m}hu)+\partial_t\big(\frac{\mu}{t}hu-2h'u\big)
+u\big(h''-\frac{\mu}{t}h'+\frac{\mu+\nu^2}{t^2}h\big)=h\vert u\vert^p.
\end{equation}
To continue constructing the test function in a separated form with respect to the time and spatial variables, we need to write  (\ref{sandueq}) in divergence form.  This requires solving the ordinary differential equation
\begin{equation}\label{ode}
h''(t)-\frac{\mu}{t}h'(t)+\frac{\mu+\nu^2}{t^2}h(t)=0,
\end{equation}
which is a typical Euler-type ordinary differential equation. Setting $h(t)=t^\lambda$, (\ref{ode}) reduces to $\big(\lambda^2-(\mu+1)\lambda+\mu+\nu^2\big)t^{\lambda-2}=0$. Notice that the discriminant of the quadratic equation $\lambda^2-(\mu+1)\lambda+\mu+\nu^2=0$ is still $\delta$.    In \cite{LiGuo2025ii}, we considered a solution $h(t)=t^{\frac{\mu+1-\sqrt{\delta}}{2}}$ to (\ref{ode}) for $\delta\geq0$ as  a component of the test function.  Now, for $\delta<0$, the  two linearly independent solutions to (\ref{ode}) are
\begin{equation*}
t^{\frac{\mu+1}{2}}\cos\left(\frac{\sqrt{-\delta}}{2}\log t\right)
\quad\text{and}\quad
t^{\frac{\mu+1}{2}}\sin\left(\frac{\sqrt{-\delta}}{2}\log t\right).
\end{equation*}
Unfortunately, regardless of which one is chosen as a component of the test function, the argument in \cite{LiGuo2025ii}  cannot be carried out further.

Based on the above analysis, regardless of whether one considers the iterative method employed by Palmieri or the test function method used in our work, although the quadratic equations arising from the two approaches are different, their discriminants are indeed the same.  This suggests that, in order to address the blow-up problem for the mass-dominated semilinear EPDT equation (\ref{eq}), it is necessary to improve upon both of the above methods.

\paragraph{Acknowledgments}
F. Guo is supported by the Priority Academic Program Development of Jiangsu Higher Education Institutions.

\section*{Appendix}
\addcontentsline{toc}{section}{Appendix}
%\addcontentsline{toc}{section}{Appendix}
\setcounter{theorem}{0}
\renewcommand{\thetheorem}{A.\arabic{theorem}}
\setcounter{lemma}{0}
\renewcommand{\thelemma}{A.\arabic{lemma}}
\setcounter{equation}{0}
\renewcommand{\theequation}{A.\arabic{equation}}
\renewcommand{\theremark}{A.\arabic{equation}}
The proof framework and methods for the local solutions, namely, Proposition \ref{propositionlocalexistence},  are comparable
to those for the global solutions, so we give a brief outline.  The analysis is carried out in the same space as
\begin{align*}
X(T)=C\big([1,T];H^1\cap L^{p+1}\big)\cap C^1\big([1,T]; L^2\big)
\end{align*}
with the norm $\Vert u\Vert_{X(T)}\doteq\max\limits_{t\in[1,T]}M[u](t)$, where
\begin{equation}\label{defintionmut}
M[u](t)=\Vert\nabla u(t,\cdot)\Vert_{L^2}+\Vert\partial_tu(t,\cdot)\Vert_{L^2}+\Vert u(t,\cdot)\Vert_{L^{p+1}}+\Vert u(t,\cdot)\Vert_{L^2}.
\end{equation}
For $T>1$,  introduce $X(T,M)=\{u\in X(T):\Vert u\Vert_{X(T)}\leq M\}$ and the operator $\mathcal{N}$ as in (\ref{mathcalmudef}), where $M>0$ is determined later. Following the proof of Proposition \ref{keypropositionmass},  we can prove
\begin{lemma}
For any $u,\tilde{u}\in X(T,M)$, the following estimates hold
\begin{align}
&\Vert\mathcal{N}u\Vert_{X(T)}\leq C_{1T}\Vert(u_0,u_1)\Vert_{\mathcal{A}_p}+C_{2T}\Vert u\Vert_{X(T)}^p,\label{appendixkey1}\\
&\Vert\mathcal{N}u-\mathcal{N}\tilde{u}\Vert_{X(T)}\leq C_{3T}\Vert u-\tilde{u}\Vert_{X(T)}\big(\Vert u\Vert_{X(T)}^{p-1}+\Vert\tilde{u}\Vert_{X(T)}^{p-1}\big),\label{appendixkey2}
\end{align}
where $C_{1T}$ is bounded and $C_{2T},C_{3T}\rightarrow0$ as $T\rightarrow1^{+}$.
\end{lemma}
\begin{proof}[Proof of Proposition \ref{propositionlocalexistence}]
by (\ref{appendixkey1}) and (\ref{appendixkey2}), one may first fix a sufficiently large $M$ and then take $T$ sufficiently close to $1$, thereby guaranteeing that the operator $\mathcal{N}$  maps  $X(T,M)$ into itself as a contraction map. According to the contraction mapping principle, there exists a $u \in X(T,M)$ satisfying $\mathcal{N}u = u$, i.e., $u$ is the local solution to (\ref{eq}).
\end{proof}


\vskip 0.2cm







\begin{thebibliography}{100}

\bibitem{BeLo1976} { {J. Bergh, J. L\"ofstr\"om}, Interpolation Spaces:
An Introduction}, \textit{Speringer-Verlag Berlin Heidelberg New York}, (1976).

\bibitem{BoEr1953}{{H. Bateman,  A. Erdelyi}, Higher transcendental functions, Vol.1}. McGraw-Hill, New York, (1953).

\bibitem{BoPHD2011}{{C. B$\ddot{\text{o}}$hme,} Decay rates and scattering states for wave models with time-dependent potential}, PhD thesis, TU Bergakademie Feriberg, Gemany,  (2011).

\bibitem{BoRe2012}{{C. B$\ddot{\text{o}}$hme,  M. Reissig,} A scale-invariant Klein-Gordon model with time-dependent potential}, \textit{Ann. Univ. Ferrara Sez. VII Sci. Mat.}, \textbf{58} (2012), 229-250.

\bibitem{Bren1975}{{P. Brenner}, On $L^p-L^{p'}$ estimates for the wave-equation}, \textit{Math. Z.}, \textbf{145} (1975), 251-254.

\bibitem{Da2021} {{M. D$^\prime$Abbicco}, Small data solutions for the Euler-Poisson-Darboux equation with a power nonlinearity}, \textit{J. Differential Equations},  \textbf{286} (2021), 531-556.

\bibitem{DaLuRe2015} {M. D$^\prime$Abbicco, S. Lucente, M. Reissig,} A shift in the critical exponent for semilinear wave equations with a not effective damping, \textit{J. Differential Equations}, \textbf{259} (2015), 5040-5073.

\bibitem{DaLuRe2013} {M. D$^\prime$Abbicco, S. Lucente, M. Reissig,} Semi-linear wave equations with effective damping, \textit{Chin. Ann. Math. Ser.B}, \textbf{259} (2015), 5040-5073.

\bibitem{DaPa2021} {M. D$^\prime$Abbicco, A. Palmieri,} A Note on $L^p$-$L^q$ estimates for semilinear critical dissipative
Klein-Gordon Equations, \textit{J. Dynam. Differential Equations}, \textbf{33} (2021), 63-74.

\bibitem{Fujita1969} {H. Fujita,}On the blowing up of solutions of the Cauchy problem for $u_t=\Delta u+u^{1+\alpha}$, \textit{J. Fac. Sci. Univ.}, \textbf{13} (1969), 109-124.

\bibitem{GeLiSo1997} {V. Georgiev,  H.  Lindblad, C. Sogge,} Weighted Strichartz estimates and global existence for semilinear wave equations, \textit{Amer. J. Math.}, \textbf{119} (1997), 1291-1319.

\bibitem{Glassey1981} {R.T. Glassey,} Existence in the large for $\Box u= F(u)$  in two space dimensions, \textit{Math. Z.}, \textbf{178}  (1981), 233-261.

\bibitem{HeLiYin2025} {D. He, Q. Li, H. Yin,} Global small data weak solutions of 2-D semilinear wave equations with scale-invariant damping, II, \textit{Discrete Contin. Dyn. Syst.}, \textbf{56},  315-336, 2026.

\bibitem{He2017II} {D. He, I. Witt, H. Yin,} On semilinear Tricomi equations with critical
exponents or in two space dimensions, \textit{J. Differential Equations}, \textbf{263} (2017), 8102-8137.

\bibitem{He2017I} {D. He, I. Witt, H. Yin,} On the global solution problem for semilinear generalized Tricomi equations, I, \textit{Calc. Var. Partial Differential Equations,} \textbf{56} (2017), 1-24.

\bibitem{He2020} {D. He, I. Witt, H. Yin,} On the global solution problem for semilinear
generalized Tricomi equations, II, \textit{Pacific J. Math.,} \textbf{314} (2021), 29-80.

\bibitem{IkSo2018} {M. Ikeda, M. Sobajima,} Life-span of solutions to semilinear wave equation with time-dependent
critical damping for specially localized initial data, \textit{Math. Ann.,}  \textbf{372} (2018), 1017-1040.

\bibitem{IkWa2020} {M. Ikeda, Y. Wakasugi,} Global well-posedness for the semilinear wave equation with
              time dependent damping in the overdamping case, \textit{Proc. Amer. Math. Soc.,}  \textbf{148} (2020), 157-172.

\bibitem{John1979} {F. John,} Blow-up of solutions of nonlinear wave equations in three space dimensions, \textit{Proc. Nat. Acad. Sci. U.S.A.,}  \textbf{28} (1979), 1559-1560.

\bibitem{KeTa1999} {M. Keel, T. Tao,} Small data blow-up for semilinear {K}lein-{G}ordon equations, \textit{Amer. J. Math.,}  \textbf{121} (1999), 629-669.

\bibitem{LaiTaWa2017} {N. Lai, H. Takamura, K. Wakasa,} Blow-up for semilinear wave equations with the scale
invariant damping and super-Fujita exponent, \textit{J. Differential Equations,}  \textbf{263} (2017), 5377-5394.

\bibitem{LaiPa2024} {N. Lai, A. Palmieri, H. Takamura,} A blow-up result for the semilinear Euler-Poisson-Darboux-Tricomi
equation with critical power nonlinearity, \textit{J. Math.  Anal. Appl.},  \textbf{553} (2026), Paper No. 129835, 33.

\bibitem{LaiTa2018} {N. Lai, H. Takamura},  Blow-up for semilinear damped wave equations with subcritical exponent
in the scattering case, \textit{Nonlinear Anal.}, \textbf{168} (2018), 222-237.

\bibitem{LiGuo2025ii} {Y. Li, F. Guo},  Critical curve for weakly coupled system of semilinear  Euler-Poisson-Darboux-Tricomi equations, preprint,  Arxiv: 2511.08084.

\bibitem{LiGuo2025} {Y. Li, F. Guo},  Global existence for the semilinear regular Euler-Poisson-Darboux-Tricomi equation, \textit{J. Differential Equations}, \textbf{553} (2026), 129835.

\bibitem{LiGuo20252} {Y. Li, F. Guo},  Lifespan estimate for the semilinear regular Euler-Poisson-Darboux-Tricomi equation, \textit{Appl. Anal.},   1-22, https://doi.org/10.1080/00036811.2026.2673374.


\bibitem{LiWangYin2025} {Q. Li, D. Wang, H. Yin},  Global small data weak solutions of 2-D semilinear wave equations with scale-invariant damping, I, preprint, 	arXiv:2503.18677.

\bibitem{LiYin2025} {Q. Li, H. Yin},  Global small data weak solutions of 2-D semilinear wave equations with scale-invariant damping, \textit{J. Differential Equations}, \textbf{455} (2026),
113972.

\bibitem{LiSo1996} {H. Lindblad, D. Sogge},  Restriction theorems and semilinear Klein-Gordon equations in $(1+3)$-dimensions, \textit{Duke Math. J.}, \textbf{85} (1996), 227-252.

\bibitem{LiuWang2018} {M. Liu, C. Wang},  Global existence for semilinear damped wave equations in relation with the Strauss conjecture, \textit{Discrete Contin. Dyn. Syst.}, \textbf{40} (2020), 709-724.

\bibitem{LinNiZhai2012} {J. Lin, K. Nishihara, J. Zhai},  Critical exponent for the semilinear wave equation with time-dependent damping, \textit{Discrete Contin. Dyn. Syst.}, \textbf{32} (2012), 4307-4320.

\bibitem{Mat1977} {A. Matsumura},  On the asymptotic behavior of solutions of semi-linear wave
equations, \textit{Publ. Res. Inst. Math. Sci.}, \textbf{12} (1976), 169-189.

\bibitem{PaRe2017} {W. N. Nascimento, A. Palmieri, M. Reissig},  Semi-linear wave models with power non-linearity and scale-invariant time-dependent mass and dissipation, \textit{Math. Nachr.}, \textbf{290} (2017), 1779-1805.

\bibitem{OlLo2010} {F.W.J. Olver, D.W. Lozier, R.F. Boisvert, C.W. Clark},  NIST Handbook of Mathematical Functions, \textit{U.S. Department of Commerce.}, Washington, (2010).

\bibitem{PaPhd2018} {A. Palmieri},  Global in time existence and blow-up results
     for a semilinear wave equation with
    scale-invariant damping and mass, PhD Thesis, TU Bergakademie Freiberg, 2018, 279 pp.

\bibitem{Pa2025} {A. Palmieri},  On the critical exponent for the semilinear Euler-Poisson-Darboux-Tricomi equation with power nonlinearity, \textit{J. Differential Equations}, \textbf{437} (2025), 113309.

\bibitem{PaRe2019} {A. Palmieri, M. Reissig},  A competition between Fujita and Strauss type exponents for blow-up of semi-linear wave equations with scale-invariant damping and mass, \textit{J. Differential Equations}, \textbf{266} (2019), 1176-1220.

\bibitem{PaRe2018} {A. Palmieri, M. Reissig}, Semi-linear wave models with power non-linearity and scale-invariant time-dependent mass and dissipation, II, \textit{Math. Nachr.}, \textbf{291} (2018), 1859-1892.

\bibitem{Str70a} {R. Strichartz}, Convolutions with kernels having singularities of a sphere.
Trans, \textit{Trans. Am. Math. Soc.},  \textbf{148} (1970), 461-471.

\bibitem{Str70b} {R. Strichartz}, A priori estimates for the wave equation and some applications,   \textit{J. Funct. Anal.}, \textbf{5} (1970), 218-235.

\bibitem{Tu2019} {A. Palmieri, Z. Tu}, Lifespan of semilinear wave equation with scale invariant dissipation and mass and sub-Strauss power nonlinearity,   \textit{J. Math. Anal. Appl.}, \textbf{470} (2019), 447-469.

\bibitem{Rac92} {R. Racke}, Lectures on nonlinear evolution equations,   \textit{Aspects of Mathematics,
E19. Friedr. Vieweg $\&$ Sohn, Braunschweig},  1992.

\bibitem{ReYa2000} {M. Reissig, K. Yagdjian}, $L^p-L^q$ decay estimates for the solutions of strictly hyperbolic equations
of second order with increasing in time coefficients,   \textit{Math. Nachr.}, \textbf{214} (2000), 71-104.

\bibitem{ToYo2001} {G. Todorova, B. Yordanov}, Critical exponent for a nonlinear wave equation with damping,    \textit{J. Differential Equations},  \textbf{174} (2001), 464-489.

\bibitem{TuLin2019} {Z. Tu, J. Lin}, Life-span of semilinear wave equations with scale-invariant damping:
critical Strauss exponent case,    \textit{Differential Integral Equations},  \textbf{32} (2019),  249-264.

\bibitem{WaYo2019} {K. Wakasa, B. Yordanov},  On the nonexistence of global solutions for critical semilinear wave
equations with damping in the scattering case,    \textit{Nonlinear Anal.},  \textbf{180} (2019), 67-74.

\bibitem{Wa2014} {Y. Wakasugi},  Critical exponent for the semilinear wave equation with scale invariant damping,    \textit{Fourier Analysis, Birkh$\ddot{\text{a}}$user, Cham}, (2014), 375-390.

\bibitem{Wirth2004} {J. Wirth},  Solution representations for a wave equation with weak dissipation,    \textit{Math. Methods
Appl. Sci.}, \textbf{27} (2004), 101-124.

\bibitem{Wirth2006} {J. Wirth},  Wave equations with time-dependent dissipation I. Non-effective dissipation,    \textit{J. Differential
Equations},   \textbf{222} (2006), 487-514.


\bibitem{Wirth2007} {J. Wirth},  Wave equations with time-dependent dissipation II. Effective dissipation,    \textit{J. Differential
Equations},   \textbf{232} (2007), 74-103.


\bibitem{Yagdjian2021} {K. Yagdjian, A. Galstian},  Fundamental solutions for the Dirac equation in curved spacetime and
generalized Euler-Poisson-Darboux equation,    \textit{J. Differential
Equations},   \textbf{300} (2021), 80-117.

\bibitem{Zh2006} {B. Yordanov,  Q. Zhang},  Finite time blow up for critical wave equations in high dimensions,    \textit{J. Funct. Anal.},  \textbf{231} (2006), 361-374.

\bibitem{Zh2001} {Q. Zhang}, A blow-up result for a nonlinear wave equation with damping: the critical case,    \textit{C. R. Acad. Sci. Paris S\'{e}r. I Math}., \textbf{333} (2001), 109-114.

\bibitem{Zhou2007} {Y. Zhou}, Blow up of solutions to semilinear wave equations with critical exponent in high dimensions,    \textit{Chin. Ann. Math.}, \textbf{28B} (2007), 205-212.


\bibitem{Zhou2014} {Y. Zhou, W. Han}, Life-span of solutions to critical semilinear wave equations,    \textit{Comm. Partial Differential Equations}, \textbf{39} (2014), 439-451.

\end{thebibliography}
\end{document}